\documentclass[11pt]{article}

\usepackage[english]{babel}
\usepackage[T1]{fontenc}
\usepackage{amsfonts}
\usepackage{amsmath}
\usepackage{amssymb}
\usepackage{amsthm}
\usepackage{stmaryrd}
\usepackage{dsfont}
\usepackage{graphicx}
\usepackage{hyperref}
\usepackage{titlesec}
\usepackage{mathabx}
\usepackage{enumitem,tocloft}
\usepackage[top=2.5cm, bottom=2.5cm, left=3cm, right=3cm]{geometry}
\usepackage{xcolor}
\usepackage{fourier-orns}
\usepackage[utf8]{inputenc}
\usepackage{bm}
\usepackage{framed}
\usepackage{supertabular}
\hypersetup{linktocpage,breaklinks=true}

\newcommand{\Z}{\mathbb{Z}}
\newcommand{\N}{\mathbb{N}}
\newcommand{\R}{\mathbb{R}}
\newcommand{\C}{\mathbb{C}}
\newcommand{\T}{\mathbb{T}}
\newcommand{\A}{\mathcal{A}}
\newcommand{\leb}{\mathrm{Leb}}
\newcommand{\ld}{\mathrm{L}}
\newcommand{\esp}{(X,\mu)}
\newcommand{\espalg}{(X,\A,\mu)}

\newcommand{\esplebalg}{([0,1],\mathcal{B}([0,1]),\leb)}
\newcommand{\aut}{\mathrm{Aut}(X,\mu)}
\newcommand{\autalg}{\mathrm{Aut}(X,\A,\mu)}
\newcommand{\auty}{\mathrm{Aut}(Y,\nu)}

\newcommand{\orb}{\mathrm{Orb}}
\newcommand{\hmu}{\mathrm{h}_{\mu}}
\newcommand{\Hmu}{\mathrm{H}_{\mu}}
\newcommand{\htop}{\mathrm{h}_{\mathrm{top}}}
\newcommand{\p}{\mathcal{P}}
\newcommand{\ptilde}{\tilde{\mathcal{P}}}
\newcommand{\xmax}{x^{+}}
\newcommand{\xmin}{x^{-}}
\newcommand{\xmaxn}{X^{+}_{n}}
\newcommand{\xminn}{X^{-}_{n}}
\newcommand{\xmaxinfty}{X^{+}_{\infty}}
\newcommand{\xmininfty}{X^{-}_{\infty}}
\newcommand{\nmax}{N^{+}}
\newcommand{\nmin}{N^{-}}
\newcommand{\spec}{\mathrm{Sp}}
\newcommand{\sn}{\zeta}
\newcommand{\Xmax}{X_{B,\mathrm{max}}}
\newcommand{\Xmin}{X_{B,\mathrm{min}}}
\newcommand{\xb}{X_{B}}
\newcommand{\rk}{\mathrm{rk}}
\newcommand{\finer}{\succcurlyeq}

\newtheorem{theorem}{Theorem}[section]
\newtheorem*{theorem*}{Theorem}
\newtheorem{corollary}[theorem]{Corollary}
\newtheorem*{corollary*}{Corollary}

\newtheorem{proposition}[theorem]{Proposition}
\newtheorem{lemma}[theorem]{Lemma}
\newtheorem*{lemma*}{Lemma}
\newtheorem{claim}{Claim}
\newtheorem*{claim*}{Claim}

\newenvironment{cproof}{\begin{proof}[Proof of the
		claim]}{\end{proof}}

\newtheorem{theoremletter}{Theorem}

\theoremstyle{definition}
\newtheorem{definition}[theorem]{Definition}
\newtheorem{remark}[theorem]{Remark}
\newtheorem{question}[theorem]{Question}
\newtheorem{example}[theorem]{Example}

\title{Odomutants and flexibility results for quantitative orbit equivalence}
\author{Corentin Correia}
\date{April 2, 2025}

\begin{document}
	
	\maketitle
	
	\begin{abstract}
		We introduce new systems that we call odomutants, built by distorting the orbits of an odometer. We use these transformations for flexibility results in quantitative orbit equivalence.\par
		It follows from the work of Kerr and Li that if the cocycles of an orbit equivalence are $\log$-integrable, the entropy is preserved. Although entropy is also an invariant of even Kakutani equivalence, we prove that this relation and $\ld^{<1/2}$ orbit equivalence are not the same, using a non-loosely Bernoulli system of Feldman which is an odomutant.\par
		We also show that Kerr and Li's result on preservation of entropy is optimal, namely we find odomutants of all positive entropies orbit equivalent to an odometer, with almost $\log$-integrable cocycles. We actually build a strong orbit equivalence between uniquely ergodic Cantor minimal homeomorphisms, so our result is a refinement of a famous theorem of Boyle and Handelman.\par
		We finally prove that Belinskaya's theorem is optimal for all the odometers, namely for every odometer, we find a odomutant which is almost-integrably orbit equivalent to it but not flip-conjugate. This yields an extension of a theorem by Carderi, Joseph, Le Maître and Tessera.
	\end{abstract}
	
	
	\smallskip
	
	{
		\small	
		\noindent\textbf{{MSC-classification:}}	
		Primary 37A35; Secondary 37A20.
	}
	
	\setcounter{tocdepth}{2}
	\tableofcontents
	
	\section{Introduction}
	
	Two ergodic probability measure-preserving bijections $S$ and $T$ on a standard atomless probability space $\espalg$, are \textit{orbit equivalent} if $S$ and some system $\Psi^{-1}T\Psi$ conjugate to $T$ have the same orbits up to measure zero. The isomorphism $\Psi$ is called an \textit{orbit equivalence} between $T$ and $S$.\par
	A stunning theorem of Dye~\cite{dyeGroupsMeasurePreserving1959} states that all ergodic measure-preserving bijections of a standard probability space are orbit equivalent. To get a more interesting theory, \textit{quantitative orbit equivalence} proposes to add quantitative restrictions on the \textit{cocycles} associated to orbit equivalence $\Psi$. These are integer-valued functions $c_S$ and $c_T$ defined by
	$$Sx=\Psi^{-1}T^{c_S(x)}\Psi(x)\text{ and }Tx=\Psi S^{c_T(x)}\Psi^{-1}(x),$$
	they are well-defined in the ergodic case. In this paper, we consider two quantitative forms of orbit equivalence: \textit{Shannon orbit equivalence} and $\varphi$\textit{-integrably orbit equivalence}, for maps $\varphi\colon\R_+\to\R_+$. Shannon orbit equivalence requires that there exists an orbit equivalence whose cocycles are Shannon, meaning that the partitions associated to $c_S$ and $c_T$ are both of finite entropy. For $\varphi$-integrable orbit equivalence, we ask that both integrals
	$$\int_X{\varphi(|c_S(x)|)\mathrm{d}\mu(x)}\text{ and }\int_X{\varphi(|c_T(x)|)\mathrm{d}\mu(x)}$$
	are finite.\par
	In this paper, when $\varphi(x)=x^p$, we are asking that both cocycles $c_S$ and $c_T$ are in $\ld^p$, and thus call it an $\ld^p$ orbit equivalence. Also when $c_S$ and $c_T$ are in $\ld^q$ for every $q<p$, we say that we have an $\ld^{<p}$ orbit equivalence. The notion of $\ld^p$ orbit equivalence can be traced back to the work of Bader, Furman and Sauer~\cite{baderIntegrableMeasureEquivalence2013} in the more general context of measure equivalence, while Shannon orbit equivalence was defined by Kerr and Li. Finally, $\varphi$-integrable orbit equivalence was first defined and studied by~\cite{delabieQuantitativeMeasureEquivalence2022}.\par
	The main goal is to understand which probability measure-preserving bijections are $\varphi$-integrably orbit equivalent or Shannon orbit equivalent. However the construction in the proof of Dye's theorem is not explicit and does not give any quantitative information on the cocycles. Then a more tractable question is the preservation of dynamical properties under these forms of quantitative orbit equivalence. In order to get flexibility results and then partially answer these questions, we introduce in this paper an explicit construction of orbit equivalence between odometers and systems with completely different properties, that we call \textit{odomutants}.\newline
	
	In recent years, odometers have been a central class of systems for explicit constructions, thanks to their combinatorial structure. For example, Kerr and Li \cite{kerrEntropyVirtualAbelianness2024} prove that every odometer is Shannon orbit equivalent to the universal odometer, providing concrete examples of Shannon orbit equivalent systems which are non conjugate. This result was generalized: we show in~\cite{correiaRankoneSystemsFlexible2024} that many rank-one systems (including the odometers and many irrational rotations) with various spectral and mixing properties are $\varphi$-integrably orbit equivalent to the universal odometer, with $\varphi\colon\R_+\to\R_+$ satisfying $\varphi(x)\underset{x\to +\infty}{=}o(x^{1/3})$. Finally, in order to show that the main result of~\cite[Theorem~1.1]{delabieQuantitativeMeasureEquivalence2022} is optimal in many examples, Delabie, Koivisto, Le Maître and Tessera provide concrete orbit equivalences between group actions\footnote{We do not give any definition in this setting, as the paper is only about probability measure-preserving bijections $S$, which can be seen as $\mathbb{Z}$-actions via $(n,x)\in \Z\times X\mapsto S^nx$.} built with F\o lner tilings (see~\cite[Section~6]{delabieQuantitativeMeasureEquivalence2022}). It turns out that we get a $\Z^k$-odometer in the case of the group $\Z^k$, thus highlighting how useful the combinatorial structures of such systems are.\par
	In our paper, the construction is also based on odometers, it is motivated by a construction by Feldman~\cite{feldmanNewKautomorphismsProblem1976}. The odomutants associated to the same odometer are explicitely built from successive distortions of its orbits, have the same point spectrum (Theorem~\ref{thspectrum}) but they can be completely different. They provide flexibility and optimality results: Theorems~\ref{thA},~\ref{corthB},~\ref{thB} and~\ref{thC} that we explain with more details in the following paragraphs.
	
	\paragraph{A theorem of preservation of entropy proved by Kerr and Li.}
	
	We may wonder whether Shannon or $\varphi$-integrable orbit equivalence are trivial or not. Kerr and Li proved that a well-known invariant of conjugacy, the measure-theoretic entropy, is an invariant of Shannon orbit equivalence.
	
	\begin{theorem*}[{\cite[Theorem~A]{kerrEntropyVirtualAbelianness2024}}]
		Entropy is preserved under Shannon orbit equivalence.
	\end{theorem*}
	
	A connection between $\varphi$-integrable orbit equivalence and Shannon orbit equivalence is given by the following statement which is a consequence of~\cite[Lemma~3.15]{carderiBelinskayaTheoremOptimal2023}.
	
	\begin{lemma*}
		Let $f\colon X\to\Z$ be a measurable map. If it is $\log$-integrable, then it is Shannon.
	\end{lemma*}
	
	As a consequence, $\varphi$-integrable orbit equivalence implies Shannon orbit equivalence when $\varphi$ is greater than $\log$ and, combined with Kerr and Li's theorem, we get the following result.
	
	\begin{theorem*}
		Let $\varphi\colon\R_+\to\R_+$ be a map satisfying $\log{t}\underset{t\to +\infty}{=}O(\varphi(t))$. Then entropy is preserved under $\varphi$-integrable orbit equivalence.
	\end{theorem*}
	
	\paragraph{On non-preservation of even Kakutani equivalence.}
	
	Entropy is also preserved under \textit{even Kakutani equivalence} (see Section~\ref{PrelKak}). We may wonder whether there is a connection between this equivalence relation and Shannon orbit equivalence or $\varphi$-integrable orbit equivalence for a map $\varphi\colon\R_+\to\R_+$ satisfying $\log{t}\underset{t\to +\infty}{=}O(\varphi(t))$. Note that these quantitative forms of orbit equivalence are not equivalence relations \textit{a priori}. In the result below, $\ld^{<1/2}$ orbit equivalence means that the cocycles are in $\ld^p$ for every $p<\frac{1}{2}$.
	
	\begin{theoremletter}[See Theorem~\ref{thAbis}]\label{thA}
		There exists an ergodic probability measure-preserving bijection $T$ which is $\ld^{<1/2}$ orbit equivalent (in particular Shannon orbit equivalent) to the dyadic odometer but not evenly Kakutani equivalent to it.
	\end{theoremletter}
	
	We actually prove that $\ld^{<1/2}$ orbit equivalence does not preserve loose Bernoullicity\footnote{Loosely Bernoulli systems form a class of ergodic systems, which is closed under Kakutani equivalence (see Section~\ref{PrelKak}).}, so it does not imply Kakutani equivalence (weaker than even Kakutani equivalence). In~\cite{feldmanNewKautomorphismsProblem1976}, Feldman builds a zero-entropy ergodic system which is not loosely Bernoulli. This system, denoted by $T$, is actually an odomutant built from the dyadic odometer $S$ (this is the first example of odomutant and the starting point of our work). We prove that $S$ and $T$ are $\ld^{<1/2}$ orbit equivalent and Theorem~\ref{thA} follows from the fact that every odometer is loosely Bernoulli.
	
	\begin{remark}
		In~\cite{feldmanNewKautomorphismsProblem1976}, Feldman did not consider the question of the point spectrum of his non loosely Bernoulli system. As a corollary of Theorem~\ref{thspectrum}, we get that it has the same point spectrum as the dyadic odometer.
	\end{remark}
	
	\begin{question}
		Does there exists a sublinear map $\varphi\colon\R_+\to\R_+$ such that $\varphi$-integrable orbit equivalence implies Kakutani equivalence or even Kakutani equivalence? Such a map would be at least $x\mapsto x^{1/2}$. We may also wonder whether loose Bernoullicity is preserved under $\varphi$-integrable orbit equivalence for some sublinear map $\varphi$. Note that the case of a linear map $\varphi$ is straightforward, as a consequence of Belinskaya's theorem.
	\end{question}
	
	\paragraph{Optimality result for the preservation of entropy.}
	
	As stated above, $\varphi$-integrable orbit equivalence preserves entropy when the map $\varphi$ satisfies $\log{t}\underset{t\to +\infty}{=}O(\varphi(t))$. Theorem~\ref{corthB} shows that this result is almost sharp.
	
	\begin{theoremletter}\label{corthB}
		Let $(X,\mu)$ be a standard atomless probability space, let $\alpha$ be either a positive real number or $+\infty$, and let $S\in\aut$ be an odometer whose associated supernatural number $\prod_{p\in\Pi}{p^{k_p}}$ satisfies the following property: there exists a prime number $p_{\star}$ such that $k_{p_{\star}}=+\infty$. Then there exists a probability measure-preserving transformation $T\in\aut$ such that
		\begin{enumerate}
			\item $\hmu(T)=\alpha$;
			\item there exists an orbit equivalence between $S$ and $T$, which is $\varphi_m$-integrable for all integers $m\geq 0$,
		\end{enumerate}
		where $\varphi_m$ denotes the map $t\to \frac{\log{t}}{\log^{(\circ m)}{t}}$ and $\log^{(\circ m)}$ the composition $\log\circ\ldots\circ\log$ ($m$ times).
	\end{theoremletter}
	
	The notion of supernatural number associated to an odometer is defined after Definition~\ref{supernatural}, it totally describes its conjugacy class. Examples of odometers $S$ to which this theorem applies are the dyadic odometer, more generally the $p$-odometer for every prime number $p$, or the universal odometer. In our proof, the transformation $T$ is an odomutant associated to $S$, we now explain how to build such a system.\par
	Theorem~\ref{corthB} is actually a corollary of Theorem~\ref{thB} which is stated in a topological framework. Indeed, to prove this corollary, the main idea was to use topological entropy instead, simpler than measure-theoretic entropy in this context, and connected to it via the variational principle. Moreover, for the topological entropy to be well-defined, we have to consider odomutants that can be extended as homeomorphisms on the Cantor set. We notice that we build a \textit{strong orbit equivalence}, namely an orbit equivalence between homeomorphisms on the Cantor set such that the equality of the orbits holds at every point of the space (and not up to measure zero), and whose associated cocycles each have at most one point of discontinuity.
	
	\begin{theoremletter}[See Theorem~\ref{thBbis}]\label{thB}
		Let $\alpha$ be either a positive real number or $+\infty$. Let $S$ be an odometer whose associated supernatural number $\prod_{p\in\Pi}{p^{k_p}}$ satisfies the following property: there exists a prime number $p_{\star}$ such that $k_{p_{\star}}=+\infty$. Then there exists a Cantor minimal homeomorphism $T$ such that
		\begin{enumerate}
			\item $\htop(T)=\alpha$;
			\item there exists a strong orbit equivalence between $S$ and $T$, which is $\varphi_m$-integrable for all integers $m\geq 0$,
		\end{enumerate}
		where $\varphi_m$ denotes the map $t\to \frac{\log{t}}{\log^{(\circ m)}{t}}$ and $\log^{(\circ m)}$ the composition $\log\circ\ldots\circ\log$ ($m$ times).
	\end{theoremletter}
	
	In order to create topological entropy, we build an odomutant $T$ from the odometer $S$ in such a way that the dynamics of $T$ describes more words $\{\p(T^i(x))_{0\leq i\leq n-1}\mid x\in X\}$ than $S$ does, for partitions $\p$ in clopen sets that we will define\footnote{$\p(y)$ denotes the atom of the partition $\p$ which contains $y\in X$.}. Note that this is more or less the strategy applied by Feldman for the construction of a non loosely Bernoulli system, since loose Bernoullicity property also deals with the words produced by a system. Then Theorem~\ref{corthB} follows from Theorem~\ref{thB} and the variational principle since such a transformation $T$ is necessarily uniquely ergodic (see Proposition~\ref{oeuniqueerg}).\newline
	
	For the study of strong orbit equivalence, \textit{Bratteli diagrams} have played a crucial role. Every properly ordered Bratteli diagram provides a Cantor minimal homeomorphism, called a Bratteli-Vershik system. Conversely, Herman, Putnam and Skau proved in~\cite{hermanOrderedBratteliDiagrams1992} that every Cantor minimal homeomorphism is topologically conjugate to a \textit{Bratteli-Vershik system}. Moreover, using this characterization, Giordano, Putnam and Skau completely classified the Cantor minimal homeomorphisms up to strong orbit equivalence, using the \textit{dimension group} which turns out to be a complete invariant (see~\cite{giordanoTopologicalOrbitEquivalence1995}). We refer the reader to Appendix~\ref{secbrat} for a brief overview.\par
	An earlier version of Theorem~\ref{corthB} (and more generally Theorem~\ref{thB}) stated that there exists an odomutant with positive entropy which is orbit equivalent to an odometer with almost $\log$-integrable cocycles. Thanks to a suggestion of Thierry Giordano, we noticed that odomutants have already appeared in~\cite{boyleEntropyOrbitEquivalence1994}. Indeed Boyle and Handelman stated a result similar to Theorem~\ref{thB}, without any quantitative information on the cocycles.
	
	\begin{theorem*}[{\cite[Theorem~2.8 and Section~3]{boyleEntropyOrbitEquivalence1994}}]
		Let $S$ be the dyadic odometer. If $\alpha$ is a positive real number or $\alpha=+\infty$, then there exists a Cantor minimal homeomorphism $T$ such that:
		\begin{enumerate}
			\item $\htop(T)=\alpha$;
			\item $S$ and $T$ are strongly orbit equivalent.
		\end{enumerate}
	\end{theorem*}

	Their proof exactly consists in building a Bratteli diagram of an odomutant associated to the dyadic odometer. We thus manage to give a similar statement but with quantitative information on the cocycles (Theorem~\ref{thB}). The case of the finite entropy is an improvement of our earlier proof, and the case of the infinite entropy is a translation of Boyle and Handelman's proof in our formalism.\par
	Another crucial point is that the orbit equivalence we build in our paper is explicit, whereas Boyle and Handelman use the dimension group and so establish the strong orbit equivalence in a more abstract way. The comparison between Boyle and Handelman's construction and our formalism will be detailed in Appendix~\ref{secbrat}.
	
	\paragraph{Optimality of Belinskaya's theorem.}
	
	Belinskaya's theorem~\cite{belinskayaPartitionsLebesgueSpace1969} states that if $S$ and $T$ are orbit equivalent and one of the two associated cocycles is integrable, then $S$ and $T$ are flip-conjugate, meaning that $S$ is conjugate to $T$ or $T^{-1}$. As a consequence, $\ld^1$ orbit equivalence is exactly flip-conjugacy. Since integrability exactly means $\varphi$-integrability for linear maps $\varphi$, it is interesting to study the sublinear case, as was done in~\cite{carderiBelinskayaTheoremOptimal2023}.
	
	\begin{theorem*}[{\cite[Theorem~1.3]{carderiBelinskayaTheoremOptimal2023}}]\label{cjlt}
		Let $\varphi\colon\R_+\to\R_+$ be a sublinear function\footnote{This means that $\lim\limits_{t\to +\infty}{\frac{\varphi(t)}{t}}=0$.}. Let $S$ be an ergodic probability measure-preserving transformation and assume that $S^n$ is ergodic
		for some $n\geq 2$. Then there is another ergodic probability measure-preserving transformation
		$T$ such that $S$ and $T$ are $\varphi$-integrably orbit equivalent but not flip-conjugate.
	\end{theorem*}
	
	The authors asked whether this holds for a system $S$ such that $S^n$ is non-ergodic for all $n\geq 2$. The following statement provides an answer for the odometers which satisfy this property.
	
	\begin{theoremletter}[See Theorem~\ref{thCbis}]\label{thC}
		Let $\varphi\colon\R_+\to\R_+$ be a sublinear map and $S$ an odometer. There exists a probability measure-preserving transformation $T$ such that $S$ and $T$ are $\varphi$-integrably orbit equivalent but not flip-conjugate.
	\end{theoremletter}
	
	As in the proofs of Theorems~\ref{thA} and~\ref{thB}, the counter-example $T$ for Theorem~\ref{thC} is again an odomutant associated to $S$. To ensure that $S$ and $T$ are not flip-conjugate, we notice that an odometer is a factor of its associated odomutants, and we use the property of \textit{coalescence} for the odometers, which states that an extension of an odometer is conjugate to it if and only if every factor map associated to this extension is an isomorphism.
	
	\begin{remark}
		Note that a probability measure-preserving transformation $S$ such that $S^n$ is non-ergodic for every $n\geq 2$ factors onto some odometer. It would be interesting to combine the proof of Theorem~\ref{thC} with this remark so as to completely remove the assumption that $S^n$ is ergodic for some $n\geq 2$ in~\cite[Theorem~1.3]{carderiBelinskayaTheoremOptimal2023}.
	\end{remark}
	
	\paragraph{Outline of the paper.} After a few preliminaries in Section \ref{secPrel}, we introduce the notion of odomutants in Section \ref{secOdo}, we study its measure-theoretic and topological properties, and the orbit equivalence with their associated odometers. Theorems~\ref{thA},~\ref{thB} and~\ref{thC} are respectively proven in Sections~\ref{secA},~\ref{secB} and~\ref{secC}. Appendix~\ref{seccombi} deals with combinatorial results preparing for the proof of Theorem~\ref{thB}. In Appendix~\ref{secbrat}, we describe odomutants as Bratteli-Vershik systems and compare our proof of Theorem~\ref{thB} with the proof of Boyle and Handelman's theorem in~\cite{boyleEntropyOrbitEquivalence1994}. Finally Appendix~\ref{secappLB} is devoted to prove the well-known (but left unproved in the literature) equivalence between definitions of loose Bernoullicity in the zero-entropy case.
	
	\paragraph{Acknowledgements.}
	
	I thank my advisors François Le Maître and Romain Tessera for their support and valuable advice on writing this paper. I also thank Fabien Durand and Samuel Petite for fruitful discussion about Cantor minimal systems. Finally, I am very greatful to Thierry Giordano for enlightening conversations about Boyle and Handelman's works and more generally the notion of strong orbit equivalence.
	
	\section{Preliminaries}\label{secPrel}
	
	\subsection{Basic definitions in ergodic theory}\label{PrelDef}
	
	\paragraph{In a measure-theoretic framework.}
	The author may refer to~\cite{kerrErgodicTheoryIndependence2016} and~\cite{vianaFoundationsErgodicTheory2016} for complete surveys about the notions introduced in this section.\par
	The probability space $\espalg$ is assumed to be standard and atomless. Such a space is isomorphic to $\esplebalg$, i.e.~there exists a bimeasurable bijection $\Psi\colon X\to [0,1]$ (defined almost everywhere) such that $\Psi_{\star}\mu=\mathrm{Leb}$, where $\Psi_{\star}\mu$ is defined by $\Psi_{\star}\mu(A)=\mu(\Psi^{-1}(A))$ for every measurable set $A$. We consider maps $T\colon X\to X$ acting on this space and which are bijective, bimeasurable and \textbf{probability measure-preserving} (\textbf{p.m.p.}), meaning that $\mu(T^{-1}(A))=\mu(A)$ for all measurable sets $A\subset X$, and the set of these transformations is denoted by $\autalg$, or simply $\aut$, two such maps being identified if they coincide on a measurable set of full measure. In this paper, elements of $\aut$ are called \textbf{transformations} or (\textbf{dynamical}) \textbf{systems}.\par
	A measurable set $A\subset X$ is $T$\textbf{-invariant} if $\mu(T^{-1}(A)\Delta A)=0$, where $\Delta$ denotes the symmetric difference. The system $T\in\aut$ is $(\mu$\textbf{-)ergodic}, or $\mu$ is $T$\textbf{-ergodic}, if every $T$-invariant set is of measure $0$ or $1$. If $T$ is ergodic, then $T$ is \textbf{aperiodic}, i.e.~$T^n(x)\not =x$ for almost every $x\in X$ and for every $n\in\Z\setminus\{0\}$, or equivalently the $T$\textbf{-orbit} of $x$, denoted by $\orb_T(x)\coloneq\{T^n(x)\mid n\in\Z\}$, is infinite for almost every $x\in X$. A transformation $T$ is \textbf{uniquely ergodic} on $X$ if it admits a unique $T$-invariant probability measure $\mu$. In this case, $\mu$ is $T$-ergodic since in full generality the extremal points of the convex set of $T$-invariant probability measures are exactly the ergodic ones.\par
	Denoting by $\ld^2\espalg$ the space of complex-valued and square-integrable functions defined on $X$, a complex number $\lambda$ is an \textbf{eigenvalue} of $T$ if there exists $f\in\ld^2\espalg\setminus\{0\}$ such that $f\circ T=\lambda f$ almost everywhere ($f$ is then called an \textbf{eigenfunction}). An eigenvalue $\lambda$ is automatically an element of the unit circle $\T\coloneq\{z\in\C\mid |z|=1\}$. The \textbf{point spectrum} of $T$, denoted by $\spec(T)$, is then the set of all its eigenvalues. Notice that $\lambda=1$ is always an eigenvalue since the constant functions are in its eigenspace. Moreover $T$ is ergodic if and only if the constant functions are the only eigenfunctions with eigenvalue one, in other words the eigenspace of $\lambda=1$ is the line of constant functions (we say that it is a simple eigenvalue). Finally, a system has \textbf{discrete spectrum} if the span of all its eigenfunctions is dense in $\ld^2\espalg$.\par
	All the properties that we have introduced are preserved under conjugacy. Two transformations $T\in\aut$ and $S\in\mathrm{Aut}(Y,\nu)$ are \textbf{conjugate} if there exists a bimeasurable bijection $\Psi\colon X\to Y$ such that $\Psi_{\star}\mu=\nu$ and $\Psi\circ T=S\circ\Psi$ almost everywhere. Some classes of transformations have been classified up to conjugacy, the two examples to keep in mind are the following. By Ornstein~\cite{ornsteinBernoulliShiftsSame1970}, entropy is a total invariant of conjugacy among Bernoulli shifts (entropy will be introduced in Section~\ref{PrelEntr}). Moreover Halmos and von Neumann~\cite{halmosOperatorMethodsClassical1942} prove that two ergodic systems with discrete spectrums are conjugate if and only if they have equal point spectrums. For example, the odometers (introduced in Section~\ref{PrelOdo}) have discrete spectrum and this theorem enables us to classify them up to conjugacy.\par
	Transformations $T$ and $S$ are said to be \textbf{flip-conjugate} if $T$ is conjugate to $S$ or to $S^{-1}$. Since the point spectrum forms a circle subgroup, the Halmos-von Neumann theorem actually states that the point spectrum is a total invariant of flip-conjugacy in the class of ergodic discrete spectrum systems. Therefore we are able to classify the odometers up to flip-conjugacy.\par
	We say that $S$ is a \textbf{factor} of $T$, or $T$ is an \textbf{extension} of $S$, if there exists a measurable map $\Psi\colon X\to Y$ which is onto and such that $\Psi_{\star}\nu=\mu$ and $S\circ\Psi=\Psi\circ T$ almost everywhere. The map $\Psi$ is called a \textbf{factor map} from $T$ to $S$.
	
	\paragraph{In a topological framework.}
	
	The notions that we have introduced are part of a measure-theoretic setting. On the topological side, a \textbf{topological (dynamical) system} is a continuous map $T\colon X\to X$ on a topological space $X$ (usually $X$ is assumed to be compact). Two topological systems $T$ and $S$, respectively on topological spaces $X$ and $Y$, are \textbf{topologically conjugate} if there exists a homeomorphism $\Psi\colon X\to Y$ such that $\Psi\circ T=S\circ\Psi$ on $X$. A topological system is \textbf{minimal} if every orbit is dense. In this paper, we will only consider \textbf{Cantor minimal homeomorphisms}, namely minimal invertible topological systems on the Cantor set.\par
	In this paper, "systems", "conjugacy", "entropy" will always refer to the measure-theoretic setting. For the topological setting, we will always specify "topological system", "topological conjugacy", "topological entropy".
	
	\subsection{Measurable partitions}\label{Prelpart}
	
	A set $\mathcal{P}$ of measurable subsets of $X$ is a \textbf{measurable partition} of $X$ if:
	\begin{itemize}
		\item for every $P_1,P_2\in\p$, we have $\mu(P_1\cap P_2)=0$;
		\item the union $\bigcup_{P\in\p}{P}$ has full measure.
	\end{itemize}
	The elements of $\p$ are called the \textbf{atoms}. If $\p$ and $\mathcal{Q}$ are measurable partitions of $\esp$, we say that $\p$ \textbf{refines} (or is a refinement of, or is finer than) $\mathcal{Q}$, denoted by $\p\finer\mathcal{Q}$, if every atom of $\mathcal{Q}$ is a union of atoms of $\p$ (up to a null set). More generally, their \textbf{joint partition} is
	$$\p\vee\mathcal{Q}\coloneq\{P\cap Q\mid P\in\p, Q\in\mathcal{Q}\},$$
	namely the coarsest partition which refines $\p$ and $\mathcal{Q}$.\par
	A measurable partition $\p$ defines almost everywhere a map $\p(.)\colon X\to\p$ where $\p(x)$ is the atom of $\p$ which contains $x$. Given a measurable map $T\colon X\to X$, $\p$ provides \textbf{coding maps}
	$$[\mathcal{P}]_{i,n}\colon x\in X\mapsto (\p(T^jx))_{i\leq j\leq n}\in\p^{\{i,\ldots,n\}}.$$
	In particular, $[\mathcal{P}]_n(x)\coloneq[\mathcal{P}]_{0,n-1}(x)$ is the $n$\textbf{-word} of $x$.\par
	Given atoms $P_i,P_{i+1},\ldots,P_n$ of $\p$, the equality $[\mathcal{P}]_{i,n}(x)=(P_i,\ldots,P_n)$ exactly means that $x$ is an element of $T^{-i}(P_i)\cap T^{-(i+1)}(P_{i+1})\cap\ldots\cap T^{-n}(P_n)$. Therefore the partition which gives the values of $[\mathcal{P}]_{i,n}$ is the following joint partition
	$$\p_i^{n}\coloneq\bigvee_{j=i}^{n}{T^{-j}(\p)}$$
	with $T^{-j}(\p)\coloneq\{T^{-j}(P)\mid P\in\p\}$, this is a division of the space given by the dynamics of $T$, over the timeline $\{i,\ldots,n\}$ and with respect to $\p$.
	
	\subsection{Measure-theoretic entropy, topological entropy}\label{PrelEntr}
	
	Here we present two notions of entropy. For more details, the reader may refer to~\cite{downarowiczEntropyDynamicalSystems2011} and~\cite{kerrErgodicTheoryIndependence2016}.
	
	\paragraph{Measure-theoretic entropy.}
	
	Entropy, or measure-theoretic entropy, or metric entropy, of a measurable transformation is an invariant of conjugacy. To define it, we first define the entropy of a partition, which then enables us to quantify how much a transformation complexifies the partitions.\par
	Let $T$ be a system on $\esp$, not necessarily invertible, and $\p$ a finite measurable partition of $X$. Let us define the \textbf{entropy of }$\p$ by
	$$\Hmu (\p)\coloneq-\sum_{P\in\p}{\mu(P)\log{\mu(P)}},$$
	where $\mu(P)\log{\mu(P)}=0$ if $P$ is a null set. This is a positive real number. The following quantity
	$$\hmu(T,\p)\coloneq\lim_{n\to +\infty}{\frac{\Hmu(\p_0^{n-1})}{n}}$$
	is well-defined, this is the \textbf{entropy of }$T$\textbf{ with respect to }$\p$, and it tells us how quickly the dynamics of $T$ is dividing the space $X$ with the partition $\p$. Finally, let us define the \textbf{entropy of }$T$ by
	$$\hmu(T)\coloneq\sup_{\p}{\hmu(T,\p)},$$
	where the supremum is over all the finite measurable partitions $\p$ of $X$. This quantity is non-negative and can be infinite.\par
	The following result, due to Kolmogorov and Sinaï, enables us to prove the well-known fact that the odometers have zero entropy (see Section~\ref{PrelOdo}).
	
	\begin{theorem}[{\cite[after Definition~4.1.1]{downarowiczEntropyDynamicalSystems2011}}]\label{propmesentr}
		Let $(\mathcal{P}_k)_{k\geq 0}$ be an increasing sequence of partitions which generates the $\sigma$-algebra of $X$ (up to restriction to full-measure sets). Then we have
		$$\hmu\left (T,\mathcal{P}_k\right )\underset{k\to +\infty}{\to}\hmu(T).$$
	\end{theorem}
	
	\paragraph{Topological entropy.}
	
	In the topological setting, topological entropy is an invariant of topological conjugacy and is defined with similar ideas.\par
	The topological space $X$ has to be compact. We define the \textbf{joint cover} of two open covers $\mathcal{U}$ and $\mathcal{V}$ by
	$$\mathcal{U}\vee\mathcal{V}\coloneq\{U\cap V\mid U\in\mathcal{U}, V\in\mathcal{V}\}.$$
	Let $T$ be a topological system on $X$ and $\mathcal{U}$ an open cover of $X$. Let us define
	$$\mathcal{U}_0^{n-1}\coloneq\bigvee_{i=0}^{n-1}{T^{-i}(\mathcal{U})},$$
	where $T^{-i}(\mathcal{U})\coloneq\{T^{-i}(U)\mid U\in\mathcal{U}\}$, and
	$$\mathcal{N}(\mathcal{U})\coloneq\min\{|\mathcal{U}'|\mid \mathcal{U}'\text{ is a subcover of }\mathcal{U}\},$$
	where $|\mathcal{U}'|$ denotes the cardinality of $\mathcal{U}'$. The quantity $\mathcal{N}(\mathcal{U})$ is finite since $X$ is compact.\par
	The \textbf{topological entropy of }$T$\textbf{ with respect to the open cover }$\mathcal{U}$ is the well-defined limit
	$$\htop(T,\mathcal{U})\coloneq\lim_{n\to +\infty}{\frac{\log{\mathcal{N}(\mathcal{U}_0^{n-1})}}{n}},$$
	it tells us how quickly the dynamics of $T$ is shrinking the open sets of $\mathcal{U}$.
	
	Finally, let us define the \textbf{topological entropy of }$T$ by
	$$\htop(T)\coloneq\sup_{\mathcal{U}}{\htop(T,\mathcal{U})},$$
	where the supremum is over all the open covers $\mathcal{U}$ of $X$. This quantity is non-negative and can be infinite.\par
	The following result will enables us to build an odomutant with a prescribed topological entropy (see Lemma~\ref{entr}). We say that a sequence $(\mathcal{U}_n)_{n\geq 0}$ of open covers generates the topology on $X$ if for every $\varepsilon>0$, there exists $N\geq 0$ such that for every $n\geq N$, the open sets of $\mathcal{U}_n$ has a diameter less than $\varepsilon$.
	
	\begin{theorem}[{\cite[Remark~6.1.7]{downarowiczEntropyDynamicalSystems2011}}]\label{proptopentr}
		Let $T$ be a topological system on $X$ and $(\mathcal{U}_n)_{n\geq 0}$ a generating sequence of open covers. Then we have
		$$\htop(T)=\lim_{n\to +\infty}{\htop(T,\mathcal{U}_n)}.$$
	\end{theorem}
	
	\begin{example}\label{clopen}
		The compact space $X$ that we consider in this paper is of the form
		$$X\coloneq\prod_{n\geq 0}{\{0,1,\ldots,q_n-1\}},$$
		with integers $q_n$ greater or equal to $2$. It admits open covers which are partitions in clopen sets. If $\mathcal{U}$ is such an open cover, then $\mathcal{U}_0^{n-1}$ denotes both joint of open covers and joint of partitions. We have $\mathcal{N}(\mathcal{U}_0^{n-1})=|\mathcal{U}_0^{n-1}\setminus\{\emptyset\}|$ and this is exactly the number of words of the form $[\mathcal{U}]_n(x)$, for $x\in X$, where $[\mathcal{U}]_n$ is the coding map associated to the partition $\mathcal{U}$ (see Section~\ref{Prelpart}). Therefore, in the proof of Theorem~\ref{thB}, a method to create topological entropy consists in building a system $T$ whose number of $n$-words (with respect to some partition in clopen sets) increases quickly enough as $n$ goes to $\infty$.\par
		More precisely, the open covers $\mathcal{U}$ that we will consider are
		$$\p(\ell)\coloneq\left\{[i_0,\ldots,i_{\ell-1}]_{\ell}\mid 0\leq i_0<q_0,\ldots, 0\leq i_{\ell-1}<q_{\ell-1}\right\},$$
		for $\ell\geq 1$, where $[i_0,\ldots,i_{\ell-1}]_{\ell}$ denotes the $\ell$-cylinder
		$$\{x=(x_n)_{n\geq 0}\mid x_0=i_0,\ldots,x_{\ell-1}=i_{\ell-1}\}.$$
		Note that $(\p(\ell))_{\ell\geq 1}$ is a generating sequence of open covers. In Definition~\ref{defodostack}, we will also consider other partitions $\ptilde(\ell)$, for some reasons explained in the paragraph following this definition.
	\end{example}
	
	\paragraph{The variational principle.} In Example~\ref{clopen}, we explain the method that we will apply in this paper to create topological entropy and then prove Theorem~\ref{thB}. However we also would like to create measure-theoretic entropy to prove Theorem~\ref{corthB}. The variational principle enables us to connect these notions.
	
	\begin{theorem}[{Variational principle~\cite[Theorem~6.8.1]{downarowiczEntropyDynamicalSystems2011}}]\label{variationalprinciple}
		Let $T\colon X\to X$ be a topological system on a metric compact set $X$. Then we have
		$$\htop(T)=\sup_{\mu}{\hmu(T)},$$
		where the supremum is over all the $T$-invariant Borel probability measures $\mu$ on $X$.
	\end{theorem}
	As a consequence, if $T$ is uniquely ergodic, then we have
	$$\htop(T)=\hmu(T),$$
	where $\mu$ denotes the only $T$-invariant Borel probability measure.
	
	\subsection{Even Kakutani equivalence, loose Bernoullicity}\label{PrelKak}
	
	The notions introduced in this section can be found in~\cite{feldmanNewKautomorphismsProblem1976} and~\cite{ornsteinEquivalenceMeasurePreserving1982}.\par
	Let $T\in\aut$. Given a measurable set $A$, the return time $r_{A}\colon A\to\N^*\cup\{\infty\}$ is defined by:
	$$\forall x\in A,\ r_{A}(x)\coloneq\inf{\{k\geq 1\mid T^kx\in A\}}.$$
	
	It follows from Poincaré recurrence theorem that, if $A$ has positive measure, then the set $\{k\in\N^*\mid T^kx\in A\}$ is infinite for almost every $x\in A$. In particular, $r_A(x)$ is finite for almost every $x\in A$.\par
	Then we can define a transformation $T_A$ on the set $\{x\in A\mid r_A(x)<\infty\}$, namely on $A$ up to a null set, called the \textbf{induced tranformation} on $A$:
	$$T_Ax\coloneq T^{r_A(x)}x.$$
	The map $T_A$ is an element of $\mathrm{Aut}(A,\mu_A)$, where $\mu_A\coloneq\mu(.)/\mu(A)$ is the conditional probability measure. Its entropy is given by Abramov's formula:
	$$\mathrm{h}_{\mu_A}(T_A)=\frac{\hmu(T)}{\mu(A)}.$$
	
	\begin{definition}
		Let $S\in\aut$, $T\in\auty$ be two ergodic transformations.
		\begin{enumerate}
			\item $T$ and $S$ are said to be \textbf{Kakutani equivalent} if $T_A$ and $S_B$ are isomorphic for some measurable sets $A\subset X$ and $B\subset Y$.
			\item Moreover they are \textbf{evenly Kakutani equivalent} if in addition two such measurable sets have the same measure: $\mu(A)=\nu(B)$.
		\end{enumerate}
	\end{definition}
	
	It is well-known that Kakutani equivalence and even Kakutani equivalence are equivalence relations. It follows from Abramov's formula that entropy is preserved under even Kakutani equivalence.\par
	Similarly to Ornstein's theory~\cite{ornsteinBernoulliShiftsSame1970} for the conjugacy problem, Ornstein, Rudolph and Weiss~\cite{ornsteinEquivalenceMeasurePreserving1982} found a class of systems, called loosely Bernoulli system, where Kakutani and even Kakutani equivalences are well understood. These systems were first introduced by Feldman~\cite{feldmanNewKautomorphismsProblem1976}.
	
	\begin{definition}[see~\cite{feldmanNewKautomorphismsProblem1976}]\label{defLB}\ 
		\begin{itemize}
			\item The $f$-metric between words of same length is defined by:
			$$f_n((a_i)_{1\leq i\leq n},(b_i)_{1\leq i\leq n})=1-\frac{k}{n}$$
			where $k$ is the greatest integer for which we can find equal subsequences $(a_{i_{\ell}})_{1\leq \ell\leq k}$ and $(b_{j_{\ell}})_{1\leq \ell\leq k}$, with $i_1<\ldots <i_k$ and $j_1<\ldots <j_k$.
			\item Let $T\in\aut$ and $\p$ be a partition of $X$. The couple $(T,\p)$, called a process, is \textbf{loosely Bernoulli} if for every $\varepsilon>0$, for every sufficiently large integer $N$ and for each $M>0$, there exists a collection $\mathcal{G}$ of "good" atoms in $\p_{-M}^{0}$ whose union has measure greater than or equal to $1-\varepsilon$, and so that for each pair $A,B$ of atoms in $\mathcal{G}$, the following holds: there is a probability measure $n_{A,B}$ on $\p^N\times\p^N$ satisfying
			\begin{enumerate}
				\item $n_{A,B}(\{w\}\times \p^N)=\mu_A(\{[\mathcal{P}]_{1,N}(.)=w\})$ for every $w\in\p^N$;
				\item $n_{A,B}(\p^N\times\{w'\})=\mu_B(\{[\mathcal{P}]_{1,N}(.)=w'\})$ for every $w'\in\p^N$;
				\item $n_{A,B}(\{(w,w')\in\p^N\times\p^N\mid f_N(w,w')>\varepsilon\})<\varepsilon$.
			\end{enumerate}
			\item $T$ is \textbf{loosely Bernoulli} if $(T,\mathcal{P})$ is loosely Bernoulli for all finite partitions $\mathcal{P}$ of $X$.
		\end{itemize}
	\end{definition}
	
	Loose Bernoullicity for a process $(T,\p)$ expresses the fact that, conditionally to two pasts $A$ and $B$, the laws for the future are close, meaning that there exists a good coupling between them, with close words for the $f$-metric.
	
	\begin{example}
		The Bernoulli shift on $\{1,\ldots,k\}^{\Z}$ is loosely Bernoulli with respect to the partition $\{[1]_1,\ldots,[k]_1\}$. Indeed, conditionally to every past, the law for the $N$-word is always the uniform distribution on $\{1,\ldots,k\}^N$, so it suffices to define $n_{A,B}$ as the uniform distribution on the diagonal of $\p^N\times\p^N$, with the notations of the previous definition. This system is more generally loosely Bernoulli since $\{[1]_1,\ldots,[k]_1\}$ is a generating partition\footnote{To prove that a system is loosely Bernoulli, it is enough to prove it with respect to a generating partition (see~\cite{ornsteinEquivalenceMeasurePreserving1982} and the equivalent notion of finitely fixed process).}.
	\end{example}
	
	We will also prove that odometers are loosely Bernoulli (see Proposition~\ref{odozeroentr} in the next section), using the following equivalent definition of loose Bernoullicity for zero-entropy systems.
	
	\begin{theorem}\label{thEquivalenceLB}
		Let $T\in\aut$ and $\p$ be a partition of $X$ and assume that $\hmu(T,\p)=0$. Then $(T,\p)$ is loosely Bernoulli if and only if for every $\varepsilon>0$ and for every sufficiently large integer $N$, there exists a collection $\mathcal{H}$ of "good" atoms in $\p_{1}^{N}$ whose union has measure greater than or equal to $1-\varepsilon$ and so that we have $f_N(w,w')\leq\varepsilon$ for every $w,w'\in [\p]_{1,N}(\mathcal{H})$.
	\end{theorem}
	
	This has been stated by Feldman~\cite[Remark in p.~22]{feldmanNewKautomorphismsProblem1976} and Ornstein, Rudolph and Weiss~\cite[after Definition~6.1]{ornsteinEquivalenceMeasurePreserving1982} for instance. However, to our knowledge, there is no justification of this statement in the literature. This is the reason why we provide a proof in Appendix~\ref{secappLB}.\newline
	
	The choice of the metric is very important. Indeed, with the $d$-metric:
	$$d_n((a_i)_{1\leq i\leq n},(b_i)_{1\leq i\leq n})=|\{1\leq i\leq n\mid a_i\not= b_i\}|,$$
	also called the Hamming distance, we get the notion of very weakly Bernoulli systems and this is exactly the class considered in Ornstein's theory for the conjugacy problem.\par
	As mentioned above, Kakutani equivalence and even Kakutani equivalence are well understood in the class of loosely Bernoulli systems.
	
	\begin{theorem}[{\cite[Theorems~5.1 and~5.2]{ornsteinEquivalenceMeasurePreserving1982}}]\label{ORWkakutaniEq}
		Let $S\in\aut$, $T\in\auty$ be two ergodic transformations.
		\begin{enumerate}
			\item If $S$ is loosely Bernoulli and is Kakutani equivalent to $T$, then $T$ is also loosely Bernoulli.
			\item If $S$ and $T$ are loosely Bernoulli, then they are evenly Kakutani equivalent if and only if they have the same entropy.
		\end{enumerate}
	\end{theorem}
	
	\subsection{Odometers}\label{PrelOdo}
	
	Given integers $q_0, q_1, q_2, \ldots$ greater than or equal to $2$, let us consider the Cantor space
	$$X\coloneq\prod_{n\geq 0}{\{0,1,\ldots ,q_n-1\}},$$
	endowed with the infinite product topology and the associated Borel $\sigma$-algebra. The \textbf{odometer} on $X$ is the adding machine $S\colon X\to X$, defined for every $x\in X$ by
	$$Sx=\left\{\begin{array}{ll}
		(\underbrace{0,\ldots,0}_{i\text{ times}},1+x_i,x_{i+1},\ldots)&\text{if }i\coloneq\min{\{j\geq 0\mid x_j\not=q_j-1\}}\text{ is finite}\\
		(0,0,0,\ldots)&\text{if }x=(q_0-1,q_1-1,q_2-1,\ldots)
	\end{array}\right..$$
	In other words, $S$ is the addition by $(1,0,0,\ldots)$ with carry over to the right.\par
	An odometer is more generally a system which is conjugate to $S$ for some  choice of integers $q_n$. In this paper, we only consider this concrete example with the adding machine and we refer to it as "the odometer on $\prod_{n\geq 0}{\{0,1,\ldots,q_n-1\}}$".\par
	Let us introduce the \textbf{cylinders of length }$k$, or $k$\textbf{-cylinders},
	$$[x_0,\ldots ,x_{k-1}]_{k}\coloneq\left \{(y_n)_{n\geq 0}\in\prod_{n\geq 0}{\{0,1,\ldots ,q_n-1\}}\ \Big |\ y_0=x_0,\ldots ,y_{k-1}=x_{k-1}\right \}.$$
	We can define a cylinder with a subset $I_j$ of $\{0,1,\ldots,q_j-1\}$ instead of $x_j$. For instance, $[x_0,I_1,x_2]_3$ denotes the set of sequences $(y_n)_{n\geq 0}$ satisfying $y_0=x_0$, $y_1\in I_1$ and $y_2=x_2$. We also use the symbol $\bullet$ when we do not want to fix the value at some coordinate. For instance, $[x_0,\bullet,x_2]_3$ denotes the set of sequences $(y_n)_{n\geq 0}$ satisfying $y_0=x_0$ and $y_2=x_2$. By convention, the $0$-cylinder is $X$. For any $n\geq 1$, we also set a partially defined map
	$$\sn_n\colon X\setminus [\bullet,\ldots,\bullet,q_{n-1}-1]_{n}\to X\setminus [\bullet,\ldots,\bullet,0]_{n}$$
	which is the addition by
	$$(\underbrace{0,\ldots ,0}_{n-1\text{ times}},1,0,0,\ldots)$$
	with carry over to the right, and which coincides with $S^{q_0\ldots q_{n-2}}$ on $X\setminus [\bullet,\ldots,\bullet,q_{n-1}-1]_n$. As illustrated in Figure~\ref{odometer}, the cylinders and the maps $\sn_n$ offer a very interesting combinatorial structure with successive nested towers $\mathcal{R}_1,\mathcal{R}_2,\ldots$.\footnote{This kind of construction that we see in Figure~\ref{odometer} is called a cutting-and-stacking construction.}
	
	\begin{figure}[ht]
		\centering
		\includegraphics[width=1\linewidth]{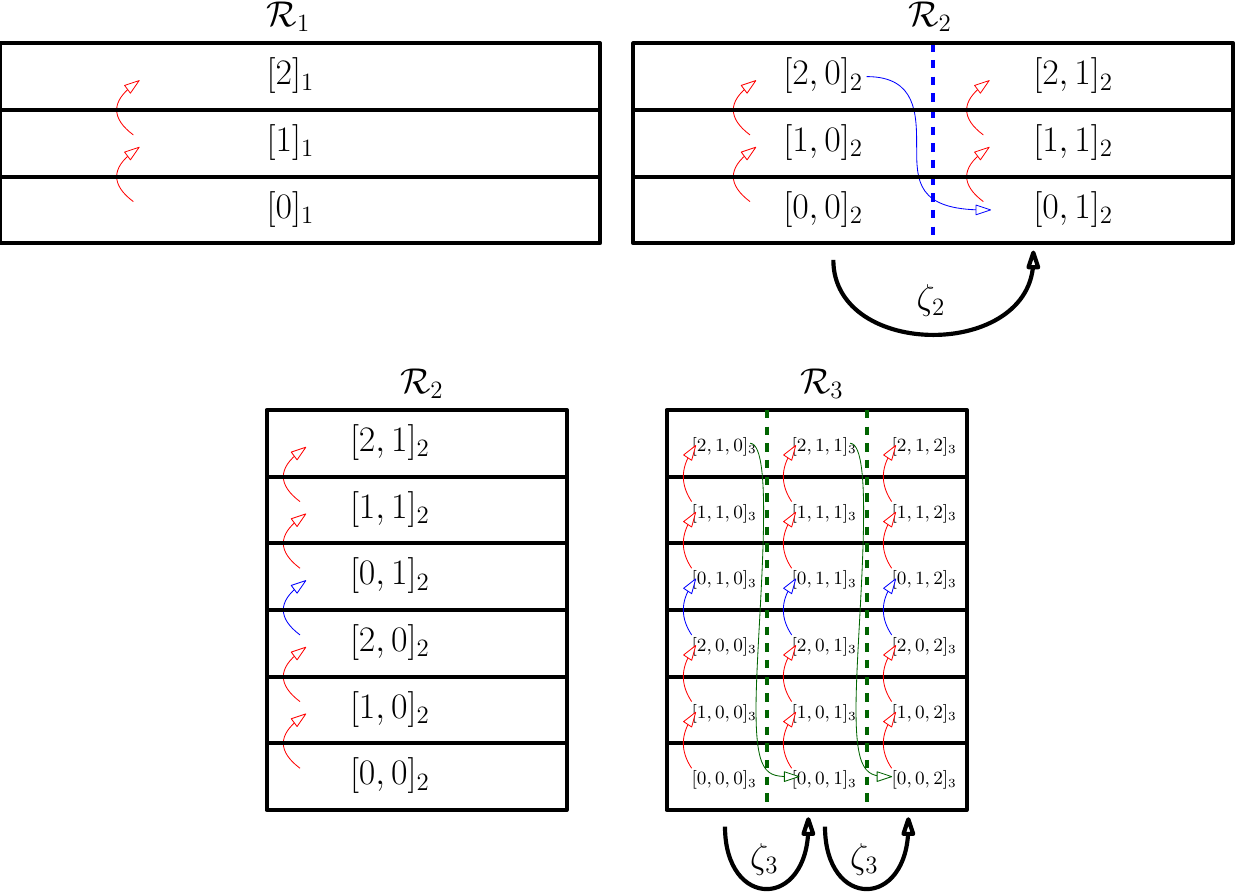}
		\caption{\footnotesize Example of odometer with $q_0=3$, $q_1=2$, $q_2=3$ (so $h_1=3$, $h_2=6$, $h_3=18$).
		}
		\label{odometer}
	\end{figure}
	
	From $(q_n)_{n\geq 0}$, a new sequence $(h_n)_{n\geq 1}$ is defined by
	$$\forall n\geq 1,\ h_n\coloneq q_0q_1\ldots q_{n-1}.$$
	The integer $h_n$ is the height of the tower $\mathcal{R}_{n}$ (see Figure~\ref{odometer}). By convention, we set $h_0\coloneq1$, the height of the tower $\mathcal{R}_0\coloneq(X)$ with a single level.\par
	As a topological system, $S$ is a Cantor minimal homeomorphism. As a measure-theoretic system, $S$ is uniquely ergodic and its only invariant measure is the product $\mu\coloneq\bigotimes_{n\geq 0}{\mu_n}$ where $\mu_n$ is the uniform distribution on $\{0,1,\ldots,q_n-1\}$. For the sake of completeness, we give a proof of the following well-known fact on odometers, which shows that the point spectrum is also fully understood.
	
	\begin{proposition}\label{ododiscr}
		Let $S$ be the odometer on $\prod_{n\geq 0}{\{0,1,\ldots,q_n-1\}}$. Its point spectrum is
		$$\spec(S)=\left\{\exp{\left (\frac{2i\pi k}{h_n}\right )}\mid n\geq 1, 0\leq k\leq h_n-1\right\}$$
		and for every $\lambda=\exp{\left (\frac{2i\pi k}{h_n}\right )}\in\spec(S)$, the map
		$$f_{\lambda}\colon x\in X\mapsto \sum_{j=0}^{h_n-1}{\lambda^j\mathds{1}_{S^j([0,\ldots ,0]_{n})}(x)}$$
		is an eigenfunction associated to $\lambda$. Moreover $S$ has discrete spectrum.
	\end{proposition}
	
	\begin{remark}
		The definition of $f_{\lambda}$ does not depend on the choice of $k$ and $n$ such that $\lambda=\exp{\left (\frac{2i\pi k}{h_n}\right )}$. Moreover, for $n=0$, we have $f_1=\mathds{1}_X$ (by convention, the $0$-cylinder is $X$).
	\end{remark}
	
	\begin{proof}[Proof of Proposition~\ref{ododiscr}]
		Let us set $\Lambda\coloneq\left\{\exp{\left (\frac{2i\pi k}{h_n}\right )}\mid n\geq 1, 0\leq k\leq h_n-1\right\}$. It is straightforward to check that $f_{\lambda}$ is an eigenfunction associated to $\lambda$, for every $\lambda\in\Lambda$. Let us show that the span of $\{f_{\lambda}\mid \lambda\in\Lambda\}$ is dense in $\ld^2\esp$. It will implies that $S$ has discrete spectrum and that $\Lambda=\spec(S)$.\par
		Let $n\geq 1$ and $\lambda=\exp{\left (\frac{2i\pi}{h_n}\right )}$.  Given $a_0,\ldots,a_{h_n-1}\in\C$, we have
		$$\sum_{\ell=0}^{h_n-1}{a_{\ell}f_{\lambda^{\ell}}}=\sum_{j=0}^{h_n-1}{P(\lambda^j)\mathds{1}_{S^j([0,\ldots ,0]_{n})}}$$
		with the polynomial $P=a_0+a_1Y+\ldots+a_{h_n-1}Y^{h_n-1}$. For every $j\in\{0,\ldots,h_n-1\}$, there exists a polynomial $P_j$ of degree less than $h_n$, satisfying $P_j(\lambda^{j})=1$ and $P(\lambda^{k})=0$ for all $k\in\{0,\ldots,h_n-1\}\setminus\{j\}$. This implies that the characteristic functions of cylinders are linear combinations of the eigenfunctions $f_{\lambda}$ for $\lambda\in\Lambda$, hence the result.
	\end{proof}
	
	Let us now explain the classification of odometers up to conjugacy (and even flip-conjugacy). Let $\Pi$ denote the set of prime numbers.
	
	\begin{definition}\label{supernatural}
		A \textbf{supernatural number} is a formal product of the form $\prod_{p\in\Pi}{p^{k_p}}$, with $k_p\in\N\cup\{+\infty\}$.
	\end{definition}
	
	Given a prime number $p\in\Pi$, denote by $\nu_p(k)$ the $p$-adic valuation of a positive integer $k$. To every odometer defined with integers $q_0,q_1,\ldots$, we associate a supernatural number $\prod_{p\in\Pi}{p^{k_p}}$ defined by
	$$k_p\coloneq\sum_{n\geq 0}{\nu_p(q_n)}.$$
	As a consequence of Proposition~\ref{ododiscr} and the Halmos-von Neumann theorem, the supernatural number $\prod_{p\in\Pi}{p^{k_p}}$ forms a total invariant of measure-theoretic conjugacy in the class of odometers. If $k_p=\infty$ for every prime number $p$, then the odometer is said to be \textbf{universal}. Given a prime number $p$, the $p$\textbf{-odometer} is the odometer such that $k_p=\infty$ and $k_q=0$ for every $q\in\Pi\setminus\{p\}$. In the case $p=2$, it is also called the \textbf{dyadic} odometer.\par
	Proposition~\ref{ododiscr} also implies that every odometer is coalescent.
	
	\begin{definition}\label{defcoalescent}
		A transformation $S\in\aut$ is coalescent if every system $T\in\aut$ which is isomorphic to $S$ satisfies the following: every factor map from $T$ to $S$ is an isomorphism.
	\end{definition}
	
	The fact that odometers are coalescent is proven in~\cite{hahnCharacteristicPropertiesDynamical1968} and~\cite{newtonCoalescenceSpectrumAutomorphisms1971}. In these articles, one proves that more general systems are coalescent and the phenomenon can be generalized in the context of group actions (see~\cite{ioanaWeakContainmentRigidity2016}). Here we give a short proof for ergodic systems with discrete spectrum.
	
	\begin{theorem}\label{thcoalescent}
		Every ergodic system with discrete spectrum is coalescent.
	\end{theorem}
	
	\begin{proof}[Proof of Theorem~\ref{thcoalescent}]
		Let $S\in\aut$ be an ergodic system with discrete spectrum, $T\in\aut$ isomorphic to $S$, and $\Psi\colon X\to X$ a factor map from $T$ to $S$. Given $\lambda\in\mathbb{T}$, let us denote by $E_S(\lambda)$ (resp.~$E_T(\lambda)$) the eigenspace of $S$ (resp.~$T$) associated to $\lambda$. First, ergodicity implies that non-zero eigenspaces have dimension $1$ (see Proper Value Theorem in~\cite[page 34]{halmosLecturesErgodicTheory1956}). Secondly, since $\Psi$ is a factor map, every eigenfunction $f$ of $S$ gives rise to the eigenfunction $f\circ\Psi$ of $T$, and more precisely $f\circ\Psi$ lies in $E_T(\lambda)$ if $f$ lies in $E_S(\lambda)$. Hence, since $S$ and $T$ are isomorphic, these two remarks imply that $E_T(\lambda)=\{f\circ\Psi\mid f\in E_S(\lambda)\}$ for every $\lambda$ in the point spectrum of $S$ (or equivalently the point spectrum of $T$). This implies
		$$\ld^2(X,\mu)=\{f\circ\Psi\mid f\in\ld^2(X,\mu)\}$$
		since they have discrete spectrum. Hence $\Psi$ is an isomorphism.
	\end{proof}
	
	For the proof of Theorem~\ref{thC}, the systems that we will consider will be an odometer $S$ and an associated odomutant $T$ (the odomutants are introduced in Section~\ref{subsecdefodo}). Since the odomutants are extensions of their associated odometer and since we explicitely know a factor map $\psi$ between them (see Proposition~\ref{aut}), Theorem~\ref{thcoalescent} will ensure that we will not build an orbit equivalence between flip-conjugate systems if $\psi$ is not invertible.\par
	Finally, odometers have the following properties.
	
	\begin{proposition}\label{odozeroentr}
		Odometers have zero measure-theoretic and topological entropies.
	\end{proposition}
	
	\begin{proposition}\label{odoLB}
		Odometers are loosely Bernoulli.\footnote{More generally, rank-one systems are loosely Bernoulli, this is proven by Ornstein, Rudolph and Weiss~\cite{ornsteinEquivalenceMeasurePreserving1982} (see Lemma 8.1) and we present their proof in the special case of odometers.}
	\end{proposition}
	
	\begin{remark}
		In the case of odometers, we can notice in the following proofs that zero entropy and loose Bernoullicity follow from a poor dynamics of these systems. Indeed, given concrete partitions (for instance the partitions $\p(k)$ given by the cylinders of length $k$, which increase to the $\sigma$-algebra), the dynamics of an odometer does not generate a lot of words and the different futures are close (in the sense of the definition of loose Bernoullicity). The idea behind the definition of odomutants will be to get systems with a less "laconic" dynamics.
	\end{remark}
	
	\begin{proof}[Proof of Proposition~\ref{odozeroentr}]
		Let $S$ be an odometer. The equality $\hmu(S)=\htop(S)$ follows from unique ergodicity and the variational principle (Theorem~\ref{variationalprinciple}). Let $\p(k)$ be the partition given by the cylinders of length $k$. The odometer $S$ acts as a cyclic permutation on the elements of $\p(k)$, so the sequence $((\p(k))_0^{n-1})_{n\geq 1}$ of partitions is stationary and we have $\hmu(S,\p(k))=0$. The sequence $(\p(k))_{k\geq 0}$ increases to the $\sigma$-algebra of $X$, so we have $\hmu(S,\p(k))\underset{k\to +\infty}{\to}\hmu(S)$ by Theorem~\ref{propmesentr}, and we get $\hmu(S)=0$.
	\end{proof}
	
	\begin{proof}[Proof of Proposition~\ref{odoLB}]
		Let $S$ be an odometer, associated to the integers $q_0,q_1,\ldots$, let $\p(k)$ be the partition given by the cylinders of length $k$. We prove that $(S,\p(k))$ is loosely Bernoulli for every $k\geq 1$, and we deduce from this that $(S,\p)$ is loosely Bernoulli for any finite partition $\p$. We use the caracterisation provided by Theorem~\ref{thEquivalenceLB}.\par
		Let us prove that $(S,\p(k))$ is loosely Bernoulli. Let $\varepsilon>0$, $N\geq 2h_k/\varepsilon$ and $\mathcal{H}=\p_1^N$. Let us denote by $W$ the word $\left (S^i([0,\ldots,0]_{k})\right )_{0\leq i\leq h_k-1}\in (\p(k))^{\{0,\ldots,h_k-1\}}$ of length $h_k$, this is the enumeration of the $k$-cylinders, with the order given by the dynamics of $S$. For every $x\in X$, the word $[\p(k)]_{1,N}(x)$ consists of the tail of the word $W$, followed by many concatenations of $W$, and the beginning of $W$. So any two words $w=[\p(k)]_{1,N}(x)$ and $w'=[\p(k)]_{1,N}(x')$ satisfy $f_N(w,w')\leq 2h_k/N\leq\varepsilon$. This proves that $(S,\p(k))$ is loosely Bernoulli.\par
		Now let $\p$ be a finite measurable partition and let us show that $(S,\p)$ is loosely Bernoulli. The sequence $(\p(k))_{k\geq 0}$ increases to the $\sigma$-algebra of $X$, so for a given $\varepsilon>0$, there exists $k\geq 0$ such that $\p$ and $\p(k)$ are close, meaning that there exists a $\p(k)$-measurable partition $\mathcal{Q}$, with $|\mathcal{Q}|=|\mathcal{P}|\eqcolon n$, and a good enumeration of the atoms of $\mathcal{Q}$ and $\mathcal{P}$ such that $\sum_{i=j}^{n}{\mu(P_j\Delta Q_j)}<\varepsilon$.
		Since $\p(k)$ refines $\mathcal{Q}$, words with respect to $\p(k)$ completely determine words with respect to $\mathcal{Q}$, so $(S,\mathcal{Q})$ is immediately loosely Bernoulli. Then, if $N$ is sufficiently large, there exists $\mathcal{H}\subset\mathcal{Q}_1^N$ covering at least $1-\varepsilon$ of the space and such that any two words $w,w'\in [\mathcal{Q}]_{1,N}(\mathcal{H})$ satisfy $f_N(w,w')\leq\varepsilon$ (the $f$-metric with respect to $\mathcal{Q}$).
		By the ergodic theorem, for every sufficiently large integer $N>0$, there exists a subset $X_0$ of $X$ such that $\mu(X_0)\geq 1-\varepsilon$ and every $x\in X_0$ satisfies
		$$\frac{1}{N}\left|\left\{i\in\{1,2,\ldots,N\}\mid S^ix\in\bigcup_{j=1}^{n}{(P_j\cap Q_j)}\right\}\right|\geq 1-2\varepsilon.$$
		This implies that for every $x\in X_0$, the word $[\mathcal{Q}]_{1,N}(x)$ determines at least a fraction $1-2\varepsilon$ of the word $[\p]_{1,N}(x)$. Therefore, given $x,x'\in X_0\cap\left (\bigcup_{C\in\mathcal{H}}{C}\right)$, the words $w=[\p]_{1,N}(x)$ and $w'=[\p]_{1,N}(x')$ satisfy $f_N(w,w')\leq 5\varepsilon$ (the $f$-metric with respect to $\p$). It remains to define $\mathcal{H}'\subset\p_1^N$ as the set of atoms with non trivial intersection with $X_0\cap\bigcup_{C\in\mathcal{H}}{C}$. It covers at least $1-3\varepsilon$ of the space and, with respect to $\p$, every two $N$-words $w$ and $w'$ produced in $\mathcal{H}'$ satisfy $f_N(w,w')\leq 5\varepsilon$, so we are done.
	\end{proof}
	
	\subsection{Orbit equivalence}\label{PrelOE}
	
	The conjugacy problem in full generality is very complicated (see~\cite{foremanConjugacyProblemErgodic2011}). We now give the formal definition of orbit equivalence, which is a weakening of the conjugacy problem.
	
	\begin{definition}\label{defoe}
		Two aperiodic transformations $S\in \aut$ and $T\in \auty$ are \textbf{orbit equivalent} if there exists a bimeasurable bijection $\Psi\colon X\to Y$ satisfying $\Psi_{\star}\mu=\nu$, such that $\mathrm{Orb}_S(x)=\mathrm{Orb}_{\Psi^{-1}T\Psi}(x)$ for almost every $x\in X$. The map $\Psi$ is called an \textbf{orbit equivalence} between $S$ and $T$.\par
		We can then define the \textbf{cocycles} associated to this orbit equivalence. These are measurable functions $c_S\colon X\to\mathbb{Z}$ and $c_T\colon Y\to\mathbb{Z}$ defined almost everywhere by
		$$Sx=\Psi^{-1}T^{c_S(x)}\Psi(x)\text{ and }Ty=\Psi S^{c_T(y)}\Psi^{-1}(y)$$
		($c_S(x)$ and $c_T(y)$ are uniquely defined by aperiodicity).
	\end{definition}
	
	\begin{remark}
		Conversely, the existence of a cocycle, let us say $c_T$, implies the inclusion of the $(\Psi^{-1}T\Psi)$-orbits in the $S$-orbits. So the existence of both cocycles $c_S$ and $c_T$ implies equality of orbits. This well-known characterization of orbit equivalence will be used in the proof of Theorem~\ref{oeq}.
	\end{remark}
	
	Given a map $\varphi\colon\R_+\to\R_+$, a measurable function $f\colon X\to\Z$ is said to be $\varphi$\textbf{-integrable} if
	$$\int_{X}{\varphi(|f(x)|)\mathrm{d}\mu}<+\infty.$$
	For example, integrability is exactly $\varphi$-integrability when $\varphi$ is non-zero and linear, and a weaker quantification on cocycles is the notion of $\varphi$-integrability for a sublinear map $\varphi$, meaning that $\lim_{t\to +\infty}{\varphi(t)/t}=0$. Two transformations in $\aut$ are said to be $\varphi$\textbf{-integrably orbit equivalent} if there exists an orbit equivalence between them whose associated cocycles are $\varphi$-integrable. The notion of $\ld^p$ \textbf{orbit equivalence} refers to the map $\varphi\colon x\to x^p$, and a $\ld^{<p}$ \textbf{orbit equivalence} is by definition an orbit equivalence which is $\ld^q$ for all $q<p$.\par
	Another form of quantitative orbit equivalence is Shannon orbit equivalence. We say that a measurable function $f\colon X\to\Z$ is \textbf{Shannon} if the associated partition $\{f^{-1}(n)\mid n\in\Z\}$ of $X$ has finite entropy, namely
	$$-\sum_{n\in\Z}{\mu(f^{-1}(n))\log{\mu(f^{-1}(n))}}<+\infty.$$
	Two transformations in $\aut$ are \textbf{Shannon orbit equivalent} if there exists an orbit equivalence between them whose associated cocycles are Shannon.\par
	Note that orbit equivalence preserves ergodicity. The next statement specifically connects orbit equivalence and unique ergodicity. Theorem~\ref{thB} and this proposition together with the variational principle directly imply Theorem~\ref{corthB}.
	
	\begin{proposition}\label{oeuniqueerg}
		Assume that two aperiodic measurable bijections $S$ and $T$ on a Borel space $X$ are orbit equivalent in the following stronger way: $S$ and $T$ are defined on the whole $X$ and the equality $\mathrm{Orb}_S(x)=\mathrm{Orb}_{T}(x)$ holds for every $x\in X$.\footnote{This is stronger than asking this property up to a null set.} Then $S$ is uniquely ergodic if and only if $T$ is uniquely ergodic. In this case, $S$ and $T$ have the same invariant probability measure.
	\end{proposition}
	
	\begin{proof}[Proof of Proposition~\ref{oeuniqueerg}]
		Assume that $S$ is uniquely ergodic and denote by $\mu$ its only invariant probability measure. The cocycle $c_S\colon X\to\Z$ is defined on the whole $X$ and is measurable. Let $\nu$ be a $T$-invariant probability measure. For every measurable set $A$, we have
		\begin{align*}
			\displaystyle\nu(S(A))&=\displaystyle\sum_{k\in\Z}{\nu(S(A\cap\{c_S=k\}))}\\
			&=\displaystyle\sum_{k\in\Z}{\nu(T^{k}(A\cap\{c_S=k\}))}\\
			&=\displaystyle\sum_{k\in\Z}{\nu(A\cap\{c_S=k\})}\\
			&=\displaystyle\nu(A),
		\end{align*}
		so $\nu$ is $S$-invariant and is equal to $\mu$. Therefore $T$ is uniquely ergodic and $\mu$ is its only invariant probability measure.
	\end{proof}
	
	For instance, strong orbit equivalence is a form of orbit equivalence, introduced in a topological framework by Giordano, Putnam and Skau~\cite{giordanoTopologicalOrbitEquivalence1995}, to which Proposition~\ref{oeuniqueerg} applies. The definition is the following.
	
	\begin{definition}
		Two Cantor minimal homeomorphisms $(X, S)$ and $(Y, T)$ are strongly orbit equivalent if there exists a homeomorphism $\Psi\colon X\to Y$ such that $S$ and $\Psi^{-1}T\Psi$ have the same orbits on $X$ and the associated cocycles each have at most one point of discontinuity.
	\end{definition}
	
	Boyle proved in his thesis~\cite{boyleTopologicalOrbitEquivalence1983} that strong orbit equivalence with continuous cocycles boils down to topological flip-conjugacy, namely $S$ is topologically conjugate to $T$ or to $T^{-1}$. As mentioned in the introduction, the classification of Cantor minimal homeomorphisms up to strong orbit equivalence is fully understood, with complete invariants such as the dimension group (see~\cite{giordanoTopologicalOrbitEquivalence1995}, and Appendix~\ref{secbrat} for a brief overview).
	
	\section{Odomutants}\label{secOdo}
	
	\subsection{Definitions}\label{subsecdefodo}
	
	Let $X\coloneq\prod_{n\geq 0}{\{0,1,\ldots, q_n-1\}}$ with integers $q_n\geq 2$, and let us recall the notation $h_n\coloneq q_0\ldots q_{n-1}$. The space $X$ is endowed with the infinite product topology and we denote by $\mu$ the product of the uniform distributions on each $\{0,1,\ldots,q_n-1\}$. We consider the odometer $S\colon X\to X$ on this space. Recall that it is defined by
	$$Sx=\left\{\begin{array}{ll}
		(\underbrace{0,\ldots,0}_{i\text{ times}},x_i+1,x_{i+1},\ldots)&\text{if }i\coloneq\min{\{j\geq 0\mid x_j\not=q_j-1\}}\text{ is finite}\\
		(0,0,0,\ldots)&\text{if }x=(q_0-1,q_1-1,q_2-1,\ldots)
	\end{array}\right.,$$
	and it is a $\mu$-preserving homeomorphism.\newline
	
	In this section, we introduce new systems that we call odomutants, defined from $S$ with successive distortions of its orbits, encoded by the following maps $\psi$ and $\psi_n$ (for $n\geq 0$).\par
	For every $n\geq 0$, we fix a finite sequence $\left (\sigma^{(n)}_i\right )_{0\leq i<q_{n+1}}$ of permutations of the set $\{0,1,\ldots,q_n-1\}$, and we introduce
	$$\psi_n\colon\left\{\begin{array}{lcl}
		X&\to &X\\
		x=(x_0,x_1,\ldots) &\mapsto & (\sigma^{(0)}_{x_1}(x_0),\sigma^{(1)}_{x_2}(x_1),\sigma^{(2)}_{x_3}(x_2),\ldots ,\sigma^{(n)}_{x_{n+1}}(x_n),x_{n+1},x_{n+2},\ldots)
	\end{array}\right..$$	
	It is not difficult to see that $\psi_n$ is a homeomorphism and preserves the measure $\mu$, its inverse is given by
	$$\psi_n^{-1}\colon \left\{\begin{array}{lcl}
		X &\to &X\\
		x=(x_0,x_1,\ldots)&\mapsto &(z_0(x),z_1(x),\ldots,z_{n}(x),x_{n+1},x_{n+2},\ldots)
	\end{array}\right.$$
	with $z_i(x)$ defined by backwards induction as follows:
	\begin{equation}
		\begin{aligned}\label{psiinv}
			&z_n(x)\coloneq\left (\sigma_{x_{n+1}}^{(n)}\right )^{-1}(x_n),\\
			&z_i(x)\coloneq\left (\sigma_{z_{i+1}(x)}^{(i)}\right )^{-1}(x_i)\text{ for every }i\in\{0,1,\ldots, n-1\}.
		\end{aligned}
	\end{equation}
	Let us also introduce
	$$\psi\colon\left\{\begin{array}{lcl}
		X&\to &X\\
		x=(x_0,x_1,\ldots) &\mapsto & \left (\sigma^{(n)}_{x_{n+1}}(x_n)\right )_{n\geq 0}
	\end{array}\right..$$
	The map $\psi$ is continuous but is not invertible in full generality. It is not difficult to see that $\psi_n(x)\underset{n\to +\infty}{\to}\psi(x)$ for every $x\in X$. The map $\psi$ also have the following properties.
	
	\begin{proposition}\label{preservonto}
		$\psi\colon X\to X$ preserves the probability measure $\mu$ and is onto.
	\end{proposition}
	
	\begin{proof}[Proof of Proposition~\ref{preservonto}]
		To prove that $\mu$ is $\psi$-invariant, it suffices to prove the equality $\mu(\psi^{-1}(A))=\mu(A)$ when $A$ is a cylinder. If $A$ is an $(n+1)$-cylinder, then $\psi^{-1}(A)=\psi_{n}^{-1}(A)$, so the $\psi$-invariance follows from the $\psi_n$-invariance for all $n\geq 0$.\par
		Given $y\in X$, let us find $x\in X$ such that $\psi(x)=y$. By definition, for every $n\geq 0$, $\psi(\psi_n^{-1}(y))$ is in the cylinder $[y_0,\ldots,y_n]_{n+1}$, so $\psi(\psi_n^{-1}(y))\underset{n\to +\infty}{\to}y$. By compactness, there exists a convergent subsequence of $\left (\psi_n^{-1}(y)\right )_{n\geq 0}$, of limit $x\in X$, and we have $\psi(x)=y$ since $\psi$ is continuous.
	\end{proof}
	
	The following computations motivate the definition of odomutants. Let us respectively set the \textbf{minimal} and \textbf{maximal} points of $X$:
	$$\xmin\coloneq(0,0,0,\ldots)\text{ and }\xmax\coloneq(q_0-1,q_1-1,q_2-1,\ldots).$$
	We define the following sets
	$$\xminn\coloneq\{x\in X\mid (x_0,\ldots,x_n)\not=(\xmin_0,\ldots,\xmin_n)\},$$
	$$\xmaxn\coloneq\{x\in X\mid (x_0,\ldots,x_n)\not=(\xmax_0,\ldots,\xmax_n)\},$$
	$$\xmininfty\coloneq X\setminus\{\xmin\}\text{ and }\xmaxinfty\coloneq X\setminus\{\xmax\}.$$
	It is not difficult to see that $\xmaxinfty$ is the increasing union of the sets $\xmaxn$, so for every $x\in\xmaxinfty$, we denote by $\nmax(x)$ the least integer $n\geq 0$ satisfying $x\in\xmaxn$. This also holds for $\xmininfty$ and $\xminn$, and $\nmin(x)$ is defined similarly.\par 
	Let $x\in\psi^{-1}(\xmaxinfty)$ and $N\coloneq\nmax(\psi(x))$. By definition of $N$, for every $n\geq N$, $S\psi_n(x)$ is equal to
	$$(\underbrace{\bm{0,\ldots,0}}_{N\text{ times }},\ \sigma^{N}_{x_{N+1}}(x_{N})\bm{+1},\ \sigma^{(N+1)}_{x_{N+2}}(x_{N+1}),\ \ldots,\ \sigma^{(n)}_{x_{n+1}}(x_{n}),\ x_{n+1},\ x_{n+2},\ \ldots).$$
	Using~\eqref{psiinv}, we get
	$$\psi_n^{-1}S\psi_n(x)=(y^{(n)}_0(x),\ldots,y^{(n)}_{n}(x),x_{n+1},x_{n+2},\ldots)$$
	with $y^{(n)}_i(x)$ defined by backwards induction as follows:
	\begin{align*}
		&y^{(n)}_n(x)\coloneq\left (\sigma_{x_{n+1}}^{(n)}\right )^{-1}(\sigma_{x_{n+1}}^{(n)}(x_n))=x_n,\\
		\forall\ n> i> N,\ &y^{(n)}_i(x)\coloneq\left (\sigma_{y^{(n)}_{i+1}(x)}^{(i)}\right )^{-1}(\sigma^{(i)}_{x_{i+1}}(x_{i})),\\
		&y^{(n)}_{N}(x)\coloneq\left (\sigma_{y^{(n)}_{N+1}(x)}^{(N)}\right )^{-1}(\sigma_{x_{N+1}}^{(N)}(x_{N})+1),\\
		\forall\ N>i\geq 0,\ &y^{(n)}_i(x)\coloneq\left (\sigma_{y^{(n)}_{i+1}(x)}^{(i)}\right )^{-1}(0).
	\end{align*}
	By induction, it is easy to get $(y^{(n)}_{N+1}(x),\ldots,y^{(n)}_n(x))=(x_{N+1},\ldots,x_n)$ and this implies the following simplification: $\psi_n^{-1}S\psi_n(x)$ is equal to $(y^{(n)}_0(x),\ldots,y^{(n)}_{N}(x),x_{N+1},x_{N+2},\ldots)$ with $y^{(n)}_i(x)$ inductively defined by
	\begin{align*}
		&y^{(n)}_{N}(x)\coloneq\left (\sigma_{x_{N+1}}^{(N)}\right )^{-1}(\sigma_{x_{N+1}}^{(N)}(x_{N})+1),\\
		\forall\ 0\leq i\leq N-1,\ &y^{(n)}_i(x)\coloneq\left (\sigma_{y^{(n)}_{i+1}(x)}^{(i)}\right )^{-1}(0).
	\end{align*}
	Finally, $(y^{(n)}_0(x),\ldots,y^{(n)}_{N}(x))$ does not depend on the integer $n\geq \nmax(\psi(x))$.
	
	\begin{definition}\label{defodomutant}
		For every $x\in\psi^{-1}(\xmaxinfty)$, let us define
		$$Tx\coloneq\psi_n^{-1}S\psi_n(x)$$
		for any $n\geq \nmax(\psi(x))$. The map $T$ is called the \textbf{odomutant} associated to the odometer $S$ and the sequences of permutations $\left (\sigma_i^{(n)}\right )_{0\leq i<q_{n+1}}$ for $n\geq 0$.
	\end{definition}
	
	\begin{figure}[!ht]
		\centering
		\includegraphics[width=0.8\linewidth]{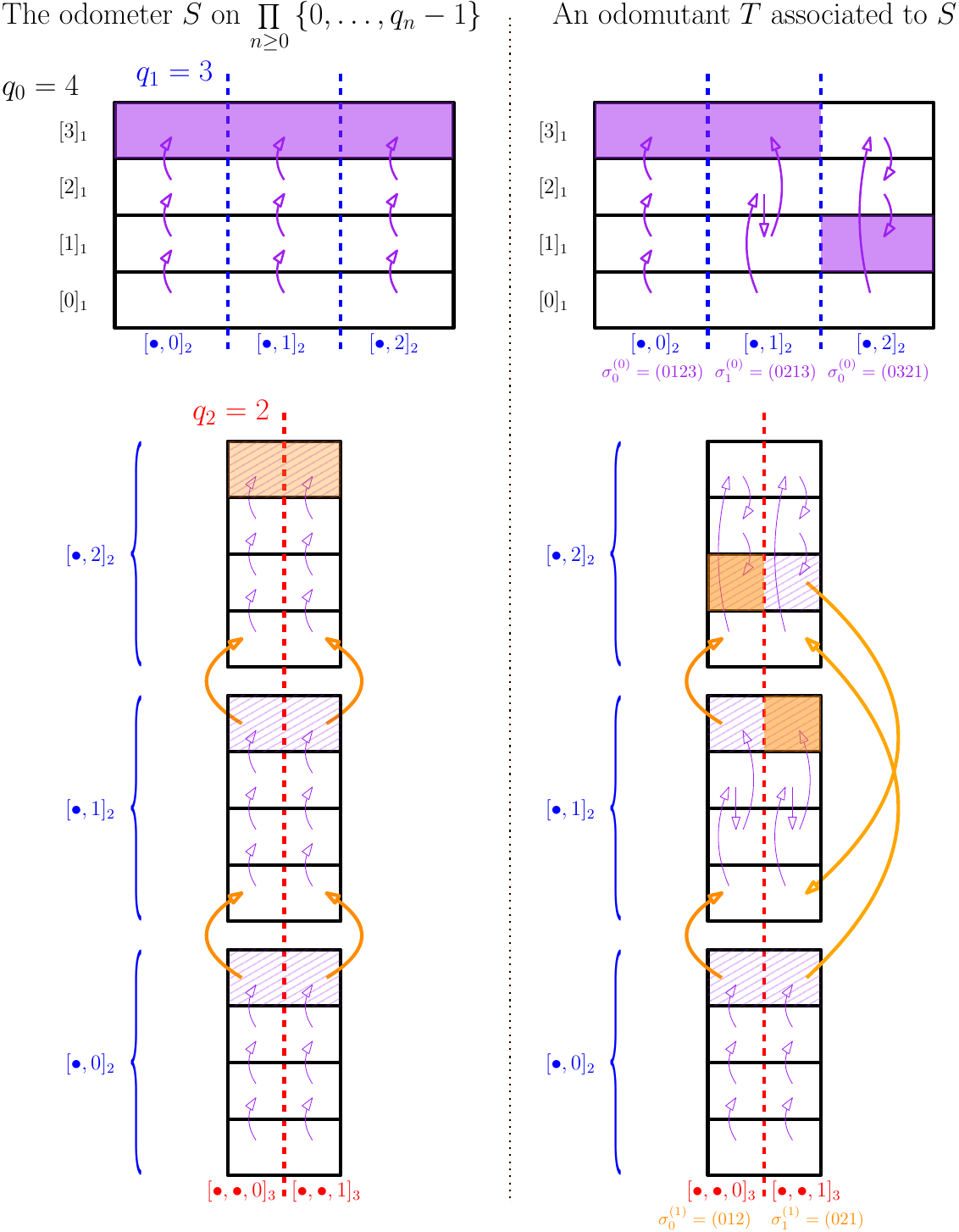}
		\caption{\footnotesize Example of the first two steps in the construction of an odometer (on the left) and an associated odomutant (on the right). For a permutation $\sigma$ of the set $\{0,\ldots,k-1\}$, the notation $\sigma=(i_0\ldots i_{k-1})$ means that $\sigma$ is defined by $\sigma(j)=i_j$ for every $j\in\{0,\ldots,k-1\}$. The area coloured in purple (resp. orange) is the subset on which $S$ and $T$ are not yet defined at the end of the first step (resp. second step), it is equal to $\{\nmax=1\}$ (resp. $\{\nmax=2\}$) for the odometer, $\{\nmax\circ\psi=1\}$ (resp. $\{\nmax\circ\psi=2\}$) for the odomutant.}
		\label{odomutant}
	\end{figure}
	
	As illustrated in Figure~\ref{odomutant}, an odomutant $T$ is a probability measure-preserving bijection that we build step by step. At step $n$, $T$ is well-defined on $\{\nmax(x)=n\}$. This is a cutting-and-stacking method very similar to the odometer, but at every step the way we connect the subcolumns of the tower depend on the next coordinates.
	
	\subsection{Odomutants with multiplicities}\label{odocutsta}
	
	At first view, when looking at Figure~\ref{odomutant}, we can think that an odomutant is encoded by a cutting-and-stacking construction where the new towers at each step are built by stacking only \textit{one} copy of the dynamics of each subcolumn. Actually, with some redondancies in the permutations of a same step, it is possible to encode a cutting-and-stacking construction where, at every step and for every subcolumn, many copies of its dynamics could appear in each new tower (as illustrated in Figure~\ref{odomutant2}). In this case, the partitions in cylinder of the same length are not the information we want to keep in mind, since they also remember that we divide the subcolumns to get many copies of its dynamics. This motivates the following definition that we explain with more details after.
	
	\begin{definition}\label{defodostack}
		Let $(q_n)_{n\geq 0}$ be a sequence of integers greater than or equal to $2$. Let $\bm{c}=\left (c_{n,0},\ldots,c_{n,\tilde{q}_n-1}\right )_{n\geq 1}$ be a sequence where $\tilde{q}_n$ and $c_{n,i}$ are positive integers satisfying $q_n=c_{n,1}+\ldots +c_{n,\tilde{q}_n}$, and $(\tau_j^{(n)})_{j\in\{0,\ldots,\tilde{q}_{n+1}-1\}}$ be a sequence of permutations of the set $\{0,\ldots,q_n-1\}$ for every $n\geq 0$. For every $n\geq 1$ and every $j\in\{0,\ldots,\tilde{q}_{n}-1\}$, we set
		$$I_j^{(n)}\coloneq\left (\sum_{i=0}^{j-1}{c_{n,i}}\right )+\left\{0,1,\ldots,c_{n,j}-1\right\}.\footnote{We write $s+\{0,1,\ldots,k\}\coloneq\{s,s+1,\ldots,s+k\}$. The family $(I^{(n)}_{0},\ldots,I^{(n)}_{\tilde{q}_{n}-1})$ forms a partition of $\{0,1,\ldots,q_{n}-1\}$.}$$
		Then we say that $T$ is the odomutant built with $\bm{c}$\textbf{-multiple permutations} $\tau^{(n)}_j$, if $T$ is the odomutant associated to the odometer on the space $\prod_{n\geq 0}{\{0,\ldots,q_n-1}\}$ and families of permutations $(\sigma_{i}^{(n)})_{0\leq i<q_{n+1}}$, where for every $n\geq 0$ and every $j\in\{0,\ldots,\tilde{q}_{n+1}\}$, we have $\sigma^{(n)}_i\coloneq\tau_j^{(n)}$ for all integers $i\in I^{(n+1)}_j$.\par
		In this case, we associate partitions $\ptilde(\ell)$ for every $\ell\geq 1$, defined by
		$$\ptilde(\ell)\coloneq\left\{[i_0,\ldots,i_{\ell-2},I_j^{(\ell-1)}]_{\ell}\mid 0\leq i_0<q_0,\ldots, 0\leq i_{\ell-2}<q_{\ell-2}, 0\leq j\leq \tilde{q}_{\ell-1}-1\right\}.$$
		
		We say that the odomutant is built with \textbf{uniformly} $\bm{c}$\textbf{-multiple permutations} if we have $c_{n,0}=\ldots =c_{n,\tilde{q}_n-1}=:c_n$ for every $n\geq 1$, and we simply write $\bm{c}\coloneq(c_n,\tilde{q}_n)_{n\geq 0}$.
	\end{definition}
	
	At the beginning of step $n$, for every $i\in\{0,\ldots,\tilde{q}_n-1\}$ there are $c_i$ subcolumns which have been defined with the same permutation\footnote{Note that the permutations $\tau^{(n-1)}_0,\ldots,\tau^{(n-1)}_{\tilde{q}_n-1}$ are not necessarily pairwise different.} $\tau^{(n-1)}_i$ at step $n-1$, they actually play the role of $c_i$ copies of the dynamics of a subcolumn that we would like to stack $c_i$ times in each tower. When considering the partition $\ptilde(n+1)$, we cannot distinguish between these "copies", as if it was the partition made up of the subcolumns that we would like to stack more than once in each tower.\par
	The odomutants built with uniformly multiple permutations, equipped with the associated partitions $(\ptilde(\ell))_{\ell\geq 1}$, better describe Boyle and Handelman's contructions~\cite{boyleEntropyOrbitEquivalence1994} than odomutants equipped with $\p(\ell)_{\ell\geq 1}$. We refer the reader to Appendix~\ref{secbrat} for more details, more precisely in Section~\ref{appendixComparison}. The sequences $(c_n)_{n}$ and $(\tilde{q}_n)_{n}$ respectively correspond to the sequences $(n_k)_{k}$ and $(m_k)_{k}$ introduced in their paper. Then, to prove Theorem~\ref{thB} in the case $\alpha=+\infty$, we will partly reformulate the proof of their similar statement with our formalism. Our proof in the case $\alpha<+\infty$ will be different than theirs since we will build an odomutant with pairwise different permutations at each step.
	
	\begin{figure}[!ht]
		\centering
		\includegraphics[width=0.74\linewidth]{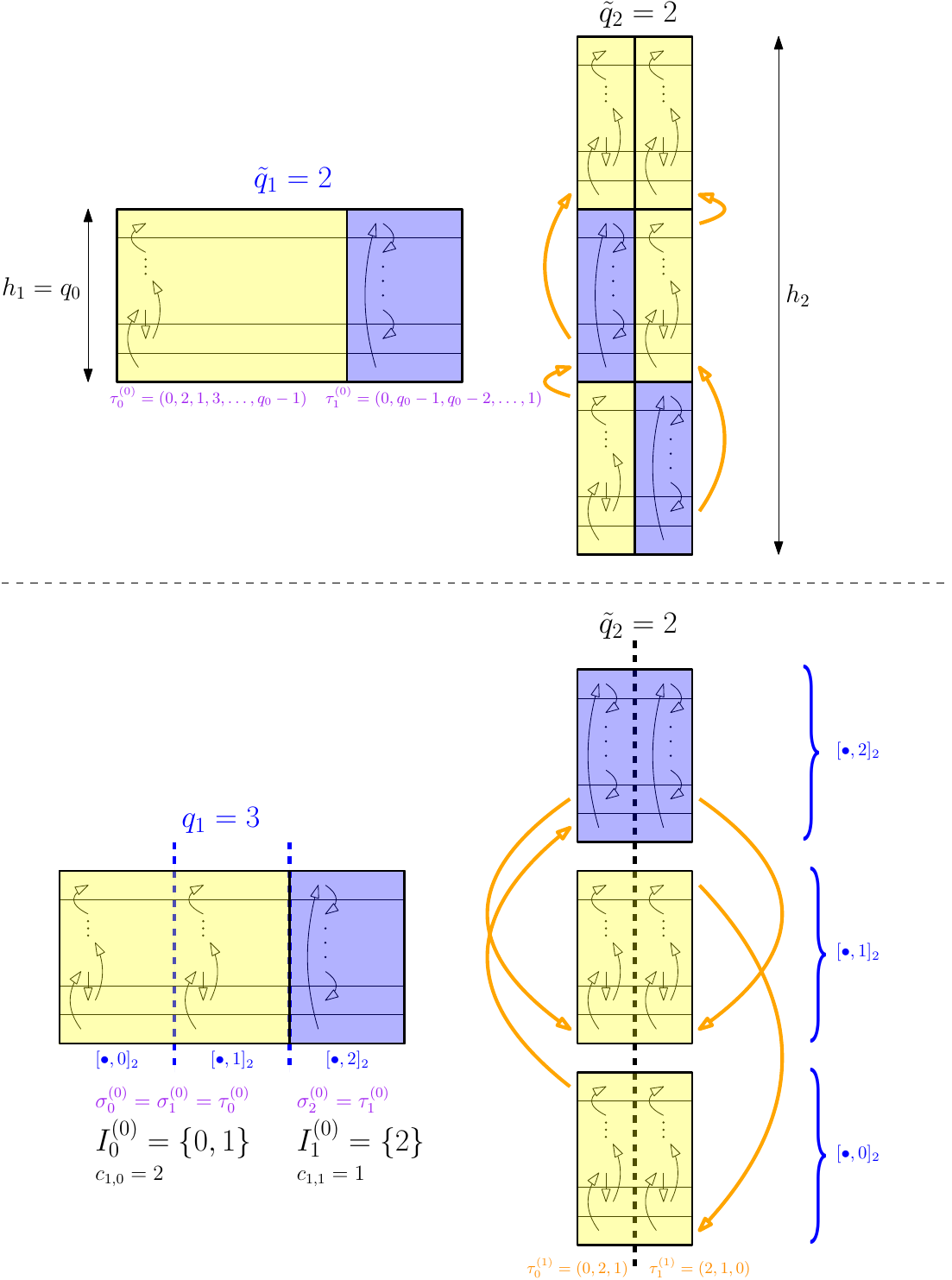}
		\caption{\footnotesize At the top, the second step of a less restrictive cutting-and-stacking construction that we want to describe with an odomutant. At the bottom, the way we encode it with such a system. Here, the dynamics of the yellow tower appears twice in each new towers, so we divide it in two subtowers. Note that the partition $\ptilde(2)$ is exactly the partition which gives the colour (yellow or blue) and the level in the $h_0$-tower to each points of the space, so that we cannot distinguish between points of the two yellow subtowers which are at the same level, contrary to the partition $\p(2)$. For the third step of the construction, the value of $q_2$ will depend on the number ($c_{2,0}$ and $c_{2,1}$) of copies for the dynamics of the two current towers in the next ones.}
		\label{odomutant2}
	\end{figure}
	
	As mentionned in the introduction, our formalism of odomutants was inspired by Feldman's construction~\cite{feldmanNewKautomorphismsProblem1976} of a non-loosely Bernoulli system. As we will see in the proof of Theorem~\ref{thA}, this system is an odomutant built with uniformly $\bm{c}$-multiple permutations where the integers $c_n$ are powers of $2$ and for a fixed $n$, the permutations $\tau^{(n)}_i$ are pairwise different at each step.
	
	\subsection{Odomutants as p.m.p. bijections on a standard probability space}
	
	In this section, we study odomutants with a measure-theoretic viewpoint.
	
	\subsubsection{First properties}
	
	\begin{proposition}\label{aut}
		$T$ is a bijection from $\psi^{-1}(\xmaxinfty)$ to $\psi^{-1}(\xmininfty)$, its inverse is given by
		$$T^{-1}y=\psi_n^{-1}S^{-1}\psi_n(y)$$
		for every $y\in\psi^{-1}(\xmininfty)$ and any $n\geq\nmin(\psi(y))$. Moreover $T$ is an element of $\aut$ and $\psi$ is a factor map from $T$ to $S$.
	\end{proposition}
	
	\begin{proof}[Proof of Proposition~\ref{aut}]
		The equality $\psi_n(Tx)=S\psi_n(x)$ implies $\psi(Tx)=S\psi(x)$ since $\psi_n$ converges pointwise to $\psi$. Moreover, the map $\psi$ preserves the measure $\mu$ and is onto (see Proposition~\ref{preservonto}). Thus, assuming that $T$ is in $\aut$, $S$ is a factor of $T$ via the factor map $\psi$.\par
		Since $\xmaxinfty$ is the increasing union of the sets $\xmaxn$, and for every $n\geq 0$, $T$ and $\psi_n^{-1}S\psi_n$ coincide on $\xmaxn$, the injectivity of $T$ on $\psi^{-1}(\xmaxinfty)$ follows from the injectivity of $S$ and the maps $\psi_n$ and $\psi_n^{-1}$.\par
		For $x\in\psi^{-1}(\xmaxinfty)$, we have $\psi(Tx)=S\psi(x)$ and $\psi(x)\not=\xmax$, so $\psi(Tx)$ is not equal to $\xmin$. Conversely, for $y\in\psi^{-1}(\xmininfty)$, the element $x\coloneq\psi_n^{-1}S^{-1}\psi_n(y)$ does not depend on the choice of an integer $n\geq\nmin(\psi(y))$ (these are the same computations as before Definition~\ref{defodomutant}) and satisfies $Tx=y$.\par
		By $\psi$-invariance, the sets $\psi^{-1}(\xmaxinfty)$ and $\psi^{-1}(\xmininfty)$ have full measure, so $T\colon X\to X$ is a bijection up to measure zero. It follows again from the properties of $S$ and the maps $\psi_n$ that $T$ is bimeasurable and preserves the measure $\mu$.
	\end{proof}
	
	The next result provides a criterion for $\psi$ to be an isomorphism between $T$ and $S$. We will not apply it in this paper but it enables us to understand that, in case permutations have common fixed points\footnote{For Theorem~\ref{thA} (resp.~Theorem~\ref{thB}), we will require $\sigma^{(n)}_i(0)=0$ (resp.~$\sigma^{(n)}_i(0)=0$ and $\sigma^{(n)}_i(q_n-1)=q_n-1$).} (see Section~\ref{fixedpoint}), we will need the sequence $(q_n)_{n\geq 0}$ to increase quickly enough, otherwise we get an odomutant $T$ conjugate to $S$.
	
	\begin{lemma}\label{fixedpoint}
		For every $n\geq 0$, we set
		$$F_n\coloneq\{x_n\in\{0,\ldots,q_n-1\}\mid \forall x_{n+1}\in\{0,\ldots,q_{n+1}-1\},\ \sigma^{(n)}_{x_{n+1}}(x_n)=x_n\}.$$
		If the series $\sum{\frac{|F_n|}{q_n}}$ diverges, then $\psi$ is an isomorphism between $S$ and $T$.
	\end{lemma}
	
	\begin{proof}[Proof of Lemma~\ref{fixedpoint}]
		By the Borel-Cantelli lemma, the set
		$$X_0\coloneq\{(x_n)_{n\geq 0}\in X\mid x_n\in F_n\text{ for infinitely many integers }n\}$$
		has full measure. It is also $S$-, $T$- and $\psi$-invariant and it is easy to check that $\psi\colon X_0\to X_0$ is a bijection, using the fact that the equality $\sigma^{(n)}_{x_{n+1}}(x_n)=y_n$ implies $x_n=y_n$ when $y_n$ is in $F_n$.
	\end{proof}
	
	\begin{remark}
		It is not hard to see, independently of Lemma~\ref{fixedpoint}, that in order to prove Theorems~\ref{thA} and~\ref{corthB}, one needs the sequence $(q_n)_{n\geq 0}$ to be unbounded. Otherwise, let $K$ denote an upper bound of the sequence, then the underlying odomutant admits a cutting-and-stacking construction with at most $K$ towers at each step. A system satisfying such property is said to have \textit{rank $K$} (and more generally \textit{finite rank}) and it is well-known that it is loosely Bernoulli and has zero entropy (see~\cite{ferencziSystemsFiniteRank1997}).
	\end{remark}
	
	\begin{question}
		Is it possible to find a necessary and sufficient condition on the permutations $\sigma^{(n)}_{i}$ (for $n\geq 0$ and $0\leq i<q_{n+1}$) for the factor map $\psi$ to be an isomorphism? Since every odometer is coalescent (see Theorem~\ref{thcoalescent}), this would enable us to know whether or not an odomutant is conjugate to its associated odometer.
	\end{question}
	
	The following two results will be useful for some computations in the proofs of Lemma~\ref{lemmaspectrum} and Proposition~\ref{cocycle}. They deal with the well-definedness of powers (positive or negative) of an odomutant at some point of $X$.
	
	\begin{proposition}\label{powerT}
		For $k\in\N$, the following assertion hold.\footnote{For instance, this holds for every $x\in\xmaxinfty$ such that $\psi(x)$ is not in $\orb_S(\xmax)$ (which is also the $S$-orbit of $\xmin$), so the hypothesis holds for a set of points $x$ of full measure.}
		\begin{itemize}
			\item If $\psi(x)$ is in $\bigcap_{i=0}^{k-1}{S^{-i}(\xmaxinfty)}$, then $Tx,T^2x,\ldots,T^kx$ are well-defined and for every $i\in\{ 0,\ldots,k\}$, we have
			$$T^ix=\psi_n^{-1}S^i\psi_n(x)$$
			for any $n\geq\max_{0\leq j\leq i-1}\nmax(\psi(T^jx))$.
			\item If $\psi(x)$ is in $\bigcap_{i=-(k-1)}^{0}{S^{-i}(\xmininfty)}$, then $T^{-1}x,T^{-2}x,\ldots,T^{-k}x$ are well-defined and for every $i\in\{ -(k-1),\ldots,0\}$, we have
			$$T^{-i}x=\psi_n^{-1}S^{-i}\psi_n(x)$$
			for any $n\geq\max_{-(i-1)\leq j\leq 0}\nmin(\psi(T^jx))$.
		\end{itemize}
	\end{proposition}
	
	\begin{proof}[Proof of Proposition~\ref{powerT}]
		For example, let us prove the first point by induction over $k\geq 1$. The proof of the second point is similar.\par
		The result is clear for $k=0$. Let $k\geq 1$. Let us assume that the result holds for $k-1$ and that
		$$\psi(x)\in\bigcap_{i=0}^{k-1}{S_n^{-i}(\xmaxinfty)}.$$
		This implies that $T^{k-1}x$ is well-defined and is equal to $\psi_n^{-1}S^{k-1}\psi_n(x)$ for any $n$ greater than or equal to $\max_{0\leq j\leq k-2}\nmax(\psi(T^jx))$. Moreover $\psi(T^{k-1}x)$ is not equal to $\xmax$. Indeed, the first $n+1$ coordinates of $\psi(T^{k-1}x)$ and $\psi_n(T^{k-1}x)$ are the same and we have
		$$\psi_n(T^{k-1}x)=S^{k-1}\psi_n(x)$$
		for any $n\geq\max_{0\leq j\leq k-2}\nmax(\psi(T^jx))$, so this follows from the fact that $S^{k-1}\psi(x)$ is not equal to $\xmax$. This implies that $T^kx$ is well-defined and equal to $\psi^{-1}_nS\psi_n(T^{k-1}x)$ for any $n\geq\nmax(\psi(T^{k-1}x))$. Finally, for any $n\geq\max_{0\leq j\leq k-1}\nmax(\psi(T^jx))$, we get
		$$T^kx=\psi^{-1}_nS\psi_n(T^{k-1}x)=\psi^{-1}_nS\psi_n(\psi_n^{-1}S^{k-1}\psi_n(x))=\psi_n^{-1}S^{k}\psi_n(x),$$
		hence the result for $k$.
	\end{proof}
	
	\begin{corollary}\label{corpowerT}
		Let $x, y\in X$ and $M\in\N^*$ such that $x_j=y_j$ for every $j\geq M$, and set
		$$K\coloneq\sum_{j=0}^{M-1}{h_j\left (\sigma_{y_{j+1}}^{(j)}(y_j)-\sigma_{x_{j+1}}^{(j)}(x_j)\right )}.$$
		Assume that $x$ and $y$ are different. Then the following hold:
		\begin{itemize}
			\item if $K>0$, then $Tx,T^2x,\ldots,T^Kx$ are well-defined;
			\item if $K<0$, then $T^{-1}x,T^{-2}x,\ldots,T^Kx$ are well-defined.
		\end{itemize}
		Moreover we have $T^Kx=y$.
	\end{corollary}
	
	\begin{remark}
		The proof of the equality $T^Kx=y$ is based on the well-understood case of an odometer, namely the permutations $\sigma^{(n)}_i$ are all identity maps and $T=S$. More precisely, given $w,z\in X$ satisfying $w_j=z_j$ for every $j$ greater than or equal to some $M$, we know that $S^Kw=z$ with
		$$K=\sum_{j=0}^{M-1}{h_j\left (z_j-w_j\right )}.$$
		It remains to apply this well-known fact to $w=\psi_n(x)$ and $z=\psi_n(y)$ for a large enough integer $n$.
	\end{remark}
	
	\begin{proof}[Proof of Corollary~\ref{corpowerT}]
		Let us consider the case $K>0$ (the proof for the other case is similar). By the previous remark, it is clear that we have
		$$y=\psi_{n}^{-1}S^K\psi_n(x)$$
		for every $n\geq M$. Using Proposition~\ref{powerT}, it remains to prove that $S^i\psi(x)$ is not equal to $\xmax$ for every $i\in\{0,\ldots,K-1\}$. If there exists a positive integer $i$ such that $S^i\psi(x)=\xmax$, then we have 
		$$\sigma_{x_{j+1}}^{(j)}(x_j)=q_j-1$$
		for every sufficiently large integers $j$, and
		\begin{align*}
			\displaystyle i&=\displaystyle\sum_{j=0}^{+\infty}{h_j\left (q_i-1-\sigma_{x_{j+1}}^{(j)}(x_j)\right )}\\
			&=\displaystyle \sum_{j=0}^{M-1}{h_j\left (q_i-1-\sigma_{x_{j+1}}^{(j)}(x_j)\right )} + \sum_{j=M}^{+\infty}{h_j\left (q_i-1-\sigma_{x_{j+1}}^{(j)}(x_j)\right )}\\
			&\geq \displaystyle \sum_{j=0}^{M-1}{h_j\left (\sigma_{y_{j+1}}^{(j)}(y_j)-\sigma_{x_{j+1}}^{(j)}(x_j)\right )}.
		\end{align*}
		Therefore $i$ is greater than or equal to $K$ and we are done.
	\end{proof}
	
	\subsubsection{An odomutant and its associated odometer have the same point spectrum}
	
	Since every odomutant $T$ factors onto its associated odometer $S$, we have the inclusion $\spec(S)\subset\spec(T)$ between the point spectrums. We actually show that this is an equality. The following lemma is inspired by Danilenko and Vieprik's methods to study the point spectrum of rank-one systems (see Proposition 3.7 in~\cite{danilenkoExplicitRank1Constructions2023}).
	
	\begin{lemma}\label{lemmaspectrum}
		Let $T$ be an odomutant built from the odometer $S$ on $X=\prod_{n\geq 0}{\{0,\ldots,q_n-1\}}$ and the families of permutations $(\sigma_{x_{n+1}}^{(n)})_{0\leq x_{n+1}<q_{n+1}}$. If $\lambda\in\T$ is an eigenvalue of $T$, then for every $\varepsilon>0$, there exists a positive integer $n$ such that for every $m\geq n$, there exist $E_{n,m}\subset\prod_{j=n}^{m}{\{0,\ldots,q_j-1\}}$ and $x_{m+1}\in\{0,\ldots q_{m+1}-1\}$ satisfying the following:
		\begin{itemize}
			\item $\displaystyle\frac{|E_{n,m}|}{q_nq_{n+1}\ldots q_m}>1-\varepsilon$;
			\item for every $(y_n,\ldots,y_m),(z_n,\ldots,z_m)\in E_{n,m}$, we have
			$$\left |1-\lambda^{\sum_{j=n}^{m}{h_j\left (\sigma_{y_{j+1}}^{(j)}(y_j)-\sigma_{z_{j+1}}^{(j)}(z_j)\right )}}\right |<\varepsilon,$$
			with $y_{m+1}=z_{m+1}\coloneq x_{m+1}$.
		\end{itemize}
	\end{lemma}
	
	\begin{proof}[Proof of Lemma~\ref{lemmaspectrum}]
		Let $\varepsilon>0$, $\lambda\in\T$ an eigenvalue of $T$ and $g_{\lambda}$ an eigenfunction of $T$ associated to $\lambda$. Without loss of generality, we assume that $\varepsilon\leq 1/2$. Moreover, the modulus of $g_{\lambda}$ is almost everywhere constant (since it is $T$-invariant and $T$ is ergodic), so we assume that $g_{\lambda}$ has modulus $1$. There exists $\alpha\in\T$ and a measurable subset $A\subset X$ of positive measure such that \begin{equation}\label{eqlemmaspectrum1}
			\forall x\in A,\ |g_{\lambda}(x)-\alpha|<\varepsilon/2.
		\end{equation}
		Since the partition given by the $n$-cylinders is increasing to the $\sigma$-algebra on $X$ as $n\to +\infty$, we can find $n>0$ and $(x_0,\ldots,x_{n-1})\in\prod_{j=0}^{n-1}{\{0,\ldots,q_j-1\}}$ such that
		$$\mu(A\cap [x_0,\ldots,x_{n-1}]_n)>(1-\varepsilon^2)\mu([x_0,\ldots,x_{n-1}]_n).$$
		Let $m\geq n$. Then there exists $x_{m+1}\in\{0,\ldots,q_{m+1}-1\}$ such that
		\begin{equation}\label{eqlemmaspectrum2}
			\mu(A\cap [x_0,\ldots,x_{n-1},\bullet,\ldots,\bullet,x_{m+1}]_{m+2})>(1-\varepsilon^2)\mu([x_0,\ldots,x_{n-1},\bullet,\ldots,\bullet,x_{m+1}]_{m+2})
		\end{equation}
		and we set
		$$E_{n,m}\coloneq\left\{(y_n,\ldots,y_m)\in\prod_{j=n}^{m}{\{0,\ldots,q_j-1\}}\ \Big | \begin{array}{l}
			\mu(A\cap [x_0,\ldots,x_{n-1},y_n,\ldots,y_m,x_{m+1}]_{m+2})>\\
			\ \ (1-\varepsilon)\mu([x_0,\ldots,x_{n-1},y_n,\ldots,y_m,x_{m+1}]_{m+2})
		\end{array}\right\}.$$
		By Inequality~\eqref{eqlemmaspectrum2}, we get
		$$\frac{|E_{n,m}|}{q_n\ldots q_m}>1-\varepsilon.$$
		Let $(y_n,\ldots,y_m),(z_n,\ldots,z_m)\in E_{n,m}$. Let us set
		$$B_y\coloneq A\cap [x_0,\ldots,x_{n-1},y_n,\ldots,y_m,x_{m+1}]_{m+2},$$
		$$B_z\coloneq A\cap [x_0,\ldots,x_{n-1},z_n,\ldots,z_m,x_{m+1}]_{m+2}.$$
		and
		$$K\coloneq\sum_{j=n}^{m}{h_j\left (\sigma_{y_{j+1}}^{(j)}(y_j)-\sigma_{z_{j+1}}^{(j)}(z_j)\right )}$$
		(with $y_{m+1}=z_{m+1}\coloneq x_{m+1}$). By Corollary~\ref{corpowerT}, the set $T^{-K}\left (B_y\right )$ is included in the cylinder
		$$C\coloneq[x_0,\ldots,x_{n-1},z_n,\ldots,z_m,x_{m+1}]_{m+2},$$
		which implies that $B\coloneq T^{-K}(B_y)\cap B_z$ has positive measure. Indeed, if $B$ were a null set, the cylinder $C$ would contain two subsets $T^{-K}(B_y)$ and $B_z$ of negligeable intersection and we would get $\mu(C)> 2(1-\varepsilon)\mu(C)$ by definition of $E_{n,m}$, this is not possible since $\varepsilon\leq 1/2$.\par
		Then we have $g_{\lambda}(T^Kx)=\lambda^K g_{\lambda}(x)$ for almost every $x\in B$, and since every $x\in B$ is in $A$ and satisfies $T^Kx\in A$, we get $|1-\lambda^K|<\varepsilon$ using~\eqref{eqlemmaspectrum1}.
	\end{proof}
	
	\begin{lemma}\label{lemmacomplex}
		Let $0<\varepsilon<2$ and $\theta=\theta(\varepsilon)>0$ such that
		$$\{\nu\in\T\mid |1-\nu|<\varepsilon\}=\{\exp{\left (2i\pi\tau\right )}\mid -\theta<\tau<\theta\}.$$
		Let $\nu\in\T\setminus\{1\}$ satisfying $|1-\nu|<\varepsilon$. We write it as $\nu=\exp{\left (2i\pi\tau\right )}$ with $-\theta<\tau<\theta$, $\tau\not=0$. If $\varepsilon$ is small enough so that $\theta<1/4$, then for every interval\footnote{By an interval of $\Z$, we mean a set of the form $\{k\in\Z\mid a\leq k\leq b\}$ for some integers $a$ and $b$.} $J$ of $\Z$, we have
		$$\sum_{j\in J}{\mathds{1}_{|1-\nu^j|<\varepsilon}}\leq\frac{3\theta}{1-2\theta}|J|+\frac{6\theta}{|\tau|}.$$
	\end{lemma}
	
	\begin{proof}[Proof of Lemma~\ref{lemmacomplex}]
		Without loss of generality, we assume that $\tau$ is positive. Let $J$ be an interval of $\Z$. If we have
		$$\sum_{j\in J}{\mathds{1}_{|1-\nu^j|}}=0,$$
		then the result is clear. Now we assume that there exists $j\in J$ such that $|1-\nu^j|<\varepsilon$. Since $\nu$ is not equal to $1$, this implies that we have $|1-\nu^k|<\varepsilon$ for infinitely many integers $k$. Since $\theta$ is less than $1/4$, we also have $|1-\nu^k|\geq\varepsilon$ for infinitely many integers $k$. Therefore we can find sequences $(n_{\ell})_{\ell\in\Z}$ and $(m_{\ell})_{\ell\in\Z}$ of integers such that $n_{\ell}<m_{\ell}<n_{\ell+1}<m_{\ell+1}$ for every $\ell\in\Z$, so that we can write
		$$\Z =\ldots\sqcup C_{-2}\sqcup D_{-2}\sqcup C_{-1}\sqcup D_{-1}\sqcup C_{0}\sqcup D_{0}\sqcup C_{1}\sqcup D_{1}\sqcup\ldots$$
		with intervals $C_{\ell}\coloneq\{k\in\Z\mid n_{\ell}\leq k<m_{\ell}\}$ and $D_{\ell}\coloneq\{k\in\Z\mid m_{\ell}\leq k<n_{\ell+1}\}$ such that
		$$\forall k\in C_{\ell},\ |1-\nu^{k}|<\varepsilon\text{ and }\forall k\in D_{\ell},\ |1-\nu^{k}|\geq\varepsilon.$$
		For every $\ell\in\Z$, we have
		$$(m_{\ell}-n_{\ell}-1)\tau<2\theta <(m_{\ell}-n_{\ell}+1)\tau$$
		$$\text{and }(n_{\ell+1}-m_{\ell}+1)\tau\geq 1-2\theta,$$
		this implies
		$$\frac{2\theta}{\tau}-1<|C_{\ell}|<\frac{2\theta}{\tau}+1$$
		$$\text{and }|D_{\ell}|\geq\frac{1-2\theta}{\tau}-1.$$
		Now we set $\ell_0\coloneq\max{\{\ell\in\Z\mid n_{\ell}\leq\min{J}\}}\text{ and }\ell_1\coloneq\max{\{\ell\in\Z\mid n_{\ell}\leq\max{J}\}}$. We then have the inclusion $\bigsqcup_{\ell_0+1\leq \ell\leq\ell_1-1}{(C_{\ell}\sqcup D_{\ell})}\subset J$ which yields
		$$|J|>(\ell_1-\ell_0-1)\left (\frac{1}{\tau}-2\right ).$$
		Finally, we have
		\begin{align*}
			\displaystyle \sum_{j\in J}{\mathds{1}_{|1-\nu^j|<\varepsilon}}
			&\leq\displaystyle \sum_{\ell=\ell_0}^{\ell_1}{|C_{\ell}|}\\
			&\leq\displaystyle (\ell_1-\ell_0+1)\left (\frac{2\theta}{\tau}+1\right )\\
			&\leq\displaystyle \left (\frac{\frac{2\theta}{\tau}+1}{\frac{1}{\tau}-2}\right )|J|+2\left (\frac{2\theta}{\tau}+1\right )\\
			&\leq\displaystyle \left (\frac{\frac{2\theta}{\tau}+\frac{\theta}{\tau}}{\frac{1}{\tau}-2\frac{\theta}{\tau}}\right )|J|+2\left (\frac{2\theta}{\tau}+\frac{\theta}{\tau}\right )\text{ since }1\leq\frac{\theta}{\tau}\\
			&=\displaystyle\frac{3\theta}{1-2\theta}|J|+\frac{6\theta}{\tau}
		\end{align*}
		and we are done.
	\end{proof}
	
	\begin{theorem}\label{thspectrum}
		Let $T$ be an odomutant built from the odometer $S$ on $X=\prod_{n\geq 0}{\{0,\ldots,q_n-1\}}$. Then $T$ and $S$ have the same point spectrum.
	\end{theorem}
	
	Using the Halmos-von Neumann theorem~\cite{halmosOperatorMethodsClassical1942}, we get the following corollary.
	
	\begin{corollary}
		Let $T$ be an odomutant built from the odometer $S$ on $X=\prod_{n\geq 0}{\{0,\ldots,q_n-1\}}$. If $T$ is conjugate to an odometer, then $T$ is conjugate to $S$.\qed
	\end{corollary}
	
	\begin{proof}[Proof of Theorem~\ref{thspectrum}]
		Since $T$ factors onto $S$, we already know that $\spec(S)\subset\spec(T)$. Let $\lambda$ be an eigenvalue of $T$. Let us show that this is an eigenvalue of $S$. Let $\varepsilon>0$ small enough so that $\theta<1/4$ and $\frac{3\theta}{1-2\theta}\leq 1/4$, with $\theta=\theta(\varepsilon)$ introduced in Lemma~\ref{lemmacomplex}. We also assume that $\varepsilon\leq 1/2$. Let $n$ be a positive integer given by Lemma~\ref{lemmaspectrum} for the eigenvalue $\lambda$, and $\nu\coloneq\lambda^{h_n}$. If $\nu=1$, we are done.\par
		Now assume $\nu\not=1$. Let us choose a sufficiently large enough integer $m$ so that $m\geq n$ and $\frac{6\theta}{\tau q_n\ldots q_m}\leq\frac{1}{4}$. We consider a set $E_{n,m}\subset\prod_{j=n}^{m}{\{0,\ldots,q_j-1\}}$ and an integer $x_{m+1}\in\{0,\ldots,q_{m+1}-1\}$ satisfying
		\begin{itemize}
			\item $\displaystyle\frac{|E_{n,m}|}{q_nq_{n+1}\ldots q_m}>1-\varepsilon$;
			\item for every $y=(y_n,\ldots,y_m),z=(z_n,\ldots,z_m)\in E_{n,m}$, we have
			$$\left |1-\nu^{H(y)-H(z)}\right |<\varepsilon,$$
			where $H\colon\prod_{j=n}^{m}{\{0,\ldots,q_j-1\}}\to\{0,\ldots,q_n\ldots q_m-1\}$ is defined by
			$$H\colon y=(y_n,\ldots,y_m)\mapsto \sum_{j=n}^{m}{\frac{h_j}{h_n}\sigma_{y_{j+1}}^{(j)}(y_j)}\text{ with }y_{m+1}\coloneq x_{m+1}.$$
		\end{itemize}
		The existence of $E_{n,m}$ and $x_m$ is granted by Lemma~\ref{lemmaspectrum}. Since $\varepsilon\leq 1/2$ and $H$ is a bijection, there exists two different elements $y$ and $z$ in $E_{n,m}$ such that $H(y)-H(z)=1$. This implies
		$$|1-\nu|<\varepsilon.$$
		Let us fix $z\in E_{n,m}$ and set
		$$J=\left \{H(y)-H(z)\mid y\in\prod_{j=n}^{m}{\{0,\ldots,q_j-1\}}\right \}.$$
		By Lemma~\ref{lemmacomplex}, we have
		$$\sum_{j\in J}{\mathds{1}_{|1-\nu^j|<\varepsilon}}\leq\frac{3\theta}{1-2\theta}|J|+\frac{6\theta}{|\tau|}\leq \frac{q_n\ldots q_m}{2}$$
		and we get a contradiction since we have
		$$\sum_{j\in J}{\mathds{1}_{|1-\nu^j|<\varepsilon}}\geq |E_{n,m}|> (1-\varepsilon)q_n\ldots q_m$$
		with $\varepsilon\leq 1/2$. Thus we have $\lambda^{h_n}=1$.
	\end{proof}
	
	\subsection{Orbit equivalence between odometers and odomutants}
	
	In this section, we prove that an odomutant and its associated odometer have the same orbits. Moreover, given a non-decreasing map $\varphi\colon\R_+\to\R_+$, we give sufficient conditions for the cocycles to be $\varphi$-integrable.
	
	\begin{proposition}\label{cocycle}
		For all $x\in\psi^{-1}(\xmaxinfty)$, we have $Tx=S^{c_T(x)}x$ where the integer $c_T(x)$ is defined by
		\begin{equation}\label{cocycleT}
			c_T(x)=\sum_{i=0}^{N_1}{h_i(y_i(x)-x_i)}
		\end{equation}
		with $N_1\coloneq\nmax(\psi(x))$ and $y_0,\ldots,y_{N_1}(x)$ inductively defined by
		\begin{align*}
			&y_{N_1}(x)\coloneq\left (\sigma_{x_{N_1+1}}^{(N_1)}\right )^{-1}(\sigma_{x_{N_1+1}}^{(N_1)}(x_{N_1})+1),\\
			\forall\ 0\leq i\leq N_1-1,\ &y_i(x)\coloneq\left (\sigma_{y_{i+1}(x)}^{(i)}\right )^{-1}(0).
		\end{align*}
		For all $x\in\xmaxinfty$, let us define the integer $c_S(x)$ by:
		\begin{equation}
			\begin{aligned}\label{cocycleS}
				\displaystyle c_S(x)=\ &
				\displaystyle h_{N_2}\left (\sigma^{(N_2)}_{x_{N_2+1}}(1+x_{N_2})-\sigma^{(N_2)}_{x_{N_2+1}}(x_{N_2})\right ) \\
				&+\displaystyle h_{N_2-1}\left (\sigma^{(N_2-1)}_{1+x_{N_2}}(0)-\sigma^{(N_2-1)}_{x_{N_2}}(x_{N_2-1})\right )\\
				&+\displaystyle \sum_{i=0}^{N_2-2}{h_i\left (\sigma^{(i)}_{0}(0)-\sigma^{(i)}_{x_{i+1}}(x_{i})\right )}
			\end{aligned}
		\end{equation}
		with $N_2\coloneq\nmax(x)$. Then we have $Sx=T^{c_S(x)}x$ for every $x\in\xmaxinfty$.
	\end{proposition}
	
	\begin{proof}[Proof of Proposition~\ref{cocycle}]
		For $x\in\psi^{-1}(\xmaxinfty)$, the value of $c_T(x)$ follows from the computations before Definition~\ref{defodomutant}. For $x\in\xmaxinfty$ and $N_2\coloneq\nmax(x)$, we have
		$$x=(q_0-1,\ldots,q_{N_2-1}-1,\underbrace{x_{N_2}}_{\not =q_{N_2}-1},x_{N_2+1},x_{N_2+2},\ldots)$$
		$$\text{and }Sx=(0,\ldots,0,1+x_{N_2},x_{N_2+1},x_{N_2+2},\ldots)$$
		so the second result is clear by Corollary~\ref{corpowerT}.
	\end{proof}
	
	\begin{theorem}\label{oeq}
		The map $\Psi\coloneq id_X$ is an orbit equivalence between $T$ and $S$. Moreover, given an non-decreasing map $\varphi\colon\R_+\to\R_+$, this orbit equivalence is $\varphi$-integrable if one of the following two conditions is satisfied:
		\begin{enumerate}[label=(C\arabic*)]
			\item\label{oecond2} the series $\sum{\frac{\varphi(h_{n+1})}{h_n}}$ converges;
			\item\label{oecond1} the series
			$$\sum_{n\geq 0}{\frac{1}{h_{n+2}}\sum_{\substack{0\leq x_n<q_{n},\\0\leq x_{n+1}<q_{n+1},\\
						\sigma^{(n)}_{x_{n+1}}(x_n)\not=q_{n}-1}}{\varphi\left (h_{n}\left (1+\left |\left (\sigma^{(n)}_{x_{n+1}}\right )^{-1}(\sigma^{(n)}_{x_{n+1}}(x_n)+1)-x_n\right |\right )\right )}}$$
			$$\text{and }\sum_{n\geq 0}{\frac{1}{h_{n+2}}\sum_{\substack{0\leq x_n\leq q_{n}-2,\\0\leq x_{n+1}<q_{n+1}}}{\varphi\left (h_{n}\left (1+\left |\sigma^{(n)}_{x_{n+1}}(1+x_n)-\sigma^{(n)}_{x_{n+1}}(x_n)\right |\right )\right )}}$$
			converge.
		\end{enumerate}
	\end{theorem}
	
	As we notice in the next proof, we need coarse bounds to get that Condition~\ref{oecond2} implies $\varphi$-integrably orbit equivalence, whereas Condition~\ref{oecond1} is a finer hypothesis. For Theorem~\ref{thC}, Condition~\ref{oecond1} will enable us to exploit the sublinearity of the map $\varphi$, and Condition~\ref{oecond2} will be enough for Theorems~\ref{thA} and~\ref{thB}.
	
	\begin{proof}[Proof of Theorem~\ref{oeq}]
		By Proposition~\ref{cocycle}, the set of points $x\in X$ satisfying $Tx=S^{c_T(x)}x$ and $Sx=T^{c_S(x)}x$ for integers $c_T(x)$ and $c_S(x)$ defined by~\eqref{cocycleT} and~\eqref{cocycleS} have full measure, so the map $id_X$ is an orbit equivalence between $S$ and $T$.\par
		The value of $c_T(x)$ gives the following bound:
		\begin{equation}\label{boundcocycle}
			|c_T(x)|\leq h_{N_1} \left |\left (\sigma^{(N_1)}_{x_{N_1+1}}\right )^{-1}(\sigma^{(N_1)}_{x_{N_1+1}}(x_{N_1})+1)-x_{N_1}\right |+\underbrace{\sum_{i=0}^{N_1-1}{h_{i} \left |y_i(x)-x_i\right |}}_{\leq h_{N_1}}
		\end{equation}
		with $N_1=\nmax(\psi(x))$. Given $n\geq 0$, $z_n\in\{0,\ldots,q_n-1\}$ and $z_{n+1}\in\{0,\ldots,q_{n+1}-1\}$ such that $\sigma^{(n)}_{z_{n+1}}(z_n)\not=q_n-1$, we have
		$$\mu(\{x\in X \mid \nmax(\psi(x))=n, x_{n}=z_n, x_{n+1}=z_{n+1}\})=\frac{1}{h_{n+2}}.$$
		We finally get
		\begin{align*}
			\displaystyle\int_{X }{\varphi(|c_T(x)|)\mathrm{d}\mu(x)}&=\displaystyle\sum_{n\geq 0}{\sum_{\substack{0\leq z_n<q_{n},\\0\leq z_{n+1}<q_{n+1},\\
						\sigma^{(n)}_{z_{n+1}}(z_n)\not=q_{n}-1}}{\int_{\substack{\nmax(\psi(x))=n,\\x_{n}=z_n,\\x_{n+1}=z_{n+1}}}{\varphi(|c_T(x)|)\mathrm{d}\mu(x)}}}\\
			&\leq\displaystyle\sum_{n\geq 0}{\frac{1}{h_{n+2}}\sum_{\substack{0\leq z_n<q_{n},\\0\leq z_{n+1}<q_{n+1},\\
						\sigma^{(n)}_{z_{n+1}}(z_n)\not=q_{n}-1}}{\varphi\left (h_{n}\left (1+\left |\left (\sigma^{(n)}_{z_{n+1}}\right )^{-1}(\sigma^{(n)}_{z_{n+1}}(z_n)+1)-z_n\right |\right )\right )}}.
		\end{align*}
		From Inequality~\eqref{boundcocycle}, we also get $|c_T(x)|\leq h_{N_1+1}$ and the following coarser bound:
		\begin{align*}
			\displaystyle\int_{X }{\varphi(|c_T(x)|)\mathrm{d}\mu(x)}&=\displaystyle\sum_{n\geq 0}{\sum_{\substack{0\leq z_n<q_{n},\\0\leq z_{n+1}<q_{n+1},\\
						\sigma^{(n)}_{z_{n+1}}(z_n)\not=q_{n}-1}}{\int_{\substack{\nmax(\psi(x))=n,\\x_{n}=z_n,\\x_{n+1}=z_{n+1}}}{\varphi(|c_T(x)|)\mathrm{d}\mu(x)}}}\\
			&\leq\displaystyle\sum_{n\geq 0}{\frac{1}{h_{n+2}}\sum_{\substack{0\leq z_n<q_{n},\\0\leq z_{n+1}<q_{n+1},\\
						\sigma^{(n)}_{z_{n+1}}(z_n)\not=q_{n}-1}}{\varphi\left (h_{n+1}\right )}}\\
			&\leq\displaystyle\sum_{n\geq 0}{\frac{1}{h_{n}}\varphi(h_{n+1})}.
		\end{align*}
		For the other cocycle, we have
		\begin{align*}
			|c_S(x)|&\leq h_{N_2} \left |\sigma^{(N_2)}_{x_{N_2+1}}(1+x_{N_2})-\sigma^{(N_2)}_{x_{N_2+1}}(x_{N_2})\right |\\
			&+\underbrace{h_{N_2-1}\left |\sigma^{(N_2-1)}_{1+x_{N_2}}(0)-\sigma^{(N_2-1)}_{x_{N_2}}(x_{N_2-1})\right |+\sum_{i=0}^{N_2-2}{h_{i} \left |\sigma^{(i)}_{0}(0)-\sigma^{(i)}_{x_{i+1}}(x_i)\right |}}_{\leq h_{N_2}}.
		\end{align*}
		with $N_2=\nmax(x)$. Moreover it is easy to get
		$$\mu(\{x\in X \mid \nmax(x)=n, x_{n}=z_n, x_{n+1}=z_{n+1}\})=\frac{1}{h_{n+2}}$$
		for every $n\geq 0$, $z_n\in\{0,\ldots,q_n-2\}$ and $z_{n+1}\in\{0,\ldots,q_{n+1}-1\}$. Thus we find a bound on the $\varphi$-integral of $c_S$ with the same method as $c_T$.
	\end{proof}
	
	\subsection{Extension to a homeomorphism on the Cantor set, strong orbit equivalence}\label{secexthom}
	
	We move on to a topological viewpoint. We give a sufficient condition for an odomutant to have an extension to a homeomorphism. It turns out that in this case the orbit equivalence that we obtained in the last section is a strong orbit equivalence.
	
	\begin{proposition}\label{exthom}
		Assume that $\sigma^{(n)}_i(0)=0$ and $\sigma^{(n)}_i(q_n-1)=q_n-1$ for every $n\geq 0$ and every $0\leq i\leq q_n-1$. Then the odomutant $T$ admits a unique extension, also denoted by $T$, which is a homeomorphism on the whole compact set $X=\prod_{n\geq 0}{\{0,1,\ldots,q_n-1\}}$. It is furthermore strongly orbit equivalent to the associated odometer $S$. In particular, it follows from Proposition~\ref{oeuniqueerg} that $T$ is uniquely ergodic.
	\end{proposition}
	
	\begin{remark}
		In this case, the equality $S\circ\psi(x)=\psi\circ T(x)$ holds for all $x\in X$.
	\end{remark}
	
	\begin{proof}[Proof of Proposition~\ref{exthom}]
		Since, for every $n\geq 0$, the points $0$ and $q_n-1$ are fixed by the $n$-th permutations, $\xmin$ is the only point $x\in X$ satisfying $\psi(x)=\xmin$ and $\xmax$ is the only point $x\in X$ satisfying $\psi(x)=\xmax$. This implies that we have
		$$\psi^{-1}(\xmininfty)=\xmininfty=X\setminus\{\xmin\}\text{ and }\psi^{-1}(\xmaxinfty)=\xmaxinfty=X\setminus\{\xmax\},$$
		and $T$ is a bijection from $X\setminus\{\xmax\}$ to $X\setminus\{\xmin\}$, so we set $T\xmax\coloneq\xmin$. The map $T\colon X\to X$ is now a well-defined bijection.\par
		The odometer $S$ and the maps $\psi_n$ are continuous on $X$ so it is not difficult to see that $T$ is continuous on each point of $X\setminus\{\xmax\}$. It is easy to check the equality $T([q_0-1,\ldots,q_n-1]_{n+1})=[0,\ldots,0]_{n+1}$, so the continuity at $\xmax$ is clear. Therefore $T\colon X\to X$ is continuous and invertible, where $X$ is a Haussdorf compact space, so $T$ is a homeomorphism.\par
		By Proposition~\ref{cocycle}, we have $Tx=S^{c_T(x)}x$ and $Sx=T^{c_S(x)}$ for every $x\in\xmaxinfty$, with $c_T(x)$ and $c_S(x)$ defined by~\eqref{cocycleT} and~\eqref{cocycleS}. These relations are extended at $\xmax$, with $c_T(x^+)=c_S(x^+)=1$. Thus $S$ and $T$ have the same orbits and it is clear that the cocycles are continuous on $\xmaxinfty$ ($\xmax$ is the only point of discontinuity if the cocycles are not continuous).\footnote{We can notice that we have $\{T^n\xmin\mid n\in\N\}=\{S^n\xmin\mid n\in\N\}$ and $\{T^{-n}\xmax\mid n\in\N\}=\{S^{-n}\xmax\mid n\in\N\}$.
		}
	\end{proof}
	
	\section{On non-preservation of loose Bernoullicity property under \texorpdfstring{$\ld^{<1/2}$}{L<1/2} orbit equivalence}\label{secA}
	
	In this section, we prove that $\ld^{<1/2}$ orbit equivalence (in particular Shannon orbit equivalence) does not imply even Kakutani equivalence.
	
	\begin{theorem}\label{thAbis}
		There exists an ergodic probability measure-preserving bijection $T$ which is $\ld^{<1/2}$ orbit equivalent (in particular Shannon orbit equivalent) to the dyadic odometer but not evenly Kakutani equivalent to it.
	\end{theorem}
	
	Feldman~\cite{feldmanNewKautomorphismsProblem1976} has built a zero-entropy system which is not loosely Bernoulli. In his construction, for some partition that we will specify, the elements in $[0,\ldots,0]_n$ produce words, describing the future, which are not pairwise $f$-close for the $f$-metric introduced in Section~\ref{PrelKak} (therefore, the underlying system is not loosely Bernoulli). The goal is to describe his system as an odomutant built from the dyadic odometer and permutations that we are going to define. These permutations will fix $0$, so that we will be able to read the words produced by the points at the bottom of the towers (using Lemmas~\ref{lemmaword} and~\ref{recword}), with respect to the partition that we will consider.\par
	Let us set $\tilde{q}_n\coloneq2^{n+10}$, $q_n\coloneq(\tilde{q}_n)^{2\tilde{q}_{n+1}+3}$ and $c_n\coloneq\frac{q_n}{\tilde{q}_n}=(\tilde{q}_n)^{2\tilde{q}_{n+1}+2}$ for every $n\geq 0$, $h_{0}\coloneq1$ and $h_{n+1}\coloneq q_{n}h_n$. We inductively define words $a^{(n)}_i$ (we keep the notations of Feldman in his paper) for every $n\geq 0$ and every $i\in\{0,\ldots,\tilde{q}_n-1\}$. Let us start with $\tilde{q}_0$ different letters $a^{(0)}_{0},\ldots, a^{(0)}_{\tilde{q}_0-1}$ seen as words of length $h_{0}=1$. For $n\geq 0$, if words $a^{(n)}_{0},\ldots, a^{(n)}_{\tilde{q}_n-1}$ of length $h_{n}$ have been set, then we define new words $a^{(n+1)}_{0},\ldots, a^{(n+1)}_{\tilde{q}_{n+1}-1}$, of length $h_{n+1}$, by
	$$a^{(n+1)}_j=\left\langle\langle a^{(n)}_{0}\rangle^{(\tilde{q}_n)^{2(j+1)}}\langle a^{(n)}_{1}\rangle^{(\tilde{q}_n)^{2(j+1)}}\ldots \langle a^{(n)}_{\tilde{q}_n-1}\rangle^{(\tilde{q}_n)^{2(j+1)}}\right\rangle^{(\tilde{q}_n)^{2(\tilde{q}_{n+1}-j)}},$$
	where $\langle w\rangle^k$ denotes the concatenation of $k$ copies of a word $w$.\par
	For $j\in\{0,1,\ldots,\tilde{q}_{n+1}-1\}$, $\tau^{(n)}_j$ is a permutation of the set $\{0,1,\ldots,q_n-1\}$ which permutes the entries of the finite sequence
	$$\bm{u}\coloneq(\underbrace{a^{(n)}_0,\ldots,a^{(n)}_0}_{c_n\text{ times}},\underbrace{a^{(n)}_1,\ldots,a^{(n)}_1}_{c_n\text{ times}},\ldots,\underbrace{a^{(n)}_{\tilde{q}_n-1},\ldots,a^{(n)}_{\tilde{q}_n-1}}_{c_n\text{ times}})$$
	so that the concatenation gives $a^{(n+1)}_j$, namely $\tau^{(n)}_j$ satisfies
	$$a^{(n)}_j=u_{\tau^{(n)}_j(0)}\cdot u_{\tau^{(n)}_j(1)}\cdot\ldots\cdot u_{\tau^{(n)}_j(q_n-1)}.$$
	We now consider the odomutant $T$ associated to the odometer $S$ on the space $X=\prod_{n\geq 0}{\{0,1,\ldots,q_n-1\}}$, and built with uniformly $\bm{c}$-multiple permutations $\tau^{(n)}_j$ where $\bm{c}\coloneq(c_n,\tilde{q}_n)_{n\geq 0}$.\par
	In view of the cutting-and-stacking construction behind the definition of this odomutant, we can be convinced that $T$ is isomorphic to the non Bernoulli system built by Feldman. However we give more details on the fact that $T$ is not loosely Bernoulli, based on the justifications of Feldman. Given $n\geq 0$, Lemmas~\ref{lemmaword} and~\ref{recword} in Appendix~\ref{seccombi} imply that the words $[\ptilde(1)]_{h_n}(x)$ for $x\in [0,\ldots,0]_n$ (i.e. the points $x$ at the bottom of the towers at step $n$) exactly correspond to the words $a^{(n)}_{0},\ldots,a^{(n)}_{\tilde{q}_n-1}$. As in~\cite{feldmanNewKautomorphismsProblem1976}, the properties we are interested in can be deduced purely from this fact. Indeed, given any point $x$ not necessarily at the bottom of the towers at step $n$, the word $[\ptilde(1)]_{h_n}(x)$ is the concatenation of the tail of some $a^{(n)}_{i}$ and the beginning of some $a^{(n)}_{j}$, and this observation leads us to apply the same reasoning as in~\cite[Step~III and Step~V in p.~36]{feldmanNewKautomorphismsProblem1976} to conclude that $T$ has zero entropy and $(T,\ptilde(1))$ is not loosely Bernoulli using the caracterisation provided by Theorem~\ref{thEquivalenceLB}. Therefore $T$ is not loosely Bernoulli.
	
	\begin{proof}[Proof of Theorem~\ref{thA}]
		Let $S$ be the odometer on $X=\prod_{n\geq 0}{\{0,1,\ldots,q_n-1\}}$ (so $S$ is loosely Bernoulli) and $T$ the odomutant described above, which is not loosely Bernoulli, so that $S$ and $T$ are not evenly Kakutani equivalent by the theory of Ornstein, Rudolph and Weiss (see Theorem~\ref{ORWkakutaniEq}). Note that $S$ is the dyadic odometer since the integers $q_n$ are powers of $2$.\par
		Let us prove that $S$ and $T$ are $\ld^{<1/2}$ orbit equivalent, using Condition~\ref{oecond2} of Theorem~\ref{oeq}. Let us fix a real number $p$ satisfying $0<p<1/2$. We have
		$$h_{n}={h_{n-1}}\tilde{q}_{n-1}^{2\tilde{q}_n+3}= {h_{n-1}}2^{(n+9)(2^{n+11}+3)}$$
		and this gives $h_n=2^{S_n}$ with
		$$S_n\coloneq\sum_{i=1}^{n}{{(i+9)(2^{i+11}+3)}}=Cn2^{n}+D2^{n}+o(2^n)$$
		for some positive constants $C$ and $D$. For a fixed constant $C'\in [C,\frac{C}{2p}]$ and for a sufficiently large integer $n$, we have
		$$Cn2^{n}<S_n<C'n2^{n},$$
		this gives
		$$h_{n}<2^{C'n2^{n}}=\left (2^{C(n-1)2^{n-1}}\right )^{2\frac{C'}{C}} 2^{C'2^{n}}<\left (2^{S_{n-1}}\right )^{2\frac{C'}{C}} 2^{C'2^{n}}=\displaystyle (h_{n-1})^{2\frac{C'}{C}} 2^{C'2^{n}}$$
		and
		$$\frac{{h_{n}}^{p}}{h_{n-1}}<\displaystyle (h_{n-1})^{2\frac{C'}{C}p-1}2^{C'p2^n}<\left (2^{C'(n-1)2^{n-1}}\right )^{2\frac{C'}{C}p-1}2^{C'p2^n}.$$
		Since we have $2\frac{C'}{C}p<1$, the series $\sum{\frac{\varphi(h_{n})}{h_{n-1}}}$ converges for $\varphi(x)=x^{p}$, so we are done by Theorem~\ref{oeq}.
	\end{proof}
	
	\section{On non-preservation of entropy under orbit equivalence with almost \texorpdfstring{$\log$}{log}-integrable cocycles}\label{secB}

	In this section, we prove that orbit equivalence with almost $\log$-integrable cocycles does not preserve entropy. The statement is actually stronger, with a topological framework:
	
	\begin{theorem}\label{thBbis}
		Let $\alpha$ be either a positive real number or $+\infty$. Let $S$ be an odometer whose associated supernatural number $\prod_{p\in\Pi}{p^{k_p}}$ satisfies the following property: there exists a prime number $p_{\star}$ such that $k_{p_{\star}}=+\infty$. Then there exists a Cantor minimal homeomorphism $T$ such that
		\begin{enumerate}
			\item $\htop(T)=\alpha$;
			\item there exists a strong orbit equivalence between $S$ and $T$, which is $\varphi_m$-integrable for all integers $m\geq 0$,
		\end{enumerate}
		where $\varphi_m$ denotes the map $t\to \frac{\log{t}}{\log^{(\circ m)}{t}}$ and $\log^{(\circ m)}$ the composition $\log\circ\ldots\circ\log$ ($m$ times).
	\end{theorem}
	
	We will crucially use the combinatorial lemmas stated in Appendix~\ref{seccombi}. The cases $\alpha<+\infty$ and $\alpha=+\infty$ will be in fact separated, but in both proofs, we will apply the following lemma which will be useful for the quantification of the cocycles.
	
	\begin{lemma}\label{summable}
		Let $(q_n)_{n\geq 0}$ be a sequence of integers greater than or equal to $2$, and let $\beta>0$ such that
		$$\liminf_{n\to +\infty}{\frac{\log{q_n}}{h_n}}\geq\beta$$
		where $h_n\coloneq q_0\ldots q_{n-1}$. Then, for every integer $m\geq 0$, we have
		$$\frac{1}{\log^{\circ m}(q_{n+m})}\leq\exp{(-\beta h_{n})}$$
		for all sufficiently large integers $n$. In particular, the sequence $\left (\frac{1}{\log^{\circ m}(q_n)}\right )_{n\geq 0}$ is summable.
	\end{lemma}
	
	\begin{proof}[Proof of Lemma~\ref{summable}]
		Let us consider an integer $N$ such that
		$$\forall n\geq N,\ \beta h_{n}\geq 1.$$
		For every $n\geq N$, we have
		$$\log{q_{n+1}}\geq\beta h_{n+1}= q_n\times \beta h_n\geq q_n.$$
		By induction, we easily get for every $n\geq N$,
		$$\log^{\circ m}(q_{n+m})\geq q_{n},$$
		so we have $\log^{\circ m}(q_{n+m})\geq\exp{(\beta h_{n})}$.
	\end{proof}
	
	Before the proof of Theorem~\ref{thBbis} in the case $\alpha<+\infty$, we need two preliminary lemmas. The first one (Lemma~\ref{entr}) provides permutations so that we can easily compute the entropy of the underlying odomutant. The second one (Lemma~\ref{choiceqn}) proves that the formula given by Lemma~\ref{entr} enables us to get all possible finite values of the entropy with a proper choice of parameters.
	
	\begin{lemma}\label{entr}
		Let $(q_n)_{n\geq 0}$ be a sequence of integers greater than or equal to $2$ and satisfying $q_{n+1}\leq (q_n-2)!$. Then there exist permutations $\sigma_{x_{n+1}}^{(n)}$, for $n\geq 0$ and $x_{n+1}\in\{0,\ldots,q_{n+1}-1\}$, satisfying the following properties:
		\begin{enumerate}
			\item for every $n\geq 0$, the maps $\sigma_{0}^{(n)},\sigma_{1}^{(n)}, \ldots, \sigma_{q_{n+1}-1}^{(n)}$ are pairwise different permutations of the set $\{0,\ldots,q_n-1\}$, fixing $0$ and $q_n-1$;
			\item the topological entropy of the underlying odomutant is equal to
			$$\lim_{n\to +\infty}{\frac{\log{q_n}}{h_n}}.$$
		\end{enumerate}
	\end{lemma}
	
	\begin{remark}\label{decreasing}
		The entropy $\htop(T)$ is well-defined by Lemma~\ref{exthom}, as well as the limit $\lim{\frac{\log{q_n}}{h_n}}$. Indeed, the sequence $\left (\frac{\log{q_n}}{h_n}\right )_{n\geq 0}$ is decreasing since we have $q_n< (q_{n-1})^{q_{n-1}}$.
	\end{remark}
	
	\begin{proof}[Proof of Lemma~\ref{entr}]
		By Lemma~\ref{prelemmaentr} in Appendix~\ref{seccombi}, we can choose permutations satisfying the first item and such that the following hold for every $n\geq\ell\geq 1$:
		$$q_n\leq N(\p(\ell)_0^{h_n-1})\leq h_{n-1}q_n q_{n-1}^2 2^{q_{n-1}}.$$
		On the one hand, we have
		$$\htop(T)\geq\htop(T,\p(1))=\lim_{n\to +\infty}{\frac{\log{N(\p(1)_0^{h_n-1})}}{h_n}}\geq\lim_{n\to +\infty}{\frac{\log{q_n}}{h_n}}.$$
		On the other hand, this gives
		\begin{align*}
			\displaystyle\htop(T,\p(\ell))&=\displaystyle\lim_{n\to +\infty}{\frac{\log{N(\p(\ell)_0^{h_n-1})}}{h_n}}\\
			&\leq\displaystyle\lim_{n\to +\infty}{\frac{\log{h_{n-1}}+\log{q_n}+2\log{q_{n-1}}+q_{n-1}\log{2}}{h_n}}\\
			&=\displaystyle\lim_{n\to +\infty}{\frac{\log{q_n}}{h_n}}
		\end{align*}
		and we finally get, using Theorem~\ref{proptopentr},
		$$\htop(T)=\lim_{\ell\to +\infty}{\htop(T,\p(\ell))}\leq\lim_{n\to +\infty}{\frac{\log{q_n}}{h_n}}.$$
		Hence the result.
	\end{proof}
	
	\begin{lemma}\label{choiceqn}
		Let $\alpha$ be a positive real number and $\prod_{p\in\Pi}{p^{k_p}}$ a supernatural number. We assume that there exists a prime number $p_{\star}$ such that $k_{p_{\star}}=+\infty$. Then there exists a sequence $(q_n)_{n\geq 0}$ of integers greater than or equal to $2$, satisfying the following properties:
		\begin{enumerate}
			\item we have $q_{n+1}\leq (q_n-2)!$ for every $n\geq 0$;
			\item the sequence $\left (\frac{\log{q_n}}{q_{0}\ldots q_{n-1}}\right )_{n\geq 0}$ tends to $\alpha$;
			\item we have $\sum_{n\geq 0}{\nu_p(q_n)}=k_p$ for every $p\in\Pi$.
		\end{enumerate}
	\end{lemma}
	
	\begin{proof}[Proof of Lemma~\ref{choiceqn}]
		Let $K$ be a large enough power of $p_{\star}$ so that the following property holds: for every integer $q$ satisfying $q\geq K$, we have $(q-2)!\geq K$. Let
		$$N\coloneq\sum_{p\in\Pi\setminus\{p_{\star}\}}{k_p}\in\N\cup\{+\infty\}$$
		and $(p_i)_{1\leq i\leq N}$ be a sequence of prime numbers satisfying $\sum_{1\leq i\leq N}{\mathds{1}_{p_i=p}}=k_p$ for every $p\in\Pi\setminus\{p_{\star}\}$, and $\sum_{1\leq i\leq N}{\mathds{1}_{p_i=p_{\star}}}=0$.\footnote{"$(p_i)_{i\geq 1}$" and "$\sum_{i\geq 1}$" in the case $N=+\infty$.}
		By induction, we build a sequence $(q_n)_{n\geq 0}$ of integers greater than or equal to $2$, an increasing sequence $(i_n)_{n\geq 0}$ and a non-decreasing sequence $(j_n)_{n\geq 0}$ of non-negative integers, satisfying the following properties:
		\begin{enumerate}
			\item $q_0>\frac{2}{\alpha}\log{p_{\star}}$ and $\log{q_0}\geq \alpha+5$;
			\item $K\leq q_{n+1}\leq (q_n-2)!$ for every $n\geq 0$;
			\item for every $n\geq 1$, the following holds:
			$$\alpha +\frac{5}{q_0\ldots q_{n-1}}\leq\frac{\log{q_n}}{q_0\ldots q_{n-1}}\leq \alpha +\frac{2}{q_0\ldots q_{n-2}},$$
			where $q_0\ldots q_{n-2}$ is equal to $1$ if $n=1$;
			\item $\displaystyle q_n=K^{i_n}\prod_{j=j_{n-1}+1}^{j_{n}}{p_j}$ for every $n\geq 0$, with $j_{-1}=j_0=0$ (so we have $q_0=p_{\star}^{i_0})$;
			\item $j_n\underset{n\to\infty}{\to} N$ if $N<+\infty$, $j_n\underset{n\to\infty}{\to} +\infty$ otherwise.
		\end{enumerate}
		Such a sequence $(q_n)_{n\geq 0}$ satisfies the assumptions of the lemma.\par
		We choose a large enough integer $i_0$ such that the hypotheses on $q_0\coloneq K^{i_0}$ are satisfied. Let $n\geq 0$. Assume that the integers $q_0,\ldots,q_{n},i_0,\ldots,i_{n},j_1,\ldots,j_{n}$ have been defined and let us build $q_{n+1},i_{n+1}, j_{n+1}$. In particular, the integers $q_0,\ldots,q_{n}$ satisfy
		$$\forall k\in\{0,\ldots,n\}, \frac{\log{q_k}}{q_0\ldots q_{k-1}}\geq \alpha +\frac{5}{q_0\ldots q_{k-1}}.$$
		Let $j_{n+1}$ be the greatest integer $k$ satisfying
		\begin{itemize}
			\item $k\geq j_{n}$ and, if $N<+\infty$, $k\leq N$;
			\item $\displaystyle K\prod_{j=j_n+1}^{k}{p_j}\leq (q_{n}-2)!$;
			\item $\displaystyle\frac{\log{\left (\prod_{j=j_n+1}^{k}{p_j}\right )}}{q_0\ldots q_n}\leq\frac{\alpha}{2}$.
		\end{itemize}
		Let us consider the sequence $(\alpha_i)_{i\geq 1}$ defined by
		$$\alpha_i\coloneq\frac{\log{\left (K^i\prod_{j=j_n+1}^{j_{n+1}}{p_j}\right )}}{q_0\ldots q_{n}},$$
		and let $I$ be the greatest integer $i$ such that $K^{i}\prod_{j=j_n+1}^{j_{n+1}}{p_j}\leq (q_n-2)!$. The sequence $(\alpha_i)_{i\geq 1}$ is an arithmetic progression with common difference
		$$\frac{\log{K}}{q_0\ldots q_{n}}.$$
		Moreover, we have
		$$\alpha_1\leq\frac{\log{K}}{q_0}+\frac{\log{\left (\prod_{j=j_n+1}^{j_{n+1}}{p_j}\right )}}{q_0\ldots q_n}\leq\alpha,$$
		and, using the assumption on $q_n$ and the inequalities $\log{(k!)}\geq k\log{(k)}-k$ and $\log{k}\leq k$,
		\begin{align*}
			\displaystyle \alpha_{I+1}&=\displaystyle \frac{\log{\left (K^{I+1}\prod_{j=j_n+1}^{j_{n+1}}{p_j}\right )}}{q_0\ldots q_{n}}\\
			&\geq\displaystyle \frac{\log{((q_n-2)!)}}{q_0\ldots q_{n}}\\
			&=\displaystyle \frac{\log{(q_n!)}}{q_0\ldots q_{n}}-\frac{\log{(q_n-1)}}{q_0\ldots q_{n}}-\frac{\log{(q_n)}}{q_0\ldots q_{n}}\\
			&\geq\displaystyle \frac{q_n\log{(q_n)}-q_n}{q_0\ldots q_{n}}-\frac{\log{(q_n-1)}}{q_0\ldots q_{n}}-\frac{\log{(q_n)}}{q_0\ldots q_{n}}\\
			&\geq\displaystyle\frac{\log{(q_n)}}{q_0\ldots q_{n-1}}-\frac{3}{q_0\ldots q_{n-1}}\\
			&\geq\displaystyle\alpha +\frac{2}{q_0\ldots q_{n-1}}.
		\end{align*}
		Therefore, there exists $i'\in\{1,\ldots,I\}$ such that
		$$\alpha +\frac{2}{q_0\ldots q_{n-1}}-\frac{\log{K}}{q_0\ldots q_{n}}\leq\alpha_{i'}\leq\alpha +\frac{2}{q_0\ldots q_{n-1}}.$$
		Since we have $2q_n\geq 2K\geq K+\log{K}\geq 5+\log{K}$, we get
		$$\alpha +\frac{5}{q_0\ldots q_{n}}\leq\alpha_{i'}\leq\alpha +\frac{2}{q_0\ldots q_{n-1}}.$$
		It remains to set $i_{n+1}\coloneq i'$ and $q_{n+1}\coloneq K^{i_{n+1}}\prod_{j=j_n+1}^{j_{n+1}}{p_j}$.\par
		Finally, we have to check that the increasing sequence $(j_n)_{n\geq 1}$ of integers diverges if $N=+\infty$, or converges to $N$ if $N$ is finite. If it was not the case, then there would exist a positive integer $n$ such that the following hold for every $k\geq n$:
		$$Kp_{j_{n}+1}>(q_{k}-2)!\text{ or }\frac{\log{p_{j_{n}+1}}}{q_0\ldots q_k}>\frac{\alpha}{2}.$$
		But the integers $q_k$ are greater than or equal to $2$, so it would mean that the sequence $(q_k)_{k\geq 0}$ is bounded, which is in contradiction with the inequality $\log{q_k}\geq \alpha q_0\ldots q_k$, so $(j_n)_{n\geq 1}$ satisfies the desired property. Hence the lemma.
	\end{proof}
	
	\begin{proof}[Proof of Theorem~\ref{thB} in the case $\alpha<+\infty$]
		Let $\alpha$ be a positive real number and let $S$ be an odometer whose associated supernatural number $\prod_{p\in\Pi}{p^{k_p}}$ satisfies the following property: there exists a prime number $p_{\star}$ such that $k_{p_{\star}}=+\infty$. Without loss of generality, $S$ is the odometer on the Cantor set $X\coloneq\prod_{n\geq 0}{\{0,1,\ldots,q_n-1\}}$, where the sequence $(q_n)_{n\geq 0}$ satisfies $2\leq q_n\leq (q_n-2)!$ for every $n\geq 0$ and $\frac{\log{q_n}}{h_n}\to\alpha$. The existence of such a sequence is granted by Lemma~\ref{choiceqn}. By Lemma~\ref{entr} and Proposition~\ref{exthom}, we can find families of permutations such that the underlying odomutant $T$ is a homeomorphism strongly orbit equivalent to $S$ and its topological entropy is equal to $\alpha$.\par
		Finally, given an increasing map $\varphi\colon\R_+\to\R_+$, the orbit equivalence is $\varphi$-integrable if $(\varphi(h_{n+1})/h_n)_n$ is summable, by Theorem~\ref{oeq} (see Condition~\ref{oecond2}). This holds for $\varphi(x)=\frac{\log(x)}{\log^{\circ m}(x)}$. Indeed, we have
		\begin{align*}
			\displaystyle\frac{\varphi(h_{n+1})}{h_n}&=\displaystyle\frac{1}{\log^{\circ m}(h_{n+1})}\frac{\log{(h_{n+1})}}{h_n}\\
			&\leq\displaystyle\frac{1}{\log^{\circ m}(q_{n})}\left (\frac{\log{(h_{n})}}{h_n}+\frac{\log{q_n}}{h_n}\right )\\
			&\leq\displaystyle\frac{1}{\log^{\circ m}(q_{n})}\left (1+\frac{\log{q_n}}{h_n}\right ),
		\end{align*}
		so using the monotonicity of the sequence $\left (\frac{\log{q_n}}{h_n}\right )_{n\geq 0}$ (see Remark~\ref{decreasing}), we get
		$$\frac{\varphi(h_{n+1})}{h_n}\leq\frac{1}{\log^{\circ m}(q_{n})}(1+\log{q_0})$$
		and we are done by Lemma~\ref{summable} with $\beta=\alpha$.
	\end{proof}
	
	In the case $\alpha=+\infty$, we prove Theorem~\ref{thB} with the same methods as in~\cite{boyleEntropyOrbitEquivalence1994}, but with our formalism. We will consider an odomutant $T$ on $\prod_{n\geq 0}{\{0,1,\ldots,q_n-1\}}$, built with uniform $\bm{c}$-multiple permutations $\tau_j^{(n)}$, where $\bm{c}=(c_n,\tilde{q}_n)_{n\geq 0}$, and for every  $n\geq 0$ and every $0\leq j<\tilde{q}_{n+1}$, $\tau^{(n)}_{j}$ is a permutation on $\{0,1,\ldots,q_n-1\}$ fixing $0$ and $q_n-1$. For every $n\geq 0$, we assume that the map
	$$j\in\{0,1,\ldots,\tilde{q}_{n+1}-1\}\mapsto (\tau^{(n)}_{j}(I^{(n)}_{0}),\ldots,\tau^{(n)}_{j}(I^{(n)}_{\tilde{q}_n-1}))$$
	is $\kappa_{n+1}$-to-one for some positive integer $\kappa_{n+1}$ (as in the assumption of Lemma~\ref{to-one}). Finally, we write $\chi_n\coloneq\frac{\tilde{q}_n}{\kappa_n}$ for every $n\geq 1$. Then we have $q_n=c_n\tilde{q}_n$ for every $n\geq 0$ and $\tilde{q}_n=\kappa_n\chi_n$ for every $n\geq 1$. The sequences $(h_n)$, $(\tilde{q}_n)$, $(c_n)$, $(\kappa_n)$, $(\chi_n)$ respectively correspond to the sequences $(l_k)$, $(m_k)$, $(n_k)$, $(j_k)$, $(\overline{m}_k)$ in~\cite{boyleEntropyOrbitEquivalence1994}. The integer $\chi_{n+1}$ is the number of sequences of the form $(\tau^{(n)}_{j}(I^{(n)}_{0}),\ldots,\tau^{(n)}_{j}(I^{(n)}_{\tilde{q}_n-1}))$ for $j\in\{0,\ldots,\tilde{q}_{n+1}-1\}$, so we have
	$$1\leq\chi_{n+1}\leq\frac{(q_n-2)!}{c_n!^{\tilde{q}_n-2}(c_n-1)!^2},$$
	thus motivating the following lemma.
	
	\begin{lemma}\label{powerK}
		Let $p$, $\tilde{q}$ and $c$ be positive integers and $q\coloneq\tilde{q}c$. Assume that $p\geq 2$ and $q\geq 3$. Then the greatest power of $p$ less than or equal to
		$$\frac{(q-2)!}{c!^{\tilde{q}-2}(c-1)!^2}$$
		is greater than or equal to
		$$\frac{1}{p}\frac{1}{\tilde{q}^2}\left (\frac{1}{ec}\right )^{\tilde{q}}\tilde{q}^{q}.$$
	\end{lemma}
	
	\begin{proof}[Proof of Lemma~\ref{powerK}]
		Using the inequalities
		$$\left (\frac{k}{e}\right )^k\leq k!\leq e\left (\frac{k+1}{e}\right )^{k+1},$$
		we get
		$$\frac{(q-2)!}{c!^{\tilde{q}-2}(c-1)!^2}=\frac{c^2}{q(q-1)}\frac{q!}{(c-1)!^{\tilde{q}}}\frac{1}{c^{\tilde{q}}}\geq\left (\frac{c}{q}\right )^2\frac{\left (\frac{q}{e}\right )^q}{  e^{\tilde{q}}\left (\frac{c}{e}\right )^{c\tilde{q}}  }\frac{1}{c^{\tilde{q}}}=\frac{1}{\tilde{q}^2}\left (\frac{1}{ec}\right )^{\tilde{q}}\tilde{q}^{q}$$
		and we are done.
	\end{proof}
	
	\begin{proof}[Proof of Theorem~\ref{thB} in the case $\alpha=+\infty$]
		Let
		$$N\coloneq\sum_{p\in\Pi\setminus\{p_{\star}\}}{k_p}\in\N\cup\{+\infty\}$$
		and $(p_i)_{1\leq i\leq N}$ be a sequence of prime numbers satisfying $\sum_{1\leq i\leq N}{\mathds{1}_{p_i=p}}=k_p$ for every $p\in\Pi\setminus\{p_{\star}\}$, and $\sum_{1\leq i\leq N}{\mathds{1}_{p_i=p_{\star}}}=0$.\footnote{"$(p_i)_{i\geq 1}$" and "$\sum_{i\geq 1}$" in the case $N=+\infty$.}
		Let us define $c_n=p_{\star}^n$ for every $n\geq 0$. By induction, we build sequences $(\kappa_n)_{n\geq 1}$ and $(\chi_n)_{n\geq 0}$ of integers, and a non-decreasing sequence $(j_n)_{n\geq 1}$ of non-negative integers, satisfying the following properties:
		\begin{enumerate}
			\item for every $n\geq 0$, $\chi_{n+1}$ is the greatest power of $p_{\star}$ less than or equal to $\frac{(q_n-2)!}{c_n!^{\tilde{q}_n-2}(c_n-1)!^2}$, where $\tilde{q}_n=\kappa_n\chi_n$ (with $\kappa_0\coloneq1$) and $q_n\coloneq c_n\tilde{q}_n$;
			\item $\displaystyle \kappa_n=p_{\star}^{h_n}\prod_{j=j_{n-1}+1}^{j_{n}}{p_j}$ for every $n\geq 1$, with $j_0\coloneq0$;
			\item $j_n\underset{n\to\infty}{\to} N$ if $N<+\infty$, $j_n\underset{n\to\infty}{\to} +\infty$ otherwise.
		\end{enumerate}
		Let us define $\tilde{q}_0\coloneq p_{\star}$. Given $n\geq 0$, assume that $\chi_0,\ldots,\chi_n,j_0,\ldots,j_n,\kappa_1,\ldots,\kappa_n$ have been set (if $n=0$, then there is no integer $\kappa_i$). We define $\chi_{n+1}$ as the greatest power of $p_{\star}$ less than or equal to $\frac{(q_{n}-2)!}{c_{n}!^{\tilde{q}_{n}}(c_{n}-1)!^2}$, $j_{n+1}$ as the greatest integer $k$ satisfying
		\begin{itemize}
			\item $j_n\leq k$ and, if $N<+\infty$, $k\leq N$;
			\item $\prod_{j=j_{n}+1}^{k}{p_j}\leq p_{\star}^{h_{n+1}}$,
		\end{itemize}
		and $\kappa_{n+1}\coloneq p_{\star}^{h_{n+1}}\prod_{j=j_{n}+1}^{j_{n+1}}{p_j}$. Let us define $T$ as the odomutant built with uniform $\bm{c}$-multiple permutations $\tau_j^{(n)}$, with $\bm{c}\coloneq(c_n,\tilde{q}_n)_{n\geq 0}$, and assume that the assumption of Lemma~\ref{to-one} in Appendix~\ref{seccombi} is satisfied: for every $n\geq 0$, the map
		$$j\in\{0,1,\ldots,\tilde{q}_{n+1}-1\}\mapsto (\tau^{(n)}_{j}(I^{(n)}_{0}),\ldots,\tau^{(n)}_{j}(I^{(n)}_{\tilde{q}_n-1}))$$
		is $\kappa_{n+1}$-to-$1$. Note that the fact that $\chi_{n+1}$ is less than or equal to
		$$\frac{(q_{n+1}-2)!}{c_{n+1}!^{\tilde{q}_{n+1}}(c_{n+1}-1)!^2}$$
		enables us to find such families of permutations. It is straightforward to prove that $j_n\to +\infty $ if $N=\infty$, or $j_n\to N$ if $N<+\infty$, so $T$ is an odomutant associated to $S$.\par
		Lemma~\ref{to-one} implies
		$$N(\ptilde(\ell)_0^{h_n-1})\geq\frac{\tilde{q}_n}{\prod_{k=\ell}^{n}{\kappa_{k}^{h_n/h_{k}}}}$$
		for all $n\geq\ell\geq 1$. By Lemma~\ref{powerK}, we have for every $i\geq 1$, 
		$$\tilde{q}_i=\kappa_i\chi_i\geq \kappa_i\frac{1}{p_{\star}}\frac{1}{(\tilde{q}_{i-1})^2}\left (\frac{1}{ec_{i-1}}\right )^{\tilde{q}_{i-1}}(\tilde{q}_{i-1})^{q_{i-1}},$$
		this gives
		$$\tilde{q}_i^{1/h_i}\geq \kappa_i^{1/h_i}\left (\frac{1}{p_{\star}\tilde{q}_{i-1}^2}\right )^{1/h_i}\left (\frac{1}{ec_{i-1}}\right )^{1/(c_{i-1}h_{i-1})}\tilde{q}_{i-1}^{1/h_{i-1}}$$
		and we can apply this inequality many times to get
		\begin{align*}
			\displaystyle\tilde{q}_{n}^{1/h_n}&\geq\displaystyle\left (\prod_{i=1}^{n}{\kappa_i^{1/h_{i}}}\right )\left (\prod_{i=1}^{n}{\left (\frac{1}{p_{\star}\tilde{q}_{i-1}^2}\right )^{1/h_i}\left (\frac{1}{ec_{i-1}}\right )^{1/(c_{i-1}h_{i-1})}}\right )\tilde{q}_0\\
			&\geq\displaystyle\left (\prod_{i=\ell}^{n}{\kappa_i^{1/h_{i}}}\right )\left (\prod_{i=1}^{n}{\left (\frac{1}{p_{\star}\tilde{q}_{i-1}^2}\right )^{1/h_i}\left (\frac{1}{ec_{i-1}}\right )^{1/(c_{i-1}h_{i-1})}}\right )p_{\star}^{\ell-1}\tilde{q}_0.
		\end{align*}
		Hence we have,
		$$\frac{\log{N(\ptilde(\ell)_0^{h_n-1})}}{h_n}\geq (\ell-1)\log{p_{\star}}+\log{\tilde{q}_0}-\sum_{i=1}^{n}{\left (\frac{\log{(p_{\star}\tilde{q}_{i-1}^2)}}{h_i}+\frac{\log{(ec_{i-1})}}{c_{i-1}h_{i-1}}\right )}.$$
		It is straightforward to check that the series $\sum_{i=1}^{+\infty}{\left (\frac{\log{(p_{\star}\tilde{q}_{i-1}^2)}}{h_i}+\frac{\log{(ec_{i-1})}}{c_{i-1}h_{i-1}}\right )}$ converges and we denote by $V$ its value. We are now able to get that $T$ has infinite topological entropy:
		$$\htop(T)\geq\lim_{\ell\to +\infty}{\htop(T,\ptilde(\ell))}\geq\lim_{\ell\to +\infty}{\left ((\ell-1)\log{p_{\star}}+\log{\tilde{q}_0}-V\right )}=+\infty.$$
		
		Let us finally check Condition~\ref{oecond2} in Lemma~\ref{oeq} to prove that there exists a strong orbit equivalence between $T$ and $S$, which is $\varphi_m$-integrable for every $m\geq 0$, where $\varphi_m(x)=\frac{\log(x)}{\log^{(\circ m)}(x)}$. We first have $c_n\leq (p_{\star})^{h_n}$, $\chi_n\leq (q_{n-1})^{q_{n-1}}\leq (h_{n})^{q_{n-1}}$ and $\log{\kappa_n}\leq 2h_n\log{p_{\star}}$ by definition, so
		$$\log{h_{n+1}}=\log{h_n}+\log{c_n}+\log{\kappa_n}+\log{\chi_n}\leq (1+3\log{p_{\star}})h_n+q_{n-1}\log{h_n},$$
		this implies
		$$\frac{\log{h_{n+1}}}{h_n}\leq (1+3\log{p_{\star}})+\frac{\log{h_n}}{h_{n-1}}$$
		and we get $\frac{\log{h_{n+1}}}{h_n}=O(n)$. Then, it remains to prove that the sequence $\left (\frac{n}{\log^{(\circ m)}(h_{n+1})}\right )_{n\geq 0}$ is summable. This is a consequence of Lemma~\ref{summable} with $\beta=\log{p_{\star}}$, since we have
		$$\log{q_n}\geq\log{\kappa_n}\geq h_n\log{p_{\star}}$$
		by definition of $\kappa_n$. So there exists a strong orbit equivalence between $T$ and $S$, which is $\varphi_m$-integrable for every $m\geq 0$.
	\end{proof}
	
	\section{Orbit equivalence with almost integrable cocycles}\label{secC}
	
	In this section, we prove that being orbit equivalent to an odometer, with almost integrable cocycles, does not imply being flip-conjugate to it.
	
	\begin{theorem}\label{thCbis}
		Let $\varphi\colon\R_+\to\R_+$ be a sublinear map and $S$ an odometer. There exists a probability measure-preserving transformation $T$ such that $S$ and $T$ are $\varphi$-integrably orbit equivalent but not flip-conjugate.
	\end{theorem}
	
	For Theorems~\ref{thA} and~\ref{thB}, some invariants (loose Bernoullicity property, entropy) ensure that we build an odomutant $T$ which is not flip-conjugate to the associated odometer $S$. For Theorem~\ref{thCbis}, we use the fact that every odometer is coalescent (see Theorem~\ref{thcoalescent}). Given a sublinear map $\varphi\colon\R_+\to\R_+$, the goal is to find families of permutations $\left (\sigma^{(n)}_{x_{n+1}}\right )_{0\leq x_{n+1}<q_{n+1}}$, for $n\geq 0$, such that the factor map
	$$\psi\colon x\in X\to (\sigma_{x_{1}}^{(0)}(x_0),\sigma_{x_{2}}^{(1)}(x_1),\sigma_{x_{3}}^{(2)}(x_2),\ldots)\in X$$
	from the associated odomutant $T$ to $S$ is not an isomorphism, with $\varphi$-integrable cocycles for the orbit equivalence between $S$ and $T$.
	
	\begin{lemma}\label{notinj}
		Let $(q_n)_{n\geq 0}$ be a sequence of integers greater or equal to $2$. For every $n\geq 0$, let $\left (\sigma_{x_{n+1}}^{(n)}\right )_{0\leq x_{n+1}<q_{n+1}}$ be a family of permutations of the set $\{0,1,\ldots ,q_n-1\}$, defined by:
		$$\forall x_{n+1}\in \{0,\ldots,q_{n+1}-1\},\ \forall i\in\{0,\ldots,q_n-1\},\ \sigma^{(n)}_{x_{n+1}}(i)=i+x_{n+1}\text{ mod }q_n.$$
		Assume that the infinite product $\prod_{n\geq 0}{\left (1-\frac{1}{q_n}\right )}$ converges\footnote{\label{footnote}By definition, the infinite product $\prod_{n\geq 0}{\left (1-\frac{1}{q_n}\right )}$ converges if the sequence $\left (\prod_{k= 0}^{n}{\left (1-\frac{1}{q_k}\right )}\right )_{n\geq 0}$ converges to a nonzero real number.}. Then $\psi\colon x\in X\to (\sigma_{x_{n+1}}^{(n)}(x_n))_{n\geq 0}\in X$ is not injective almost everywhere.
	\end{lemma}
	
	\begin{proof}[Proof of Lemma~\ref{notinj}]
		Let $Y_1\coloneq\{x\in X\mid\forall n\geq 0, x_n\not =(q_n-1)\mathds{1}_{n\text{ is even}}\}$ and $Y_2\coloneq\{x\in X\mid\forall n\geq 0, x_n\not =(q_n-1)\mathds{1}_{n\text{ is odd}}\}$. It is straightforward to check that
		$$\mu(Y_1)=\mu(Y_2)=\prod_{n\geq 0}{\left (1-\frac{1}{q_n}\right )}>0.$$
		Let $\theta\colon X\to X$ defined by:
		$$\theta(x)\coloneq(x_n+(-1)^n\text{ mod }q_n)_{n\geq 0}.$$
		The map $\theta$ is in $\aut$ since $X$ can be seen as the compact group $\prod_{n\geq 0}{\Z/q_n\Z}$, with its Haar probability measure $\mu$ and $\theta$ as the translation by $((-1)^n)_{n\geq 0}$. Moreover, $\theta$ is a bijection from $Y_1$ to $Y_2$ and we have $\psi(\theta(x))=\psi(x)$ for all $x\in Y_1$.\par
		Let us prove by contradiction that $\psi$ is not injective almost everywhere. Assume that $\psi$ is injective on a measurable set $X_0$ of full measure. This hypothesis and the equality $\psi\circ\theta=\psi$ on $Y_1$ imply that the sets $X_0$ and $\theta(X_0\cap Y_1)$ are disjoint. This finally gives
		$$\mu((X_0)^c)\geq \mu\left (\theta(X_0\cap Y_1)\right )=\mu(X_0\cap Y_1)=\mu(Y_1)>0$$
		and we get a contradiction since $(X_0)^c$ has zero measure.
	\end{proof}
	
	Before the proof of Theorem~\ref{thC}, we use a lemma stated in~\cite{carderiBelinskayaTheoremOptimal2023} and which enables us to reduce to the case where the sublinear map $\varphi$ is non-decreasing (actually the statement is stronger but we only need the monotonicity).
	
	\begin{lemma}[Lemma 2.12 in~\cite{carderiBelinskayaTheoremOptimal2023}]\label{metriccompatible}
		Let $\varphi\colon\R_+\to\R_+$ be a sublinear function. Then there is a sublinear non-decreasing function $\tilde{\varphi}\colon\R_+\to\R_+$ such that $\varphi(t)\leq\tilde{\varphi}(t)$ for all $t$ large enough.
	\end{lemma}
	
	\begin{proof}[Proof of Theorem~\ref{thC}]
		Let $\varphi\colon\R_+\to\R_+$ be a sublinear map. If $\tilde{\varphi}$ is another sublinear map satisfying $\varphi(t)=O(\tilde{\varphi}(t))$, then $\tilde{\varphi}$-integrability implies $\varphi$-integrability. Therefore, by Lemma~\ref{metriccompatible}, we assume without loss of generality that $\varphi$ is non-decreasing.\par
		Let $(q_n)_{n\geq 0}$ be a sequence of integers greater or equal to $2$ and $S$ the odometer on $X\coloneq\prod_{n\geq 0}{\{0,1,\ldots,q_n-1\}}$. The Halmos-von Neumann theorem implies that $S$ is conjugate to the odometer on $\prod_{n\geq 0}{\{0,1,\ldots,q_{i_{n-1}}\ldots q_{i_n-1}-1\}}$ for any increasing sequence $(i_n)_{n\geq 0}$ satisfying $i_0=0$. Therefore, we can assume without loss of generality that the integers $q_n$ are sufficiently large so that they satisfy the following properties:
		\begin{enumerate}
			\item $\prod_{n\geq 0}{\left (1-\frac{1}{q_n}\right )}$ converges\footref{footnote};
			\item the series $\displaystyle\sum{\frac{\varphi(2h_n)}{h_n}}$ converges.
		\end{enumerate}
		Let $T$ be the odomutant built from $S$ and the same families $\left (\sigma_{x_{n+1}}^{(n)}\right )_{0\leq x_{n+1}<q_{n+1}}$ as in Lemma~\ref{notinj}. By this lemma and Theorem~\ref{thcoalescent}, $S$ and $T$ are not conjugate. Since $S$ is conjugate to its inverse $S^{-1}$ (by the Halmos-von Neumann theorem), $S$ and $T$ are not flip-conjugate.\par
		It remains to quantify the cocycles, using Condition~\ref{oecond1} of Theorem~\ref{oeq}. Let $n\geq 0$ and $x_{n+1}\in\{0,\ldots,q_{n+1}-1\}$, and $i\in\{0,\ldots,q_n-1\}$ such that $x_{n+1}=i\text{ mod }q_n$. For every $x\in\{0,\ldots,q_n-2\}\setminus\{q_n-i-1\}$, we have 
		$$\left (\sigma^{(n)}_{x_{n+1}}\right )^{-1}(\sigma^{(n)}_{x_{n+1}}(x_n)+1)-x_n=\sigma^{(n)}_{x_{n+1}}(1+x_n)-\sigma^{(n)}_{x_{n+1}}(x_n)=1.$$
		For $x_n=q_n-1$, we consider the following bounds:
		$$\left |\left (\sigma^{(n)}_{x_{n+1}}\right )^{-1}(\sigma^{(n)}_{x_{n+1}}(x_n)+1)-x_n\right |\leq q_n$$
		$$\text{and }\left |\sigma^{(n)}_{x_{n+1}}(1+x_n)-\sigma^{(n)}_{x_{n+1}}(x_n)\right |\leq q_n.$$
		We finally get
		
		\begin{align*}
			&\displaystyle\sum_{n\geq 0}{\frac{1}{h_{n+2}}\sum_{\substack{0\leq x_n<q_{n},\\0\leq x_{n+1}<q_{n+1},\\
						\sigma^{(n)}_{x_{n+1}}(x_n)\not=q_{n}-1}}{\varphi\left (h_{n}\left (1+\left |\left (\sigma^{(n)}_{x_{n+1}}\right )^{-1}(\sigma^{(n)}_{x_{n+1}}(x_n)+1)-x_n\right |\right )\right )}}\\
			&=\displaystyle\sum_{n\geq 0}{\frac{1}{h_{n+2}}\sum_{\substack{0\leq x_n\leq q_{n}-2,\\0\leq x_{n+1}<q_{n+1}\\x_n\not=q_n-i-1}}{\varphi\left (h_{n}\left (1+\left |\left (\sigma^{(n)}_{x_{n+1}}\right )^{-1}(\sigma^{(n)}_{x_{n+1}}(x_n)+1)-x_n\right |\right )\right )}}\\
			&\leq\displaystyle\sum_{n\geq 0}{\frac{1}{h_{n+2}}\sum_{0\leq x_{n+1}<q_{n+1}}{\left ((q_n-2)\varphi(2h_{n})+\varphi(h_{n}(1+q_n))\right )}}\\
			&\leq\displaystyle\sum_{n\geq 0}{\frac{\varphi(2h_{n})}{h_{n}}}+\sum_{n\geq 0}{\frac{\varphi(2h_{n+1})}{h_{n+1}}}<+\infty
		\end{align*}
		and similarly
		$$\sum_{n\geq 0}{\frac{1}{h_{n+2}}\sum_{\substack{0\leq x_n\leq q_{n}-2,\\0\leq x_{n+1}<q_{n+1}}}{\varphi\left (h_{n}\left (1+\left |\sigma^{(n)}_{x_{n+1}}(1+x_n)-\sigma^{(n)}_{x_{n+1}}(x_n)\right |\right )\right )}}<\infty,$$
		so $S$ and $T$ are $\varphi$-integrably orbit equivalent.
	\end{proof}
	
	\begin{remark}
		As Theorem~\ref{thB}, the odomutants $T$ in Theorem~\ref{thA} and~\ref{thB} can be built as homeomorphisms, with a strong orbit equivalence between them and the odometers $S$. This is clear for Theorem~\ref{thA} since we may assume $\sigma^{(n,i)}(q_n-1)=q_n-1$ without loss of generality. For Theorem~\ref{thC}, we have to slightly modify the settings in Lemma~\ref{notinj} and its proof. For example, we can define $\sigma^{(n)}_{x_{n+1}}$ as the permutation mapping $0$ to $0$, $q_n-1$ to $q_n-1$ and $i\in\{1,\ldots,q_n-2\}$ to $1+(i-1+x_{n+1}\text{ mod }q_n-2)$. The set $Y_1$ becomes the set of $x\in X$ such that $x_n\not\in \{0,q_n-2,q_n-1\}$ if $n$ is even, $x_n\not\in\{0,1,q_n-1\}$ if $n$ is odd, and vice versa for $Y_2$. Then the ideas remain the same.
	\end{remark}
	
	\appendix
	\section{Some combinatorial properties}\label{seccombi}
	
	In this section, we fix an odomutant $T$ built with uniformly $\bm{c}$-multiple permutations, with $\bm{c}=(c_n,\tilde{q}_n)_{n\geq 0}$ and $q_n\coloneq c_n\tilde{q}_n$. We refer the reader to Definition~\ref{defodostack} for all the notations that we will use, although not defined in this section (for instance the partitions $\ptilde(\ell)$, the segments $I_j^{(\ell)}$, etc).\par
	In the proof of Theorem~\ref{thB}, for combinatorial purposes appearing in the computation of topological entropy, we need to understand the dynamics of this odomutant with respect to the associated partition $\ptilde(\ell)$ for some $\ell$. Indeed, as explained in Example~\ref{clopen}, computing the topological entropy with respect to a clopen partition partly consists in counting words given by the associated coding map. Recall that, given $c_n=1$ for every $n\geq 0$, and an odomutant built with $\bm{c}$-multiple permutations, $\ptilde(\ell)$ is the partition $\p(\ell)$ in $\ell$-cylinders of the space $X=\prod_{n\geq 0}{\{0,\ldots,q_n-1\}}$, as introduced in Example~\ref{clopen}.\par
	As we can notice in the proofs of the following results, it is more convenient for the computations that the permutations have common fixed points (here this is the point $0$), as in Section~\ref{secexthom} when one wants to extend an odomutant to a homeomorphism. With this assumption, at each step of the cutting-and-stacking construction, we can study the words produced by the points in the first level of the towers, and the recurrence relation describing such a word at step $n+1$ as a concatenation of words at step $n$ (Lemmas~\ref{lemmaword} and~\ref{recword}). Counting only these words gives a lower bound of the number of all the words produced by the coding map, thus providing a lower bound of the topological entropy with respect to the clopen partition that we consider. If this lower bound of $\htop(T)$ diverges to $+\infty$, then we have built an odomutant of infinite entropy. This is the strategy that we will apply in the proof of Theorem~\ref{thB} in the case $\alpha=+\infty$, using a lower bound on the number of words provided by Lemma~\ref{to-one} when the odomutant satisfies some assumptions. Note that this lemma is a reformulation of the main ideas of Boyle and Handelman for the proof of their similar statement~\cite[Section~3]{boyleEntropyOrbitEquivalence1994}. In the case $\alpha<+\infty$, we will need an exact formula on the entropy. To this purpose, Lemma~\ref{prelemmaentr} provides an upper bound of the number of all words produced by a coding map, and thus a finer upper bound of the entropy as we see in the proof of Theorem~\ref{thB}.
	
	\begin{lemma}\label{lemmaword}
		Let $\ell\geq 1$ and $T$ be an odomutant built with uniformly $\bm{c}$-multiple permutations fixing $0$.
		\begin{enumerate}
			\item For every $n\geq \ell-1$, for every $x_{n}\in \{0,1,\ldots,q_{n}-1\}$, the set
			$$\{[\ptilde(\ell)]_{h_{n}}(x)\mid x\in [0,\ldots,0,x_{n}]_{n+1}\}$$
			is a singleton, denoted by $\{W(\ptilde(\ell))^{(n)}_{x_{n}}\}$.
			\item The following holds in the case $n=\ell-1$: the preimages of the map
			$$x_{\ell-1}\in\{0,1,\ldots,q_{\ell-1}-1\}\mapsto W(\ptilde(\ell))^{(\ell-1)}_{x_{\ell-1}}$$
			are $I^{(\ell-1)}_{0},\ldots,I^{(\ell-1)}_{\tilde{q}_{\ell-1}}$. Therefore this map is $c_{\ell-1}$-to-$1$.
			\item For every $n< \ell-1$, for every $(x_{n},\ldots,x_{\ell-1})\in\prod_{n\leq i\leq \ell-1}{\{0,1,\ldots,q_{i}-1\}}$, the set
			$$\{[\ptilde(\ell)]_{h_{n}}(x)\mid x\in [0,\ldots,0,x_{n},\ldots,x_{\ell-1}]_{\ell}\}$$
			is a singleton, denoted by $\{W(\ptilde(\ell))^{(n)}_{x_{n},\ldots,x_{\ell-1}}\}$.
			\item For every $n<\ell-1$ and every $(x_{n},\ldots,x_{\ell-2})\in\prod_{n\leq i\leq \ell-2}{\{0,1,\ldots,q_{i}-1\}}$, the preimages of the map
			$$x_{\ell-1}\in\{0,1,\ldots,q_{\ell-1}-1\}\mapsto W(\ptilde(\ell))^{(n)}_{x_n,\ldots,x_{\ell-2},x_{\ell-1}}$$
			are $I^{(\ell-1)}_{0},\ldots,I^{(\ell-1)}_{\tilde{q}_{\ell-1}}$. Therefore this map is $c_{\ell-1}$-to-$1$.
		\end{enumerate}
	\end{lemma}
	
	\begin{remark}\label{remarklemmaword}\ 
		\begin{itemize}
			\item In the case of multiple permutations with $c_n=1$ for every $n\geq 0$ (so $\tilde{q}_n=q_n$), we get $\ptilde(\ell)=\p(\ell)$ and $I^{(\ell)}_j=\{j\}$ for every $\ell\geq 1$ and every $j\in\{0,\ldots,q_{\ell}-1\}$, so the map
			$$x_{\ell-1}\in\{0,1,\ldots,q_{\ell-1}-1\}\mapsto W(\p(\ell))^{(\ell-1)}_{x_{\ell-1}}$$
			is injective.
			\item The first point of the above lemma remains true if we replace $\ptilde(\ell)$ by any partition $\p$ refined by $\p(\ell)$. Indeed, the result is true for the partition $\p(\ell)$ (it suffices to consider $T$ as an odomutant built with multiple permutations and $c_n=1$). Moreover, the word $[\ptilde]_{h_n}(x)$ is obtained from the word $[\ptilde(\ell)]_{h_n}(x)$ by applying letters by letters the projection which maps $P\in\ptilde(\ell)$ to the atom of $\p$ containing $P$.
		\end{itemize}
	\end{remark}
	
	\begin{proof}[Proof of Lemma~\ref{lemmaword}]
		Let $x\in [0,\ldots,0,x_{n}]_{n+1}$. We can write $x=(\underbrace{0,\ldots,0}_{n\text{ times}},x_{n},x_{n+1},\ldots)$. All the permutations fix $0$, so for every $i\geq n-1$, we have
		$$\psi_i(x)=(\underbrace{0,\ldots,0}_{n\text{ times}},\sigma^{(n)}_{x_{n+1}}(x_{n}),\ldots,\sigma^{(i)}_{x_{i+1}}(x_{i}),x_{i+1},x_{i+2},\ldots).$$
		For $k\in \{0,1,\ldots ,h_n-1\}$, let $(k_{0},k_{1},\ldots ,k_{n-1})$ be the $n$-tuple in $\prod_{0\leq i\leq n-1}{\{0,1,\ldots,q_i-1\}}$ satisfying
		$$k=k_{0}+h_1k_{1}+\ldots+h_{n-1}k_{n-1}.$$
		We then have
		$$S^k\psi_i(x)=(k_{0},k_{1},\ldots,k_{n-1},\sigma^{(n)}_{x_{n+1}}(x_{n}),\ldots,\sigma^{(i)}_{x_{i+1}}(x_{i}),x_{i+1},x_{i+2},\ldots)$$
		so $T^{k}x$ is equal to $(y^{(k)}_0,\ldots,y^{(k)}_{n-1},x_{n},x_{n+1},\ldots)$ where $y^{(k)}_i$ defined by
		\begin{align*}
			&y^{(k)}_{n}=x_{n},\\
			\forall\ 0\leq i\leq n-1,\ &y^{(k)}_{i}=\left (\sigma^{(i)}_{y^{(k)}_{i+1}}\right )^{-1}(k_{i}).
		\end{align*}
		Denote by $j(k,\ell-1)$ the integer in $\{0,1,\ldots,\tilde{q}_{\ell-1}\}$ satisfying $y^{(k)}_{\ell-1}\in I^{(\ell-1)}_{j(k,\ell-1)}$. For every $k\in \{0,1,\ldots ,h_n-1\}$, $(y^{(k)}_0,\ldots,y^{(k)}_{n})$ does not depend on $x_{n+1},x_{n+2},\ldots$ and only depends on $x_{n}$, so does the $h_n$-tuple $([y_0^{(k)},\ldots,y_{\ell-2}^{(k)},I^{(\ell-1)}_{j(k,\ell-1)}]_{\ell})_{0\leq k\leq h_n-1}$ which is equal to $[\ptilde(\ell)]_{h_n}(x)$.\par
		In the case $n=\ell-1$, we have $y^{(k)}_{\ell-1}=x_{n}$, so the value of the word $W(\ptilde(\ell))^{(n)}_{x_n}$ only depends on the interval $I^{(n)}_{j}$ containing $x_n$.\par
		We similarly prove the last two items.
	\end{proof}
	
	\begin{lemma}\label{recword}
		Let $\ell\geq 1$ and $T$ be an odomutant built with uniformly $\bm{c}$-multiple permutations fixing $0$. Let us recall the words $W(\ptilde(\ell))^{(n)}_{x_{n}}$ defined in Lemma~\ref{lemmaword}. Then, for every $n\geq \ell-1$ and $x_{n}\in \{0,1,\ldots,q_{n}-1\}$, we have
		$$W(\ptilde(\ell))^{(n+1)}_{x_{n+1}}=W(\ptilde(\ell))^{(n)}_{0}\cdot W(\ptilde(\ell))^{(n)}_{\left (\sigma^{(n)}_{x_{n+1}}\right )^{-1}(1)}\cdot\ldots\cdot W(\ptilde(\ell))^{(n)}_{\left (\sigma^{(n)}_{x_{n+1}}\right )^{-1}(q_{n}-1)}.$$
	\end{lemma}
	
	\begin{proof}[Proof of Lemma~\ref{recword}]
		Given $n\geq\ell-1$, note that we have
		$$\{0,1,\ldots,h_{n+1}-1\}=\bigsqcup_{0\leq i<q_{n}}{\Big (\{0,1,\ldots,h_{n}-1\}+h_{n}i\Big )}.$$
		Moreover if $i$ is in $\{0,1,\ldots,q_{n}-1\}$, if $x_{n+1}$ is in $\{0,1,\ldots,q_{n+1}-1\}$, we have
		$$T^{ih_{n}}([0,\ldots,0,0,x_{n+1}]_{n+2})=[0,\ldots,0,\left (\sigma^{(n)}_{x_{n+1}}\right )^{-1}(i),x_{n+1}]_{n+2}.$$
		This implies that, for a fixed $x\in [0,\ldots,0,0,x_{n+1}]_{n+2}$, the element $y_i\coloneq T^{ih_{n}}(x)$ is in $[0,\ldots,0,\left (\sigma^{(n)}_{x_{n+1}}\right )^{-1}(i)]_{n+1}$ and we get
		$$[\ptilde(\ell)]_{ih_{n},(i+1)h_{n}-1}(x)=[\ptilde(\ell)]_{h_{n}}(T^{ih_{n}}(x))=[\ptilde(\ell)]_{h_{n}}(y_i)=W(\ptilde(\ell))^{(n)}_{\left (\sigma^{(n)}_{x_{n+1}}\right )^{-1}(i)}$$
		by Lemma~\ref{lemmaword}. Finally the $h_{n+1}$-word on $x$ is the following concatenation:
		\begin{align*}
			W(\ptilde(\ell))^{(n)}_{x_{n+1}}&=[\ptilde(\ell)]_{h_{n+1}}(x)\\
			&=[\ptilde(\ell)]_{0,h_{n+1}-1}(x)\\
			&=[\ptilde(\ell)]_{0,h_{n}-1}(x)\cdot [\ptilde(\ell)]_{h_{n},2h_{n}-1}(x)\cdot\ldots\cdot [\ptilde(\ell)]_{h_{n}(q_{n+1}-1),h_{n+1}-1}(x)\displaystyle\\
			&=\displaystyle W(\ptilde(\ell))^{(n)}_{\left (\sigma^{(n)}_{x_{n+1}}\right )^{-1}(0)}\cdot W(\ptilde(\ell))^{(n)}_{\left (\sigma^{(n)}_{x_{n+1}}\right )^{-1}(1)}\cdot\ldots\cdot W(\ptilde(\ell))^{(n)}_{\left (\sigma^{(n)}_{x_{n+1}}\right )^{-1}(q_{n}-1)}\\
			&=\displaystyle W(\ptilde(\ell))^{(n)}_{0}\cdot W(\ptilde(\ell))^{(n)}_{\left (\sigma^{(n)}_{x_{n+1}}\right )^{-1}(1)}\cdot\ldots\cdot W(\ptilde(\ell))^{(n)}_{\left (\sigma^{(n)}_{x_{n+1}}\right )^{-1}(q_{n}-1)}
		\end{align*}
		and we are done.
	\end{proof}
	
	\begin{lemma}\label{to-one}
		Let $T$ be an odomutant built with uniformly $\bm{c}$-multiple permutations $\tau^{(n)}_j$ fixing $0$. Let $(\kappa_n)_{n\geq 1}$ be a sequence of positive integers and assume that for every $n\geq 0$, the map
		$$j\in\{0,1,\ldots,\tilde{q}_{n+1}-1\}\mapsto (\tau^{(n)}_{j}(I^{(n)}_{0}),\ldots,\tau^{(n)}_{j}(I^{(n)}_{\tilde{q}_n-1}))$$
		is $\kappa_{n+1}$-to-$1$ (in particular, $\kappa_{n+1}$ divides $\tilde{q}_{n+1}$)\footnote{Let us go back to the intuition behind uniformly multiple permutations. Since we consider the partitions $\ptilde(\ell)$ instead of $\p(\ell)$, we cannot distinguish between the copies of a subcolumn that we stack to form each tower. Therefore, given two permutations $\tau^{(n)}_j$ and $\tau^{(n)}_{j'}$, if $(\tau^{(n)}_{j}(I^{(n)}_{0}),\ldots,\tau^{(n)}_{j}(I^{(n)}_{\tilde{q}_n-1}))=(\tau^{(n)}_{j'}(I^{(n)}_{0}),\ldots,\tau^{(n)}_{j'}(I^{(n)}_{\tilde{q}_n-1}))$, then we cannot distinguish between the permutations that they encode, although these permutations are different.}. Then, for all $n\geq\ell\geq 1$, we have
		$$\left|\left\{W(\ptilde(\ell))^{(n)}_{x_n}\mid 0\leq x_n\leq q_n-1\right\}\right|\geq\frac{\tilde{q}_n}{\prod_{k=\ell}^{n}{\kappa_{k}^{h_n/h_{k}}}}.$$
	\end{lemma}
	
	\begin{remark}\label{remarklemmato-one}
		In the case of uniform permutations with pairwise different permutations, the lemma implies that
		$$\left|\left\{W(\ptilde(\ell))^{(n)}_{x_n}\mid 0\leq x_n\leq q_n-1\right\}\right|=q_n$$
		so $W(\ptilde(\ell))^{(n)}_{x_n}$ is an injective function of $x_n$. This could also be deduced from Lemma~\ref{recword}. Therefore, odomutants can have more words in their language than odometer, and then their entropy can be positive.
	\end{remark}
	
	\begin{proof}[Proof of Lemma~\ref{to-one}]
		Let $(\ptilde(\ell))_{\ell\geq 1}$ be the sequence of partitions associated to the construction of this odomutant with uniformly $\bm{c}$-multiple permutations. Given $n\geq\ell\geq 1$, we consider the projection $\pi_{n+1,\ell}\colon\ptilde(n+1)\to\ptilde(\ell)$ which maps $P\in\ptilde(n+1)$ to the atom of $\ptilde(\ell)$ containing $P$. This projection induces a map on the set of words with letters in $\ptilde(n+1)$, it consists in projecting each entry on $\ptilde(\ell)$.
		\begin{claim}
			Let $x\in [0,\ldots,0]_n$ and $k\in\{1,\ldots,n\}$. For every $i\in\{0,1,\ldots,\frac{h_n}{h_k}-1\}$ and every $j\in\{0,1,\ldots,q_{k-1}-1\}$, the point $x^{(i,j)}\coloneq T^{ih_k+jh_{k-1}}x$ is in $[0,\ldots,0]_{k-1}$. Moreover, $(x^{(i,j)})_k$ does not depend on $j$ and we have
			$$(x^{(i,j)})_{k}=\sigma_{(x^{(i,j)})_{k+1}}^{(k)}(i_{k})\text{ and }(x^{(i,j)})_{k-1}=\sigma_{(x^{(i,j)})_{k}}^{(k-1)}(j)$$
			where $i_k\coloneq\left\lfloor\frac{i}{q_k}\right\rfloor$.
		\end{claim}
		\begin{cproof}
			Let us write $ih_k=i_kh_k+i_{k+1}h_{k+1}+\ldots +i_{n-1}h_{n-1}$ with $i_m\in\{0,\ldots,q_{m}-1\}$ for every $m\in\{k,\ldots,n-1\}$. Given $j\geq n$, we have
			$$\psi_j(x)=(\underbrace{0,\ldots,0}_{n\text{ times}},\sigma^{(n)}_{x_{n+1}}(x_n),\ldots)$$
			and
			$$S^{ih_k+jh_{k-1}}\psi_j(x)=(\underbrace{0,\ldots,0}_{k-1\text{ times}},j,i_k,\ldots,i_{n-1},\sigma^{(n)}_{x_{n+1}}(x_n),\ldots).$$
			Hence we get $x^{(i,j)}=\psi_j^{-1}S^{ih_k+jh_{k-1}}\psi_j(x)$ for every $j\geq n$, which implies
			\begin{align*}
				&(x^{(i,j)})_n=x_n,\\
				&(x^{(i,j)})_{n-1}=\sigma_{x_n}^{(n-1)}(i_{n-1}),\\
				&(x^{(i,j)})_{n-2}=\sigma_{(x^{(i,j)})_{n-1}}^{(n-2)}(i_{n-2}),\\
				&\vdots\\
				&(x^{(i,j)})_{k}=\sigma_{(x^{(i,j)})_{k+1}}^{(k)}(i_{k}),\\
				&(x^{(i,j)})_{k-1}=\sigma_{(x^{(i,j)})_{k}}^{(k-1)}(j),
			\end{align*}
			so we are done.
		\end{cproof}
		\begin{claim}
			With the hypotheses of the lemma, for every $k\in\{\ell,\ldots,n\}$, the map
			$$\pi_{k+1,k}\colon\left\{W(\ptilde(k+1))^{(n)}_{x_n}\mid 0\leq x_n\leq q_n-1\right\}\to\left\{W(\ptilde(k))^{(n)}_{x_n}\mid 0\leq x_n\leq q_n-1\right\}$$
			is at most $\displaystyle\kappa_{k}^{h_n/h_{k}}$-to-$1$.
		\end{claim}
		\begin{cproof}
			Let $x\in [0,\ldots,0]_n$. We have $[\ptilde(k+1)]_{h_n}(x)=W(\ptilde(k+1))^{(n)}_{x_n}$ and $[\ptilde(k)]_{h_n}(x)=W(\ptilde(k))^{(n)}_{x_n}$. We first write these words as a concatenation of words of size $h_{k-1}$, namely the words $[\ptilde(k+1)]_{h_{k-1}}(T^{mh_{k-1}}x)$ or $[\ptilde(k)]_{h_{k-1}}(T^{mh_{k-1}}x)$ for $m\in\{0,1,\ldots,\frac{h_n}{h_{k-1}}-1\}$. Given $m\in\{0,1,2,\ldots,\frac{h_n}{h_{k-1}}-1\}$, the point $T^{mh_{k-1}}(x)$ is in $[0,\ldots,0]_{k-1}$ by the last claim, so we have
			$$[\ptilde(k+1)]_{h_{k-1}}(T^{mh_{k-1}}(x))=W(\ptilde(k+1))^{(k-1)}_{(T^{mh_{k-1}}(x))_{k-1},(T^{mh_{k-1}}(x))_{k}}$$
			and
			$$[\ptilde(k)]_{h_{k-1}}(T^{mh_{k-1}}(x))=W(\ptilde(k))^{(k-1)}_{(T^{mh_{k-1}}(x))_{k-1}}.$$
			Secondly, we gather these words of length $h_{k-1}$ in groups $M(k+1)_{x_n,i}$ or $M(k)_{x_n,i}$ of $q_{k-1}$ words:
			\begin{align*}
				M(k+1)_{x_n,i}\coloneq&\ W(\ptilde(k+1))^{(k-1)}_{(x^{(i,0)})_{k-1},(x^{(i,0)})_{k}}\cdot W(\ptilde(k+1))^{(k-1)}_{(x^{(i,1)})_{k-1},(x^{(i,1)})_{k}}\\
				&\cdot\ldots\cdot W(\ptilde(k+1))^{(k-1)}_{(x^{(i,q_{k-1}-1)})_{k-1},(x^{(i,q_{k-1}-1)})_{k}}
			\end{align*}
			and
			$$M(k)_{x_n,i}\coloneq W(\ptilde(k))^{(k-1)}_{(x^{(i,0)})_{k-1}}\cdot W(\ptilde(k))^{(k-1)}_{(x^{(i,1)})_{k-1}}\cdot\ldots\cdot W(\ptilde(k))^{(k-1)}_{(x^{(i,q_{k-1}-1)})_{k-1}}$$
			for all $i\in\{0,1,\ldots,\frac{h_n}{h_{k}}-1\}$
			in such a way that we have
			$$W(\ptilde(k+1))^{(n)}_{x_n}=M(k+1)_{x_n,0}\cdot M(k+1)_{x_n,1}\cdot\ldots\cdot M(k+1)_{x_n,\frac{h_n}{h_k}-1}$$
			and
			$$W(\ptilde(k))^{(n)}_{x_n}=M(k)_{x_n,0}\cdot M(k)_{x_n,1}\cdot\ldots\cdot M(k)_{x_n,\frac{h_n}{h_k}-1}.$$
			To prove the lemma, it now remains to prove that, for every $i\in\{0,1,\ldots,\frac{h_n}{h_k}-1\}$, the map
			$$\pi_{k+1,k}\colon\left\{M(k+1)_{x_n,i}\mid 0\leq x_n\leq q_n-1\right\}\to\left\{M(k)_{x_n,i}\mid 0\leq x_n\leq q_n-1\right\}$$
			is at most $\kappa_{k}$-to-$1$. Let us fix a word $M(k)_{x_n,i}$ with $i\in\{0,1,\ldots,\frac{h_n}{h_k}-1\}$ and $x_n\in\{0,1,\ldots,q_n-1\}$. We write $i_k=\lfloor i/q_k\rfloor$. By the last claim, the quantities $(x^{(i,0)})_k,\ldots,(x^{(i,q_{k-1}-1)})_k$ are equal and their common value is denoted by $X_k$, and we have
			\begin{equation}\label{eqlemmacombi}
				(x^{(iq_{k-1},j)})_{k-1}=\left (\sigma^{(k-1)}_{X_{k}}\right )^{-1}(j)
			\end{equation}
			for every $j\in\{0,1,\ldots,q_{k-1}\}$. This first implies that
			$$(x^{(i,0)})_{k-1},(x^{(i,1)})_{k-1},\ldots,(x^{(i,q_{k-1}-1)})_{k-1}$$
			are $q_{k-1}$ pairwise different elements of $\{0,1,\ldots,q_{k-1}-1\}$. Since we know each subword $W(\ptilde(k))^{(k-1)}_{(x^{(i,j)})_{k-1}}$ of $M(k)_{x_n,i}$, the third item of Lemma~\ref{lemmaword} implies that we completely know the sets $I^{(k-1)}_0,\ldots,I^{(k-1)}_{\tilde{q}_{k-1}-1}$, so
			$$\left (\sigma^{(k-1)}_{X_{k}}(I^{(k-1)}_0),\ldots,\sigma^{(k-1)}_{X_{k}}(I^{(k-1)}_{\tilde{q}_{k-1}-1})\right )$$
			is also completely determined. By assumptions, $X_k$ is in the disjoint union of $\kappa_k$ sets of the form $I^{(k)}_j$.\par
			To conclude, we have proved that, if we have $\pi_{k+1,k}(M(k+1)_{y_n,i})=M(k)_{x_n,i}$ for some $y_n\in\{0,1,\ldots,q_n-1\}$, then $M(k+1)_{y_n,i}$ is of the form
			$$W(\ptilde(k+1))^{(k-1)}_{(x^{(i,0)})_{k-1},X_{k}}\cdot W(\ptilde(k+1))^{(k-1)}_{(x^{(i,1)})_{k-1},X_{k}}\cdot\ldots\cdot W(\ptilde(k+1))^{(k-1)}_{(x^{(i,q_{k-1}-1)})_{k-1},X_{k}}$$
			with $X_k$ in the disjoint union of $\kappa_k$ sets of the form $I^{(k)}_j$, and which completely determines $(x^{(i,0)})_{k-1},(x^{(i,1)})_{k-1},\ldots,(x^{(i,q_{k-1}-1)})_{k-1}$ by Equality~\eqref{eqlemmacombi}. But since two elements $X_k$ and $X'_k$ in the same $I^{(k)}_j$ provide the same word $M(k+1)_{y_n,i}$ (by the last item of Lemma~\ref{lemmaword}), we get that there are at most $\kappa_k$ possible values for the word $M(k+1)_{y_n,i}$.    
		\end{cproof}
		By the last claim, the map
		$$\pi_{n+1,\ell}\colon\left\{W(\ptilde(n+1))^{(n)}_{x_n}\mid 0\leq x_n\leq q_n-1\right\}\to\left\{W(\ptilde(\ell))^{(n)}_{x_n}\mid 0\leq x_n\leq q_n-1\right\}$$
		is at most $\displaystyle\left (\prod_{k=\ell}^{n}{\kappa_{k}^{h_n/h_{k}}}\right )$-to-$1$, so we have
		$$\left|\left\{W(\ptilde(\ell))^{(n)}_{x_n}\mid 0\leq x_n\leq q_n-1\right\}\right|\geq\frac{\left|\left\{W(\ptilde(n+1))^{(n)}_{x_n}\mid 0\leq x_n\leq q_n-1\right\}\right|}{\prod_{k=\ell}^{n}{\kappa_{k}^{h_n/h_{k}}}}$$
		and the result follows from the second item of Lemma~\ref{lemmaword}.
	\end{proof}
	
	\begin{lemma}\label{prelemmaentr}
		Let $(q_n)_{n\geq 0}$ be a sequence of integers greater than or equal to $2$ and satisfying $q_{n+1}\leq (q_n-2)!$. Then there exist permutations $\sigma_{x_{n+1}}^{(n)}$, for $n\geq 0$ and $x_{n+1}\in\{0,\ldots,q_{n+1}-1\}$, satisfying the following properties:
		\begin{enumerate}
			\item for every $n\geq 0$, the maps $\sigma_{0}^{(n)},\sigma_{1}^{(n)}, \ldots, \sigma_{q_{n+1}-1}^{(n)}$ are pairwise different permutations of the set $\{0,\ldots,q_n-1\}$, fixing $0$ and $q_n-1$;
			\item the following bounds hold for every $n\geq\ell\geq 1$:
			$$q_n\leq N(\p(\ell)_0^{h_n-1})\leq h_{n-1}q_n q_{n-1}^2 2^{q_{n-1}}$$        
		\end{enumerate}
	\end{lemma}
	
	\begin{proof}[Proof of Lemma~\ref{prelemmaentr}]
		Let us recall that $N(\p(\ell)_0^{h_n-1})$ is equal to the cardinality of $\{[\p(\ell)]_{h_n}(x)\mid x\in X\}$. If the permutations $\sigma_{0}^{(n)},\sigma_{1}^{(n)}, \ldots, \sigma_{q_{n+1}-1}^{(n)}$ are pairwise different for every $n\geq 0$, then we get $N(\p(\ell)_0^{h_n-1})\geq q_n$ for the underlying odomutant (see Remark~\ref{remarklemmato-one} following Lemma~\ref{to-one}).\par
		Given $n\geq 0$, let $i_n\in\{2,\ldots,q_n-2\}$ be such that $(i_n-1)!<q_{n+1}\leq i_n!$ and let us choose any family $\left (\sigma_{x_{n+1}}^{(n)}\right )_{0\leq x_{n+1}<q_{n+1}}$ of pairwise different permutations of the set $\{0,\ldots,q_n-1\}$ fixing $0,i_n+1,i_n+2,\ldots,q_n-1$. Given an integer $\ell\geq 1$, let us find an upper bound of $N(\p(\ell)_0^{h_n-1})$ for every $n\geq\ell$.
		Let $n\geq\ell$ and $x\in X$. There exists $i\in\{0,1,\ldots,h_n-1\}$ such that $y\coloneq T^{-i}x$ is in $[0,\ldots,0,x_n]_{n+1}$. Let us write $z\coloneq T^{h_n}x$. Thus $[\p(\ell)]_{h_n}(x)$ is the concatenation of a final subword of $W(\p(\ell))_{x_n}^{(n)}$ and an initial subword of $W(\p(\ell))_{z_n}^{(n)}$. Writing $i=jh_{n-1}+r$ with integers $j\in\{0,\ldots,q_{n-1}-1\}$ and $r\in\{0,\ldots,h_{n-1}-1\}$, and using Lemma~\ref{recword}, we have
		\begin{align*}
			[\p(\ell)]_{h_n}(x)=&\ w\cdot W(\p(\ell))^{(n-1)}_{\left (\sigma^{(n-1)}_{x_n}\right )^{-1}(j+1)}\cdot\ldots\cdot W(\p(\ell))^{(n-1)}_{\left (\sigma^{(n-1)}_{x_n}\right )^{-1}(q_{n-1}-2)}\cdot W(\p(\ell))^{(n-1)}_{q_{n-1}-1}\\
			&\cdot W(\p(\ell))^{(n-1)}_{0}\cdot W(\p(\ell))^{(n-1)}_{\left (\sigma^{(n-1)}_{z_n}\right )^{-1}(1)}\cdot\ldots\cdot W(\p(\ell))^{(n-1)}_{\left (\sigma^{(n-1)}_{z_n}\right )^{-1}(j-1)}\cdot w'
		\end{align*}
		where $w$ is a final subword of $W(\p(\ell))^{(n-1)}_{\left (\sigma^{(n-1)}_{x_n}\right )^{-1}(j)}$ of length $h_{n-1}-r$, and $w'$ an initial subword of $W(\p(\ell))^{(n-1)}_{\left (\sigma^{(n-1)}_{z_n}\right )^{-1}(j)}$ of length $r$. Therefore, to every word of the form $[\p(\ell)]_{h_n}(x)$ for the points $x\in X$ sharing the same integers $x_n,z_n,j\text{ and }r$, we can associate the family
		$$\left (\left (\sigma^{(n-1)}_{x_n}\right )^{-1}(j),\ldots,\left (\sigma^{(n-1)}_{x_n}\right )^{-1}(q_{n-1}-2),\left (\sigma^{(n-1)}_{z_n}\right )^{-1}(1),\ldots,\left (\sigma^{(n-1)}_{z_n}\right )^{-1}(j)\right ).$$
		In the particular cases $j=0$ and $j=q_{n-1}-1$, this family is respectively equal to
		$$\left (\left (\sigma^{(n-1)}_{x_n}\right )^{-1}(1),\ldots,\left (\sigma^{(n-1)}_{x_n}\right )^{-1}(q_{n-1}-2)\right )$$
		and
		$$\left (\left (\sigma^{(n-1)}_{z_n}\right )^{-1}(1),\ldots,\left (\sigma^{(n-1)}_{z_n}\right )^{-1}(q_{n-1}-2)\right ).$$
		Moreover, this association is injective, as a consequence of Remark~\ref{remarklemmaword} following Lemma~\ref{lemmaword}. Thus we have
		\begin{align*}
			\displaystyle N(\p(\ell)_0^{h_n-1})&\leq\displaystyle \sum_{r=0}^{h_{n-1}-1}{\left (a^{(n)}_1+b^{(n)}_{q_{n-1}-2}+\sum_{j=1}^{q_{n-1}-2}{a^{(n)}_j\times b^{(n)}_j}\right )}\\
			&\leq \displaystyle h_{n-1}\left (a^{(n)}_1+b^{(n)}_{q_{n-1}-2}+\sum_{j=1}^{q_{n-1}-2}{a^{(n)}_j\times b^{(n)}_j}\right )
		\end{align*}
		where $a^{(n)}_j$ and $b^{(n)}_j$ are respectively the cardinality of
		$$\left\{\left (\left (\sigma^{(n-1)}_{x_n}\right )^{-1}(j),\ldots,\left (\sigma^{(n-1)}_{x_n}\right )^{-1}(q_{n-1}-2)\right )\mid x_n\in\{0,1,\ldots,q_n-1\}\right\}$$
		and
		$$\left\{\left (\left (\sigma^{(n-1)}_{x_n}\right )^{-1}(1),\ldots,\left (\sigma^{(n-1)}_{x_n}\right )^{-1}(j)\right )\mid x_n\in\{0,1,\ldots,q_n-1\}\right\}.$$
		
		We now find upper bounds of the quantities $a^{(n)}_j$ and $b^{(n)}_j$, using the properties of the permutations that we have chosen at the beginning of this proof. We have $a^{(n)}_1=q_n$, $a^{(n)}_j\leq i_{n-1}\times\ldots\times j$ if $1\leq j\leq i_{n-1}$ and $a^{(n)}_j=1$ if $i_{n-1}+1\leq j\leq q_{n-1}-2$. We also have $b^{(n)}_j\leq i_{n-1}\times\ldots\times (i_{n-1}-j+1)$ if $1\leq j\leq i_{n-1}-1$ and $b^{(n)}_j= q_n$ if $i_{n-1}\leq j\leq q_{n-1}-2$. We then get
		\begin{align*}
			&\displaystyle a^{(n)}_1+b^{(n)}_{q_{n-1}-2}+\sum_{j=1}^{q_{n-1}-2}{a^{(n)}_j\times b^{(n)}_j}\\
			&\leq \displaystyle 2q_n +\left (\sum_{j=1}^{i_{n-1}-1}{\frac{i_{n-1}!}{(j-1)!}\frac{i_{n-1}!}{(i_{n-1}-j)!}}\right )+q_ni_{n-1}+\left (\sum_{j=i_{n-1}+1}^{q_{n-1}-2}{q_n}\right )\\
			&\leq \displaystyle q_nq_{n-1}+i_{n-1}^2(i_{n-1}-1)!\sum_{j=1}^{i_{n-1}-1}{\binom{i_{n-1}-1}{j-1}}\\
			&\leq q_nq_{n-1}+q_{n-1}^2q_n 2^{i_{n-1}-1}\\
			&\leq q_n q_{n-1}^2 2^{i_{n-1}}\\
			&\leq q_n q_{n-1}^2 2^{q_{n-1}},
		\end{align*}
		and we are done
	\end{proof}

	\section{Further comments on odomutants: Bratteli diagrams, strong orbit equivalence}\label{secbrat}

	\subsection{Bratteli diagrams, strong orbit equivalence}\label{PrelBratteli}
	
	We introduce the most important definitions and results in the context of strong orbit equivalence. For more details, we refer the reader to~\cite{hermanOrderedBratteliDiagrams1992} and~\cite{giordanoTopologicalOrbitEquivalence1995}.
	
	\subsubsection{Bratteli diagrams}
	
	A \textbf{Bratteli diagram} is a graph $B=(V,E)$ with the set of vertices
	$$V=\bigsqcup_{k\geq 0}{V_k}$$
	and the set of edges
	$$E=\bigsqcup_{k\geq 0}{E_k},$$
	where $V_k$ and $E_k$ are finite, $V_0=\{v^{(0)}\}$ and the edges in $E_k$ connect vertices in $V_{k}$ to vertices in $V_{k+1}$ (multiple edges between two vertices are allowed). If $e_k\in E_k$ connects $v_{k}\in V_{k}$ to $v_{k+1}\in V_{k+1}$, we write $s(e_k)=v_{k}$ and $r(e_k)=v_{k+1}$, this provides maps $s\colon E\to V$ (\textbf{source map}) satisfying $s(E_k)\subset V_{k}$ and $r\colon E\to V$ (\textbf{range map}) satisfying $r(E_k)\subset V_{k+1}$. We assume that
	$$\forall v\in V,\ s^{-1}(v)\not =\emptyset$$
	and
	$$\forall v\in V\setminus V_0,\ r^{-1}(v)\not =\emptyset.$$
	For $k<\ell$, a \textbf{path} from $v_k\in V_k$ to $v_{\ell}\in V_{\ell}$ is a tuple $(e_{k},e_{k+1},\ldots ,e_{\ell-1})$ satisfying $s(e_{k})=v_k$, $r(e_i)=s(e_{i+1})$ for every $i\in\{k,\ldots,\ell-2\}$ and $r(e_{\ell-1})=v_{\ell}$.\par
	An \textbf{ordered} Bratteli diagram is a Bratteli diagram together with a linear order in $r^{-1}(v)$ for every $v\in V\setminus V_0$, namely we consider a bijection
	$$r^{-1}(v)\to \{0,1,\ldots,|r^{-1}(v)|-1\}$$
	for every $v\in V\setminus V_0$. Then we consider $E_k$ as a subset of $V_{k}\times V_{k+1}\times\N$: an edge $e_k\in E_k$ is written as $(v_{k},v_{k+1},\rho_k)$ where $v_{k}=s(e_k)$, $v_{k+1}=r(e_k)$ and $\rho_k\in\{0,\ldots,|r^{-1}(e_k)|-1\}$ is the \textbf{rank} of $e_k$ for the linear order in $r^{-1}(v_{k+1})$, we write $\rho_k=\rk(e_k)$.\par
	Let us set
	$$\xb\coloneq\left\{(e_k)_{k\geq 0}\in\prod_{k\geq 0}{E_k}\mid\forall k\geq 0,\ r(e_k)=s(e_{k+1})\right\},$$
	$$\Xmin\coloneq\left\{(e_k)_{k\geq 0}\in \xb\mid\forall k\geq 0,\ \rk(e_k)=0\right\}$$
	$$\text{and }\Xmax\coloneq\left\{(e_k)_{k\geq 0}\in X\mid\forall k\geq 0,\ \rk(e_k)=|r^{-1}(r(e_k))|-1\right\}.$$
	As a subset of $\prod_{k\geq 0}{E_k}$, $\xb$ is endowed with the induced product topology. $\xb$ is a compact and totally disconnected metric space. By definition, the cylinders\footnote{defined as in Section~\ref{PrelOdo}} of $\xb$ are clopen sets and form a basis of the topology.\par
	A Bratteli diagram is \textbf{simple} if there exists a subsequence $(k_n)$ such that for every pair of vertices in $V_{k_n}\times V_{k_{n+1}}$, there exits a path between them. If an ordered Bratteli diagram is simple, then $\xb$ has no isolated points, so it is a Cantor set.\par
	Given a Bratteli diagram $B=(V,E)$, we can enumerate the vertices of each $V_n$:
	$$V_n=\left\{v_0^{(n)},\ldots,v_{|V_n|-1}^{(n)}\right\};$$
	and define the \textbf{incidence matrices}
	$$M_n\coloneq\left (m^{(n)}_{i,j}\right )_{\substack{0\leq i\leq |V_{n+1}|-1\\0\leq j\leq |V_n|-1}}$$
	where $m^{(n)}_{i,j}$ is the number of edges of $E_{n}$ connecting $v^{(n)}_j$ to $v^{(n+1)}_i$.
	
	\subsubsection{Bratteli-Vershik systems}
	
	Given an ordered Bratteli diagram $B$, we define a map $T_B\colon \xb\setminus\Xmax\to\xb\setminus\Xmin$ in the following way.\par
	Let $x=(e_k)_{k\geq 0}\in\xb\setminus\Xmax$ and
	$$N\coloneq\min\{i\geq 0\mid \rk(e_i)<|r^{-1}(r(e_i))|-1\}.$$
	Let $f_N$ be the edge in $r^{-1}(r(e_N))$ satisfying $\rk(f_N)=\rk(e_N)+1$ and $(f_0,\ldots,f_{N-1})$ the minimal path from $v^{(0)}$ to $s(f_{N})$, namely this is the unique path satisfying $\rk(f_i)=0$ for every $i\in\{0,\ldots,N-1\}$.\footnote{We find this path in an inductive way: $f_{N-1}$ is the unique edge satisfying $r(f_{N-1})=s(f_N)$ and $\rk(f_{N-1})=0$, $f_{N-2}$ is the unique edge satisfying $r(f_{N-2})=s(f_{N-1})$ and $\rk(f_{N-2})=0$, and so on.} Then we define
	$$T_Bx\coloneq(f_1,\ldots,f_{N},e_{N+1},e_{N+2},\ldots).$$
	The map $T_B$ is called the \textbf{Bratteli-Vershik system} associated to the ordered Bratteli diagram $B$.\par
	An ordered and simple Bratteli diagram is \textbf{properly ordered} if $\Xmin$ and $\Xmax$ are singletons. Given a properly ordered Bratteli diagram, we extend $T$ to the whole set $\xb$ by setting
	$$T_B(x_{\mathrm{max}})\coloneq x_{\mathrm{min}}$$
	where $\Xmax\coloneq\{x_{\mathrm{max}}\}$ and $\Xmin=\{x_{\mathrm{min}}\}$. In this case, we can check that $T_B$ is a Cantor minimal homeomorphism.\par
	
	\begin{figure}[ht]
		\centering
		\includegraphics[width=0.4\linewidth]{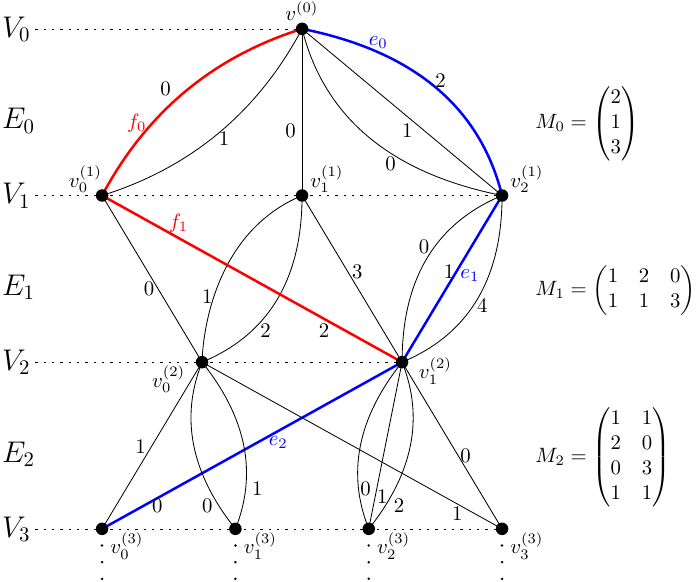}
		\caption{\footnotesize Example of ordered Bratteli diagram $B$. The image of $(e_0,e_1,e_2,\ldots)$ by $T_B$ is $(f_0,f_1,e_2,\ldots)$.
		}
		\label{bratteli diagram}
	\end{figure}
	
	For example, the Bratteli-Vershik system of the diagram
	in Figure~\ref{odometerbrat} is topologically conjugate to the odometer on $X\coloneq\prod_{n\geq 0}{\{0,1,\ldots,q_n-1\}}$, the following map
	$$\Psi\colon (x_n)_{n\geq 0}\in X\mapsto \left (\left (v^{(0)},v^{(1)},x_0\right ),\left (v^{(1)},v^{(2)},x_1\right ),\left (v^{(2)},v^{(3)},x_2\right ),\ldots\right )\in\xb$$
	is a conjugation between them. As we explain in the next part, every Cantor minimal homeomorphism can be described by a Bratteli diagram.
	
	\begin{figure}[ht]
		\centering
		\includegraphics[width=0.5\linewidth]{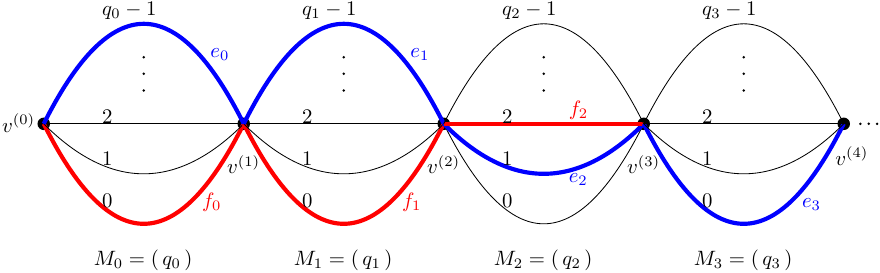}
		\caption{\footnotesize An ordered Bratteli diagram describing the odometer on $\prod_{n\geq 0}{\{0,\ldots,q_n-1\}}$. The image of $(e_0,e_1,e_2,e_3,\ldots)$ by $T_B$ is $(f_0,f_1,f_2,e_3,\ldots)$.
		}
		\label{odometerbrat}
	\end{figure}
	
	\subsubsection{Cantor minimal homeomorphisms}\label{brattelitohom}
	
	The Bratteli-Vershik systems of properly ordered Brattali diagrams describe all the Cantor minimal homeomorphisms.
	
	\begin{theorem*}[Herman, Putnam, Skau~\cite{hermanOrderedBratteliDiagrams1992}]
		If $T$ is a Cantor minimal homeomorphism, then there exists a properly ordered Bratteli diagram $B$ such that the associated Bratteli-Vershik system $T_B$ is topologically conjugate to $T$.
	\end{theorem*}
	
	We briefly describe how a Cantor minimal homeomorphism $T\colon X\to X$ is encoded by a properly ordered Bratteli diagram. All we have to find is an increasing sequence $(\p_n)_{n\geq 0}$ of partitions generating the topology, of the form
	$$\p_n\coloneq\{T^j(B_{n,i})\mid 0\leq i\leq k_n-1,\ 0\leq j\leq h^{(n)}_i-1\}$$
	where $k_n, h^{(n)}_1,\ldots ,h^{(n)}_{k_n}$ are positive integers, and the sequence $(B_n)_{n\geq 0}$ defined by
	$$B_n\coloneq\bigsqcup_{0\leq i\leq k_n-1}{B_{n,i}},$$
	is decreasing to a singleton $\{y\}$.\par
	By "increasing sequence of partitions", we mean that $\p_{n+1}$ is finer than $\p_n$, namely the atoms of $\p_n$ are unions of atoms of $\p_{n+1}$. The partition $\p_n$ is composed of $k_n$ towers and given $i\in\{0,\ldots,k_n-1\}$, the tower
	$$\mathcal{T}_{n,i}\coloneq\{B_{n,i},T(B_{n,i}),\ldots,T^{h^{(n)}_i-1}(B_{n,i})\}$$
	has height $h^{(n)}_i$.\par
	Without the assumption that the partitions have to generate the topology, the construction only consists in considering in an inductive way a clopen subset $B_{n+1}$ of $B_n$, that we partition in $B_{n,0},\ldots,B_{n,k_n-1}$ according to the value of the first return time. The underlying sequence of partitions does not necessarily generate the topology. For a generating sequence, we refer the reader to Lemma 4.1 of in~\cite{hermanOrderedBratteliDiagrams1992} and Lemma 3.1 in~\cite{putnamAlgebrasAssociatedMinimal1989} for more details.\par
	The properly ordered Bratteli diagram $B=(V,E)$ is defined as follows. Assume that $\p_0=X$ (so $k_0=h_1=1$) and define
	$$V_n\coloneq\left\{\mathcal{T}_{n,i}\mid 0\leq i\leq k_n-1\right\}.$$
	Given $n\geq 1$ and $i\in\{0,\ldots,k_n-1\}$, the tower $\mathcal{T}_{n,i}$ visits successively the towers
	$$\mathcal{T}_{n-1,\ell^{(n,i)}_1},\mathcal{T}_{n-1,\ell^{(n,i)}_2},\ldots,\mathcal{T}_{n-1,\ell^{(n,i)}_{r_{n,i}}}$$
	with integers $\ell^{(n,i)}_j\in\{0,\ldots,k_{n-1}-1\}$ and $r_{n,i}\geq 1$. Then $E$ is defined so that $r^{-1}(\mathcal{T}_{n,i})$ has cardinality $r_{n,i}$ and
	$$r^{-1}(\mathcal{T}_{n,i})\coloneq\left\{\left (\mathcal{T}_{n-1,\ell^{(n,i)}_{j+1}},\mathcal{T}_{n,i},j\right )\mid 0\leq j\leq r_{n,i}-1\right\}.$$
	The underlying Bratteli diagram is properly ordered and the associated Bratteli-Vershik system $T_B\colon X_B\to X_B$ is topologically conjugate to $T\colon X\to X$. Note that $x_{\mathrm{min}}\in\xb$ corresponds to the point $y\in X$.\par
	To sum up, a Bratteli diagram encodes a cutting-and-stacking process defining a system (see Figure~\ref{cuttingstacking}).
	
	\begin{figure}[ht]
		\centering
		\includegraphics[width=0.9\linewidth]{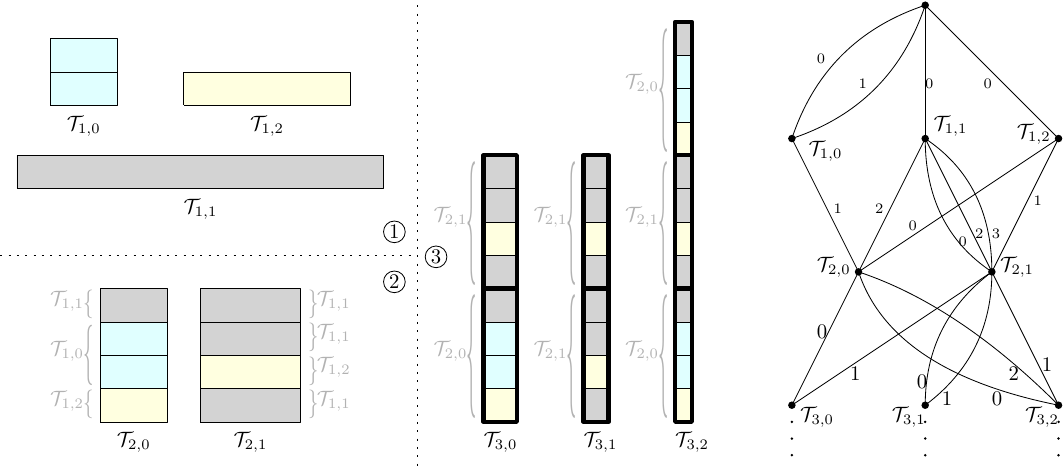}
		\caption{\footnotesize Example of towers $\mathcal{T}_{n,i}$, and the associated Bratteli diagram.
		}
		\label{cuttingstacking}
	\end{figure}
	
	\subsubsection{Classification up to strong orbit equivalence}
	
	Here we present a complete invariant of strong orbit equivalence, due to Giordano, Putnam and Skau.\par
	Recall the incidence matrices
	$$M_n\coloneq\left (m^{(n)}_{i,j}\right )_{\substack{0\leq i\leq |V_{n+1}|-1\\0\leq j\leq |V_n|-1}}$$
	given an enumeration of the vertices of each $V_n$. Then let us define the group $G(B)$ as the following inductive limit
	$$G(B)\coloneq\lim{\Z^{|V_0|}\overset{M_0}{\longrightarrow}\Z^{|V_1|}\overset{M_1}{\longrightarrow}\Z^{|V_2|}\overset{M_2}{\longrightarrow}\ldots}.$$
	With the usual ordering on each $\Z^{|V_n|}$, $G(B)$ has a structure of ordered group, the unit order is chosen as the image of $1$ in $\Z=\Z^{|V_0|}$. The ordered group $G(B)$ is called the \textbf{dimension group} of $B$. We refer the reader to~\cite{giordanoTopologicalOrbitEquivalence1995} for more details.\par
	For instance, for the dyadic odometer, the incidence matrices are all $(1\times 1)$-matrices equal to $(2)$, and the dimension group is $\Z\left[1/2\right ]$.
	
	\begin{theorem}[Giordano, Putnam, Skau~\cite{giordanoTopologicalOrbitEquivalence1995}]
		Let $S$ and $T$ be two Cantor minimal homeomorphisms. The following assertions are equivalent:
		\begin{enumerate}
			\item $S$ and $T$ are strongly orbit equivalent;
			\item If $B$ (resp. $B'$) denotes a Bratteli diagram associated to $S$ (resp. $T$), then the dimension groups $G(B)$ and $G(B')$ with distinguished order unit are order isomorphic.
		\end{enumerate}
	\end{theorem}
	
	\subsection{Bratteli diagrams of odomutants}\label{odoBratteli}
	
	Let $X\coloneq\prod_{n\geq 0}{\{0,\ldots,q_n-1\}}$. Denoting by $\p(n)$ the partition whose atoms are the $n$-cylinders, with $\p(0)=(X)$, note that the sequence $\left (\p(n)\right )_{n\geq 0}$ generates the infinite product topology on $X$ and $\p(n+1)$ is composed of $q_n$ towers of height $h_n$, denoted by
	$$\mathcal{T}_{n+1,i}\coloneq\left \{B_{n+1,i},T(B_{n+1,i}),\ldots,T^{h_n-1}(B_{n+1,i})\right \}$$
	where $B_{n+1,i}\coloneq[0,\ldots,0,i]_{n+1}$, for every $i\in\{0,\ldots,q_n-1\}$ (see Figure~\ref{odomutant}). The atoms of $\mathcal{T}_{n+1,i}$ are the cylinders of the form $[x_0,\ldots,x_{n-1},i]_{n+1}$ with $x_k\in\{0,\ldots,q_k-1\}$ for every $k\in\{0,\ldots,n-1\}$.\par
	Given $n\geq 1$ and $i\in\{0,\ldots,q_n-1\}$, the tower $\mathcal{T}_{n+1,i}$ visits the $n$-th towers with the following order:
	$$\mathcal{T}_{n,\left (\sigma_{i}^{(n-1)}\right )^{-1}(0)},\mathcal{T}_{n,\left (\sigma_{i}^{(n-1)}\right )^{-1}(1)},\ldots,\mathcal{T}_{n,\left (\sigma_{i}^{(n-1)}\right )^{-1}(q_n-1)}.$$
	According to Section~\ref{brattelitohom}, we get the Bratteli diagram $B$ of $T$ illustrated in Figure~\ref{odomutantbrat}.
	
	\begin{figure}[ht]
		\centering
		\includegraphics[width=0.9\linewidth]{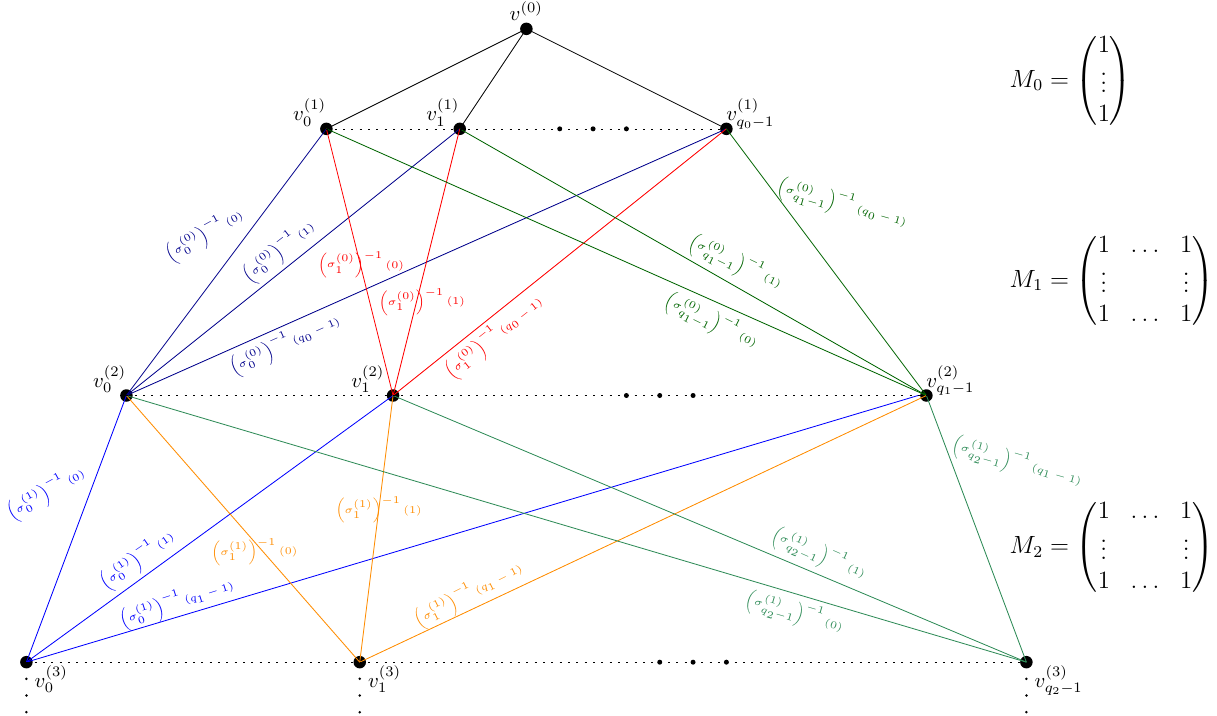}
		\caption{\footnotesize An ordered Bratteli diagram describing the odomutant built from the odometer on $\prod_{n\geq 0}{\{0,\ldots,q_n-1\}}$ and families of permutations $\left (\sigma_{i}^{(n)}\right )_{0\leq i<q_{n+1}}$ for $n\geq 0$.
		}
		\label{odomutantbrat}
	\end{figure}
	
	The following map
	$$\Psi\colon (x_n)_{n\geq 0}\in X\mapsto \left (\left (v^{(0)},v^{(1)}_{x_0},0\right ),\left (v^{(1)}_{x_0},v^{(2)}_{x_1},\sigma^{(0)}_{x_1}(x_0)\right ),\left (v^{(2)}_{x_1},v^{(3)}_{x_2},\sigma^{(1)}_{x_2}(x_1)\right ),\ldots\right )\in\xb$$
	is a conjugation between $T$ and the Bratteli-Vershik system $T_B$, it satisfies
	$$\Psi(\psi^{-1}(\xmaxinfty))=\xb\setminus\Xmax$$
	$$\text{and }\Psi(\psi^{-1}(\xmininfty))=\xb\setminus\Xmin.$$
	In the case the permutations satisfy $\sigma^{(n)}_i(0)=0,\ \sigma^{(n)}_i(q_n-1)=q_n-1$ for every $n\geq 0$, the Bratteli diagram is proprely ordered and we have
	$$\Xmax=\left\{\Psi(\xmax)\right\}$$
	$$\text{and }\Xmin=\left\{\Psi(\xmin)\right\}.$$
	
	\subsection{Comparisons between Boyle and Handelman's system and our odomutants.}
	
	As mentioned in the introduction, Boyle and Handelman have shown the following result.
	
	\begin{theorem*}[Boyle, Handelman~\cite{boyleEntropyOrbitEquivalence1994}]
		Let $S$ be the dyadic odomete. Let $\alpha$ be either a positive real number or $+\infty$. Then there exists a Cantor minimal homeomorphism $T$ such that:
		\begin{enumerate}
			\item $S$ and $T$ are strongly orbit equivalent;
			\item $\htop(T)=\alpha$.
		\end{enumerate}
	\end{theorem*}
	
	In their proof, they build a Bratteli diagram $B_{\mathrm{BH}}$ (see Figure~\ref{odomutantboylehandelman}) similar to the diagram in Figure~\ref{odomutantbrat}, the only difference is that for every $k\geq 1$, for every $v_i^{(k)}\in V_k$, $v_j^{(k+1)}\in V_{k+1}$, with $0\leq i<q_{k-1}$ and $0\leq j\leq q_k-1$, there are $n_k$ edges connecting these vertices. Then the ideas remain almost the same. Every vertex $v_j^{(k+1)}\in V_{k+1}$ provides a permutation $\sigma^{(k-2)}_j$ on the $n_kq_{k-1}$ edges of range $v_j^{(k+1)}$, satisfying
	$$\sigma^{(k-2)}_j(0)=0$$
	$$\text{and }\sigma^{(k-2)}_j(n_kq_{k-1}-1)=n_kq_{k-1}-1,$$
	so that the diagram is properly ordered and the associated Bratteli-Vershik system $T\coloneq T_{B_{\mathrm{BH}}}$ can be extended to a homeomorphism on the Cantor set. The permutations are chosen in order to get $\htop(T)=\alpha$ (we refer the reader to their proof for more details, note that their proof in the case $\alpha=+\infty$ has been here entirely reformulated in our formalism).
	
	\begin{figure}[ht]
		\centering
		\includegraphics[width=0.9\linewidth]{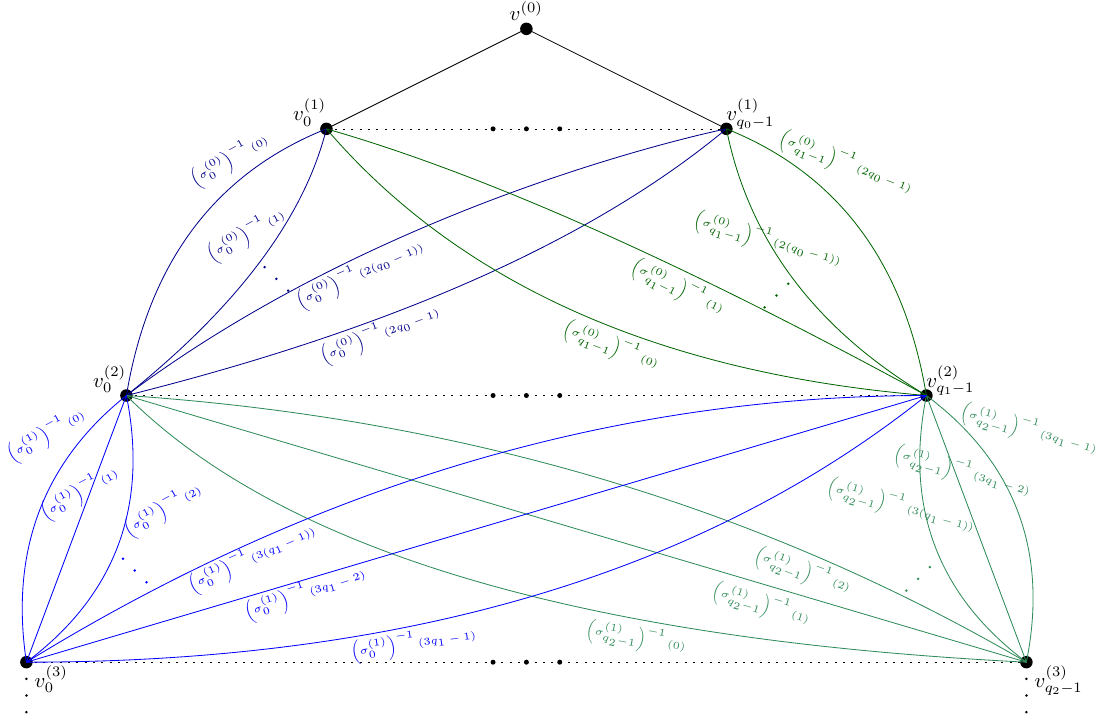}
		\caption{\footnotesize Bratteli diagram built by Boyle and Handelman in the proof of their Theorem 2.8~\cite{boyleEntropyOrbitEquivalence1994}, with $n_1=2$, $n_2=3$.
		}
		\label{odomutantboylehandelman}
	\end{figure}
	
	It turns out that their Bratteli diagram is a diagram for an odomutant. Let us recall the following facts.
	
	\begin{itemize}
		\item In a Bratteli diagram, for some fixed vertex $v^{(k+1)}_j\in V_{k+1}$, with $0\leq j\leq q_k-1$, the set of edges $r^{-1}(v^{(k+1)}_j)$, with its linear ordering, encodes the stacking of subtowers of $\mathcal{T}_{k,0},\ldots,\mathcal{T}_{k,q_{k-1}-1}$ to build the tower $\mathcal{T}_{k+1,j}$ (see Figure~\ref{cuttingstacking}).
		\item In the cutting-and-stacking construction of an odomutant described in Figure~\ref{odomutant}, every tower $\mathcal{T}_{k,i}$, with $0\leq i<q_{k-1}$, is uniformly cut in $q_{k+1}$ subtowers $\left (\mathcal{T}_{k,i}(\ell)\right )_{0\leq \ell <q_{k+1}}$ and we build every tower $\mathcal{T}_{k+1,j}$, with $0\leq j<q_{k}$, by choosing only one subtower in each $\mathcal{T}_{k,0},\ldots,\mathcal{T}_{k,q_{k-1}-1}$ and stacking them.
	\end{itemize}
	
	The Bratteli diagram $B_{\mathrm{BH}}$ of Boyle and Handelman describes the following cutting-and-stacking construction: every tower $\mathcal{T}_{k,i}$ is uniformly cut in $n_kq_{k+1}$ subtowers $\left (\mathcal{T}_{k,i}(\ell)\right )_{0\leq \ell\leq n_kq_{k+1}-1}$ and we build every tower $\mathcal{T}_{k+1,j}$ by choosing exactly $n_k$ subtowers in each $\mathcal{T}_{k,0},\ldots,\mathcal{T}_{k,q_{k-1}-1}$ and stacking them.
	
	As explained in Section~\ref{odocutsta}, to understand why this is equivalent to the construction of an odomutant, it suffices to cut each $k$-th tower $\mathcal{T}_{k,i}$ in $n_k$ (sub)towers $\mathcal{T}_{k,(i,0)},\ldots,\mathcal{T}_{k,(i,n_k-1)}$, in such a manner that for every $(k+1)$-th tower $\mathcal{T}_{k+1,j}$, each tower $\mathcal{T}_{k,(i,m)}$ contains only one of the $n_k$ subtowers $\mathcal{T}_{k,i}(\ell)$ which form $\mathcal{T}_{k+1,j}$. We then replace the former $k$-th towers $\mathcal{T}_{k,i}$, with $0\leq i<q_{k-1}$, by the new ones $\mathcal{T}_{k,(i,m)}$, with $0\leq i<q_{k-1}$ and $0\leq m\leq n_k-1$, and we recover the cutting-and-stacking process of an odomutant described above (with new integers $n_k$ equal to $1$). In other words, each vertex in $V_k$ is split in $n_k$ copies, and we get the Bratteli diagram $B'_{\mathrm{BH}}$ illustrated in Figure~\ref{newodomutantboylehandelman}.
	
	\begin{figure}[ht!]
		\centering
		\includegraphics[width=1\linewidth]{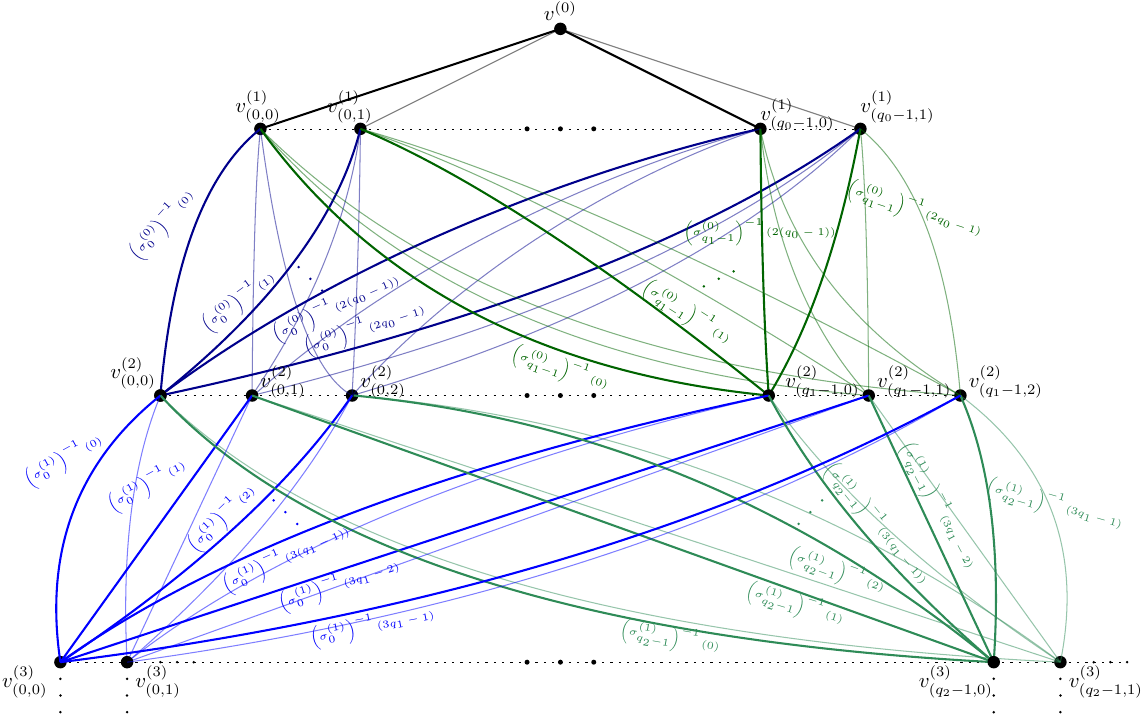}
		\caption{\footnotesize To get $B'_{\mathrm{BH}}=(V',E')$ from the Bratteli diagram $B_{\mathrm{BH}}$ of Boyle and Handelman (see Figure~\ref{odomutantboylehandelman}), we successively duplicate the vertices and the edges.\\
			In $B_{\mathrm{BH}}$, the integer $n_1$ is equal to $2$ (from each vertex in $V_1$ to each one in $V_{2}$, there are two edges), so each vertex $v^{(1)}_i$ (with $0\leq i<q_0$) is split in two new vertices $v^{(1)}_{(i,0)}$ and $v^{(1)}_{(i,1)}$, each one being associated to one of the two edges in $r^{-1}(v^{(2)}_j)$ (for every $0\leq j<q_1$). For every $0\leq m_1\leq n_1-1$, there is only one edge between $v^{(0)}$ to $v^{(1)}_{i,m_1}$ (this edge can be considered as a copy of the edge from $v^{(0)}$ to the former vertex $v^{(1)}_{i}$).\\
			The integer $n_2$ is equal to $3$ (from each vertex in $V_2$ to each one in $V_{3}$, there are three edges), so each vertex $v^{(2)}_j$ (with $0\leq j<q_1$) is split in three new vertices $v^{(2)}_{(j,0)}$, $v^{(2)}_{(j,1)}$ and $v^{(2)}_{(j,2)}$, each one being associated to one of the three edges in $r^{-1}(v^{(3)}_k)$ (for every $0\leq k<q_2$). For every $0\leq m_2\leq n_2-1$, we define the edges of range $v^{(2)}_{j,m_2}$ as copies of the edges of range the former vertex $v^{(2)}_{j}$. A thicker edge corresponds to one copy. We do not indicate the rank of the other edges (the thinner ones) for clarity.\\
			Then we apply the same algorithm to define the new vertices and edges in $E'_2, V'_3, E'_3, V'_4,\ldots$.}
		\label{newodomutantboylehandelman}
	\end{figure}
	
	An infinite path of $X_{B_{\mathrm{BH}}}$ can be uniquely written as
	$$\left ((v^{(k)}_{i_k},v^{(k+1)}_{i_{k+1}},\left (\sigma^{(k-1)}_{i_{k+1}}\right )^{-1}(n_ki_k+m_k))\right )_{k\geq 1}$$
	with $0\leq i_k<q_{k-1}$ and $0\leq m_k\leq n_k-1$ (we omit the first edge $(v^{(0)},v^{(1)}_{i_1},0)$). With the notations of Figure~\ref{newodomutantboylehandelman}, the map
	\begin{align*}
		&\left ((v^{(k)}_{i_k},v^{(k+1)}_{i_{k+1}},\left (\sigma^{(k-1)}_{i_{k+1}}\right )^{-1}(n_ki_k+m_k))\right )_{k\geq 1}\in X_{B_{\mathrm{BH}}}\\
		&\mapsto \left ((v^{(k)}_{(i_k,m_k)},v^{(k+1)}_{(i_{k+1},m_{k+1})},\left (\sigma^{(k-1)}_{i_{k+1}}\right )^{-1}(n_ki_k+m_k))\right )_{k\geq 1}\in X_{B'_{\mathrm{BH}}}
	\end{align*}
	is a conjugation between the Bratteli-Vershik systems $T_{B_{\mathrm{BH}}}$ and $T_{B'_{\mathrm{BH}}}$.
	
	\subsection{Comparisons between Boyle and Handelman's proof and our techniques.}\label{appendixComparison}
	
	Unlike Boyle and Handelman, we prove the case $\alpha<+\infty$ of Theorem~\ref{thB} with a cutting-and-stacking process where all new towers contain only one copy of each former tower. This is naturally the construction encoded by an odomutant endowed with the sequence $(\p(\ell))_{\ell\geq 1}$ of partitions in $\ell$-cylinders (see Figure~\ref{odomutant}). In order to get the case $\alpha=+\infty$, the main trick is to understand that a less restrictive cutting-and-stacking process, namely where every former tower may appear many times in the new ones, is encoded by an odomutant equipped with another sequence of partitions. Here the partitions are the ones associated to a description of this odomutant by multiple permutations, namely the partitions $\ptilde(\ell)$ (see Definition~\ref{defodostack}). A first way to understand why this is relevant is to notice that with these partitions, we cannot distinguish between towers of the same step, as if they were the copies of the same former tower which appear in a new one (see Figure~\ref{odomutant2}). Another remark is that Boyle and Handelman use the partitions in cylinders in the Cantor space $X_{B_{\mathrm{BH}}}$ on which their system $T_{B_{\mathrm{BH}}}$ is defined. Therefore, if we want to reformulate their proof in our formalism and with the odomutant conjugate to $T_{B_{\mathrm{BH}}}$, we have to consider the image of these partitions by the conjugation that we explicit above. It turns out that we get the partitions $\ptilde(\ell)$.\par
	To prove that the system $T_{B_{\mathrm{BH}}}$ is strongly orbit equivalent to the dyadic odometer $S$, Boyle and Handelman use the Giordano-Putnam-Skau theorem and the fact that the dimension group of $T_{B_{\mathrm{BH}}}$ is $\Z[1/2]$. In our proof of Theorem~\ref{thB}, the orbit equivalence is explicit and this enables us to directly show that the cocycles have at most one point of discontinuity. This also enables us to quantify the integrability of this orbit equivalence.

	\section{Equivalence between definitions of loose Bernoullicity in the zero-entropy case}\label{secappLB}
	
	To our knowledge, justifications for the equivalence between two definitions of loose Bernoullicity in the zero-entropy case (see Theorem~\ref{thEquivalenceLB}) is missing is the literature. Here we provide a proof. Let us first recall these definitions, that we already wrote in Section~\ref{PrelKak}.
	
	\begin{definition}
		Let $T\in\aut$ and $\p$ be a partition of $X$.
		\begin{itemize}
			\item $(T,\p)$ is loosely Bernoulli, and we write $T$ is \textbf{LB}, if for every $\varepsilon>0$, for every sufficiently large integer $N$ and for each $M>0$, there exists a collection $\mathcal{G}$ of "good" atoms in $\p_{-M}^{0}$ whose union has measure greater than or equal to $1-\varepsilon$, and so that for each pair $A,B$ of atoms in $\mathcal{G}$, the following holds: there is a probability measure $n_{A,B}$ on $\p^N\times\p^N$ satisfying
			\begin{enumerate}[label=(\Roman*)]
				\item\label{item1LB} $n_{A,B}(\{w\}\times \p^N)=\mu_A(\{[\mathcal{P}]_{1,N}(.)=w\})$ for every $w\in\p^N$;
				\item\label{item2LB} $n_{A,B}(\p^N\times\{w'\})=\mu_B(\{[\mathcal{P}]_{1,N}(.)=w'\})$ for every $w'\in\p^N$;
				\item\label{item3LB} $n_{A,B}(\{(w,w')\in\p^N\times\p^N\mid f_N(w,w')>\varepsilon\})<\varepsilon$.
			\end{enumerate}
			\item We say that $(T,\p)$ is \textbf{LB}\bm{$_0$} if for every $\varepsilon>0$ and for every sufficiently large integer $N$, there exists a collection $\mathcal{H}$ of "good" atoms in $\p_{1}^{N}$ whose union has measure greater than or equal to $1-\varepsilon$ and so that we have $f_N(w,w')\leq\varepsilon$ for every $w,w'\in [\p]_{1,N}(\mathcal{H})$.
		\end{itemize}
	\end{definition}
	
	\begin{theorem}\label{appthEquivalenceLB}
		Let $T\in\aut$ and $\p$ be a partition of $X$. If $\hmu(T,\p)=0$, then $(T,\p)$ is LB if and only if it is LB$_0$.
	\end{theorem}
	
	This theorem relies on the following key lemma which crucially uses the assumption on the entropy. Note that, when considering a set $\mathcal{Q}$ of subsets of $X$, for instance a set of atoms of a partition, $\mu(\mathcal{Q})$ will abusively denote the measure of $\bigcup_{A\in\mathcal{Q}}{A}$.
	
	\begin{lemma}\label{keylemmaEquivalenceLB}
		Let $T\in\aut$ and $\p$ be a partition of $X$, such that $\hmu(T,\p)=0$. Let $\alpha>0$ and $N$ be a positive integer. Then there exists an integer $M_0\geq 0$ such that the following holds for every $M\geq M_0$: there exists a collection $(\mathcal{Q}_{C})_{C\in\p_1^N}$ of disjoints subsets of $\p_{-M}^0$ such that
		\begin{itemize}
			\item for every $C\in\p_1^N$, for every $A\in\mathcal{Q}_C$, we have $\mu_A(C)\geq 1-\sqrt{\alpha}$;
			\item for every $C\in\p_1^N$, we have $\mu\left (\mathcal{Q}_C\right )\geq (1-2\sqrt{\alpha})\mu(C)$.
		\end{itemize}
	\end{lemma}
	
	\noindent Note that the second item implies $\mu\left (\bigcup_{C\in\p_1^N}\mathcal{Q}_C\right )\geq 1-2\sqrt{\alpha}$.
	
	\begin{proof}[Proof of Lemma~\ref{keylemmaEquivalenceLB}]
		By~\cite[Fact 2.3.12]{downarowiczEntropyDynamicalSystems2011}, the assumption $\hmu(T,\p)=0$ implies that $\p_1^N$ is $\p_{-\infty}^{0}$-mesurable, where
		$$\p_{-\infty}^{0}\coloneq\sigma\left (\p_{-M}^{0},M\geq 0\right),$$
		namely $\p_{-\infty}^{0}$ is the $\sigma$-algebra generated by the increasing sequence of algebras $\left (\sigma(\p_{-M}^{0})\right)_{M\geq 0}$. Then the following holds for every $C\in\p_{1}^N$: for every $\eta>0$, there exists $B_C\in\bigcup_{M\geq 0}{\sigma(\p_{-M}^{0})}$ such that $\mu(C\Delta B_C)\leq\eta$. Applied to $\eta=\alpha\mu(C)$, this fact provides an integer $M_0\geq 0$ such that every atom $C\in\p_{1}^N$ is closed to some $B_C\in\sigma(\p_{-M_0}^{0})$, namely $\mu(C\Delta B_C)\leq\alpha\mu(C)$. Let us fix an integer $M\geq M_0$, and notice that $B_C$ is also in $\sigma(\p_{-M}^0)$.\par
		For every $C\in\p_1^N$, let us set
		$$\mathcal{Q}_C\coloneq\left\{A\in\p_{-M}^0\mid A\subset B_C,\ \mu_A(C)>\min{(1-\sqrt{\alpha},1/2)}\right\}.$$
		Given two distinct atoms $C,C'\in\p_1^N$, the sets $\mathcal{Q}_C$ and $\mathcal{Q}_{C'}$ are disjoint, otherwise we would have an atom $A\in\p_{-M}^0$ lying in $\mathcal{Q}_C$ and $\mathcal{Q}_{C'}$ and such that the following occurs:
		$$\mu(A)\geq\mu(A\cap C)+\mu(A\cap C')=\left (\mu_A(C)+\mu_A(C')\right )\mu(A)>\mu(A),$$
		a contradiction.\par
		Given $C\in\p_1^N$, it remains to prove $\mu(\mathcal{Q}_C)\geq (1-2\sqrt{\alpha})\mu(C)$. Let us write
		$$\mathcal{Q}_C^-\coloneq\{A\in\p_{-M}^0\mid A\subset B_C\}\setminus\mathcal{Q}_C.$$
		On the one hand, we have
		\begin{align*}
			\mu(B_C\cap C)&=\sum_{A\in\mathcal{Q}_C}{\mu(A\cap C)}+\sum_{A\in\mathcal{Q}_C^-}{\mu(A\cap C)}\\
			&\leq\mu(\mathcal{Q}_C)+(1-\sqrt{\alpha})\mu(\mathcal{Q}_C^-)\\
			&=\mu(\mathcal{Q}_C)+(1-\sqrt{\alpha})(\mu(B_C)-\mu(\mathcal{Q}_C))\\
			&=(1-\sqrt{\alpha})\mu(B_C)+\sqrt{\alpha}\mu(\mathcal{Q}_C)\\
			&\leq (1-\sqrt{\alpha})(1+\alpha)\mu(C)+\sqrt{\alpha}\mu(\mathcal{Q}_C).
		\end{align*}
		where the last inequality comes from
		$$\mu(B_C)\leq \mu(B_C\Delta C)+\mu(C)\leq (1+\alpha)\mu(C)$$
		On the other hand, we have
		$$\mu(B_C\cap C)\geq\mu(C)-\mu(B_C\Delta C)\geq (1-\alpha)\mu(C).$$
		Combining all these inequalities, we get
		$$\mu(\mathcal{Q}_C)\geq \frac{1}{\sqrt{\alpha}}\left (1-\alpha-(1-\sqrt{\alpha})(1+\alpha)\right )\mu(C)=(1-\sqrt{\alpha})^2\mu(C)\geq (1-2\sqrt{\alpha})\mu(C),$$
		as wanted.
	\end{proof}
	
	\begin{proof}[Proof of Theorem~\ref{appthEquivalenceLB}]
		Assume that $(T,\p)$ is LB. Let us fix $\varepsilon\in ]0,1[$ and a sufficiently large integer $N$ as in the definition of LB. With $\alpha>0$ small enough so that
		$$(1-\sqrt{\alpha})(1-\sqrt{\alpha}-\varepsilon)\geq 1-2\varepsilon$$
		$$\text{and }1-2\sqrt{\alpha}\geq\varepsilon,$$
		we apply Lemma~\ref{keylemmaEquivalenceLB} to get $M$ and $(\mathcal{Q}_C)_{C\in\p_1^N}$ as described in the statement. By definition of LB associated to the quantities $\varepsilon$, $N$ and $M$, we get $\mathcal{G}\subset\p_{-M}^0$ covering at least $1-\varepsilon$ of the space, and a family $(n_{A,B})_{A,B\in\mathcal{G}}$ of probabilities on $\p^N\times\p^N$ satisfying items~\ref{item1LB},~\ref{item2LB} and~\ref{item3LB}. Let us define
		$$\mathcal{H}\coloneq\left\{C\in\p_1^N\mid\mathcal{G}\cap\mathcal{Q}_C\not=\emptyset\right\}.$$
		We first have
		\begin{align*}
			\mu(\mathcal{H})=\sum_{C\in\mathcal{H}}{\mu(C)}\geq\sum_{C\in\mathcal{H}}{\sum_{A\in\mathcal{Q}_C\cap\mathcal{G}}{\mu(C\cap A)}}&\geq (1-\sqrt{\alpha})\sum_{C\in\mathcal{H}}{\sum_{A\in\mathcal{Q}_C\cap\mathcal{G}}{\mu(A)}}\\
			&=(1-\sqrt{\alpha})\mu\left (\mathcal{G}\cap\bigcup_{C\in\p_1^N}{\mathcal{Q}_C}\right )\\
			&\geq (1-\sqrt{\alpha})(1-\sqrt{\alpha}-\varepsilon)\\
			&\geq 1-2\varepsilon.
		\end{align*}
		Secondly, let us consider $C,C'\in\mathcal{H}$ and let us prove that $w\coloneq[\p]_{1,N}(C)$ and $w'\coloneq[\p]_{1,N}(C')$ are $f_N$-close. By definition, we can pick $A\in \mathcal{G}\cap\mathcal{Q}_C$ and $B\in \mathcal{G}\cap\mathcal{Q}_{C'}$, and using items~\ref{item1LB} and~\ref{item2LB} we have
		$$n_{A,B}(\{w\}\times\p^N)\geq\mu_A(C)\geq 1-\sqrt{\alpha}$$
		$$\text{and }n_{A,B}(\p^N\times\{w'\})\geq\mu_A(C')\geq 1-\sqrt{\alpha}.$$
		This implies
		$$n_{A,B}(\{(w,w')\})\geq 1-2\sqrt{\alpha}\geq\varepsilon,$$
		so $f_N(w,w')\leq\varepsilon$ by item~\ref{item3LB}. We have proved that $(T,\p)$ satisfies LB$_0$ for $2\varepsilon$.\par
		Let us now assume that $(T,\p)$ is LB$_0$, we fix $\varepsilon>0$, a sufficiently large integer $N>0$ and an associated $\mathcal{H}\subset\p_1^N$ as in the definition of LB$_0$. With $\alpha>0$ small enough so that
		$$(1-\sqrt{\alpha})^2\geq 1-\varepsilon$$
		$$\text{and }(1-2\sqrt{\alpha})(1-\varepsilon)\geq 1-2\varepsilon,$$
		we apply Lemma~\ref{keylemmaEquivalenceLB} to get $M_0$ and for every $M\geq M_0$, an associated collection $(\mathcal{Q}_C)_{C\in\p_1^N}$ as described in the statement. Let us fix $M\geq M_0$ and let us consider
		$$\mathcal{G}\coloneq\bigcup_{C\in\mathcal{H}}{\mathcal{Q}_C}$$
		and for every $A,B\in\mathcal{G}$, the probability $n_{A,B}$ on $\p^N\times\p^N$ defined by
		$$n_{A,B}(\{(w,w')\})=\mu_A(\{[\mathcal{P}]_{1,N}(.)=w\})\mu_B(\{[\mathcal{P}]_{1,N}(.)=w'\}),$$
		they automatically satisfy items~\ref{item1LB} and~\ref{item2LB}. Given $C,C'\in\mathcal{H}$, $A\in\mathcal{Q}_C$ and $B\in\mathcal{Q}_{C'}$, and $w\coloneq[\p]_{1,N}(C)$ and $w'\coloneq[\p]_{1,N}(C')$, we have
		$$n_{A,B}(\{(w,w')\})\geq\mu_A(C)\mu_B(C')\geq (1-\sqrt{\alpha})^2\geq 1-\varepsilon,$$
		and since $f_N(w,w')\leq\varepsilon$, we get item~\ref{item3LB}. Finally, we have
		$$\mu(\mathcal{G})=\sum_{C\in\mathcal{H}}{\mu(\mathcal{Q}_{C})}\geq (1-2\sqrt{\alpha})\sum_{C\in\mathcal{H}}{\mu(C)}=(1-2\sqrt{\alpha})\mu(\mathcal{H})\geq (1-2\sqrt{\alpha})(1-\varepsilon)\geq 1-2\varepsilon.$$
		We have proved that $(T,\mathcal{P})$ satisfies LB for $2\varepsilon$, $N$ large enough and $M\geq M_0$. By~\cite[Corollary~2]{feldmanNewKautomorphismsProblem1976}, we can replace "for each $M>0$" by "for every sufficiently large $M>0$" in the definition of LB, so we are done.
	\end{proof}
	
	\bibliographystyle{alphaurl}
	\bibliography{biblio}
	
	{\bigskip
		\footnotesize
		
		\noindent C.~Correia, \textsc{Université Paris Cité, Institut de Mathématiques de Jussieu-Paris Rive Gauche, 75013 Paris, France}\par\nopagebreak\noindent
		\textit{E-mail address: }\texttt{corentin.correia@imj-prg.fr}}
	
\end{document}